\documentclass[12pt, leqno]{article}
\usepackage{amsmath,amssymb}
\usepackage{cite}
\usepackage[utf8]{inputenc}
\usepackage{mathrsfs}
\usepackage{euler}
\usepackage{marvosym}
\usepackage{enumitem}
\usepackage[amsmath,thmmarks]{ntheorem}
\usepackage{xargs}                      
\usepackage[pdftex,dvipsnames]{xcolor}  
\usepackage{chngcntr}
\usepackage{xifthen}
\usepackage{algorithm2e}
\usepackage{graphicx}

\numberwithin{equation}{section}
\counterwithin*{equation}{section}
\counterwithin*{equation}{subsection}
\counterwithin*{figure}{section}
\counterwithin*{figure}{subsection}

\newtheorem{theorem}{Theorem}

\numberwithin{section}{part}
\newtheorem{corollary}{Corollary}
\newtheorem{lemma}{Lemma}
\newtheorem*{remark}{Remark}
\newtheorem*{definition}{Definition}
\newtheorem*{remarks}{Remarks}
\newcommand{\LA}[1]{\refstepcounter{equation}\text{(\theequation)}\label{#1}}
\newenvironment{proof}[1][Proof]{\noindent\textbf{#1.} }{\ \rule{0.5em}{0.5em}}

\newcommand{\G}{\Gamma}

\newcommand{\vg}{\vec{\Gamma}}
\newcommand{\vvg}{\vec{\vec{\Gamma}}}
\newcommand{\tg}{\tilde{\Gamma}}

\newcommand{\blobdef}[1][]{%
  \ifthenelse{\isempty{#1}}%
  {(\G(x, M,\tau))_{x\in E, M>0, \tau\in (0,\tau_{\max}]}}
  {(\G_{#1}(x, M,\tau))_{x\in E, M>0, \tau\in (0,\tau_{\max}]}}
  }

\newcommand{\R}{\mathbb{R}}

\newcommand{\Boxx}{\textbf{Box} }
\renewcommand{\P}{\mathcal{P}}
\newcommand{\A}{\mathcal{A}}
\newcommand{\M}{\mathcal{M}}
\renewcommand{\qed}{\hfill$\blacksquare$}

\begin{document}
\title{Efficient Algorithms for Approximate Smooth Selection}
\date{\today }
\author{Charles Fefferman\thanks{Princeton University. Supported by AFOSR FA9550-18-1-069, NSF DMS-1700180, BSF 2014055. Supported by the US-Israel BSF.}, Bernat Guill\'{e}n Pegueroles\thanks{Princeton University. Supported by Fulbright-Telef\'{o}nica, AFOSR FA9550-18-1-069, NSF. Contact: bernatp@princeton.edu}}

\maketitle
\vspace{0.05\textheight}
\emph{In fond memory of Elias Stein.}
\part{Introduction and Notation}

\section{Introduction\label{sec:intro}}

This paper continues a study of extension and approximation of functions, going back to H. Whitney \cite{whitneyAnalyticExtensionsDifferentiable1934, whitneyDifferentiableFunctionsDefined1934, whitneyFunctionsDifferentiableBoundaries1934}, with important contributions from E. Bierstone, Y. Brudnyi, C. Fefferman, G. Glaeser, A. Israel, B. Klartag, E. Le Gruyer, G. Luli, P. Milman, W. Paw\l ucki, P. Shvartsman and N. Zobin.

See \cite{bierstone2007mathcal,bierstoneDifferentiableFunctionsDefined2003,bierstoneHigherorderTangentsFefferman2006,brudnyi1985linear,brudnyiGeneralizationsWhitneyExtension1994,brudnyiWhitneyExtensionProblem2001,brudnyiWhitneyProblemExistence1997,feffermanExtensionOmegaSmooth2009,feffermanFinitenessPrinciplesSmooth2016,feffermanFittingSmoothFunction2009,feffermanFittingSmoothFunction2009a,feffermanGeneralizedSharpWhitney2005,feffermanInterpolationExtrapolationSmooth2005,feffermanSharpFormWhitney2005,feffermanSobolevExtensionLinear2014,feffermanWhitneyExtensionProblem2006,glaeserEtudeQuelquesAlgebres1958,le2009minimal,luliCmExtensionBoundeddepth2010,shvartsmanWhitneytypeExtensionTheorems2017,zobinExtensionSmoothFunctions1999,zobinWhitneyProblemExtendability1998}.

The motivation of these problems is to reconstruct functions from data. In particular, the work of \cite{feffermanFittingSmoothFunction2009, feffermanFittingSmoothFunction2009a} shows how to interpolate a function given precise data points. However, in real applications the data is measured with error. A ``finiteness" theorem underlies the results of \cite{feffermanFittingSmoothFunction2009,feffermanFittingSmoothFunction2009a} for interpolation of perfectly specified data. The paper  \cite{feffermanFinitenessPrinciplesSmooth2016} proves a corresponding finiteness theorem for interpolation of data measured with error. However, the proofs of the main results of \cite{feffermanFinitenessPrinciplesSmooth2016} are nonconstructive. The interpolation of data specified with error remains a challenging problem.

Fix positive integers $m, n, D$. We work in $\mathcal{C}^{m}(\R^n, \R^D)$, the space of all $F:\R^n \to \R^D$ with all partial derivatives of order up to $m$ continuous and bounded on $\R^n$. We use the norm
\begin{itemize}
	\item[\LA{eq:intro-1}] $\|F\| = \sup_{x\in\R^n} \max_{|\alpha|\leq m} |\partial^\alpha F(x)|$
\end{itemize}
(or an equivalent one) which is finite. We write $c, C, C'$, etc. to denote constants depending only on $m,n,D$. These symbols may denote different constants in different occurrences.

Let $E \subset \R^n$ be a finite set with $N$ elements. For each $x\in E$, suppose we are given a \textbf{bounded} convex set $K(x)\subset \R^D$.  
A {\bf $\mathcal{C}^m$ selection} of $\vec{K}:=(K(x))_{x\in E}$ is a function $F\in \mathcal{C}^{m}(\R^n, \R^D)$ such that $F(x)\in K(x)$ for all $x\in E$. We want to compute a $\mathcal{C}^m$ selection $F$ whose norm $\|F\|$ is as small as possible up to a factor of $C$. Such problems arise naturally when we try to fit smooth functions to data. A simple example with $n=D=1$ is shown in Figure \ref{fig:1}; the sets $K(x)\subset \R^1$ are "error bars".

\begin{figure}[h!]
	\centering
	\includegraphics[width=\textwidth]{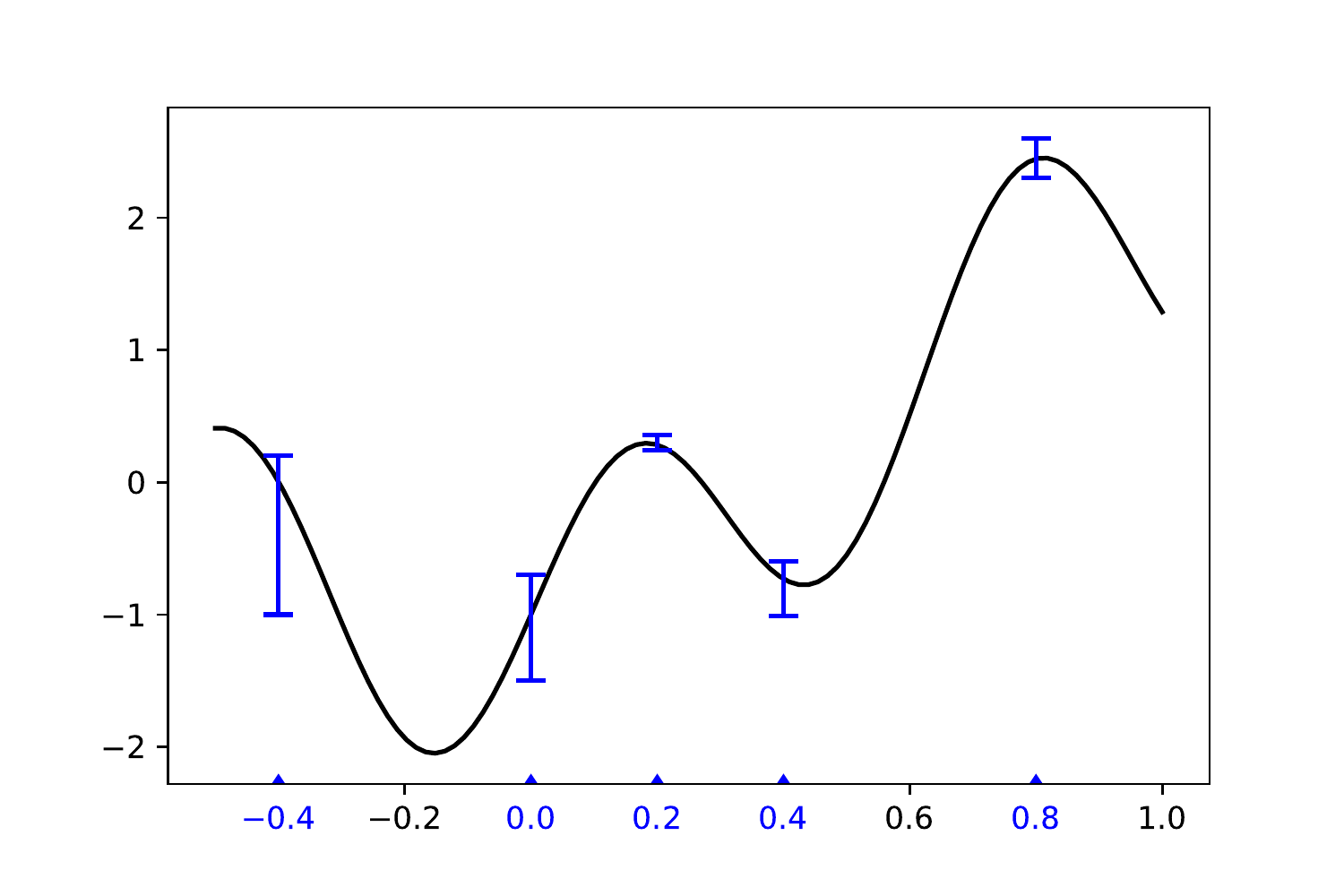}
	\renewcommand{\thefigure}{\Roman{section}.\arabic{figure}}
	\label{fig:1}
	\caption{A simple $\mathcal{C}^m$ selection problem. The set $E$ consists of the dots on the $x-$axis marked by a triangle. Above each $x\in E$ is an interval $K(x)$. The function $F$ shown here satisfies $F(x) \in K(x)$ for all $x\in E$.}
\end{figure}
If each $K(x)$ consists of a single point, then our $\mathcal{C}^m$ selection problem reduces to the problem of {\bf interpolation}: We are given a function $f:E\to \R^D$, and we want to compute an $F\in \mathcal{C}^m(\R^n, \R^D)$ such that $F=f$ on $E$, with $\|F\|$ as small as possible up to a factor $C$. For interpolation, we can take $D=1$ without loss of generality.

We want to solve the above problems by algorithms, to be implemented on an (idealized) computer with standard von Neumann architecture, able to deal with real numbers to infinite precision (no roundoff errors). We hope our algorithms will be {\bf efficient}, i.e., they require few computer operations. (An ``operation'' consists e.g. of fetching a number from RAM or multiplying two numbers.)

For interpolation problems, the following algorithm was presented in \cite{feffermanFittingSmoothFunction2009, feffermanFittingSmoothFunction2009a}:

\begin{algorithm}[h!]
	\TitleOfAlgo{Interpolation Algorithm}
	\KwData{$N-$element set $E\in \R^n$, $f:E\to \R^D$}
	\KwResult{\texttt{Query} function, a subroutine that receives $x\in \R^n$ and returning the $m-$th degree Taylor Polynomial at $x$ of a function $F\in\mathcal{C}^m(\R^n, \R^D)$. The function $F$ is uniquely determined by the data $m,n,D, E, f$. In particular, $F$ does not depend on the points $x$ for which we call the query subroutine. Furthermore, $F$ is guaranteed to satisfy $F=f$ at $E$, with $\|F\|$ as small as possible up to a factor $C$.}
	\label{alg:interp}
	\caption{Interpolation algorithm definition}
\end{algorithm}

\begin{remark}
	We can think of \texttt{Query} as an efficient, computer-friendly encoding of a fixed function $F\in\mathcal{C}^m(\R^n, \R^D)$ that gives us all the information we can have of $F$ at point $x$: Its $m-$th degree Taylor Polynomial.
\end{remark}
Moreover, Algorithm \ref{alg:interp} requires at most $CN\log N$ operation, and each call to the \texttt{Query} function requires at most $C\log N$ operations.	

Very likely the above $N\log N$ and $\log N$ are the best possible.

We hope to find an equally efficient algorithm for $\mathcal{C}^m$ selection problems. Already in simple one-dimensional cases like the problem depicted in Figure \ref{fig:1}, we don't know how to do that.

To make the problem easier, we allow ourselves to enlarge the "targets" $K(x)$ slightly. Given $\tau\in(0, 1)$ and $K\subset \R^D$ bounded and convex, we define
\begin{itemize}
	\item[\LA{eq:enlarg}] $(1+\tau)\blacklozenge K := \{v + \frac{\tau}{2} v' - \frac{\tau}{2} v'': v, v', v'' \in K\}$
\end{itemize}
If $\tau$ is small, then $(1+\tau)\blacklozenge K$ is a slightly enlarged version of $K$ whenever $K$ is bounded.

We would like to find $F\in \mathcal{C}^m(\R^n, \R^D)$ such that $F(x)\in K(x)$ for all $x\in E$. Instead, we will find an $F$ that satisfies $F(x) \in (1+\tau)\blacklozenge K(x)$ for a given small $\tau$. As $\tau\to 0$, the work of our algorithm increases rapidly.

In its simplest form, the main result of this paper is the $\mathcal{C}^m$ Selection Algorithm (Algorithm \ref{alg:cmselect}). This algorithm receives as input real numbers $M>0$ and $\tau \in (0, 1)$, a finite set $E \subset \R^n$, and a convex polytope $K(x) \subset \R^D$ for each $x \in E$. We suppose that each $K(x)$ is specified by at most $C$ linear constraints. 

Given the above input, we produce one of the following outcomes. 

\begin{itemize}
	\item {\bf Success:} We return a function $f:E\to\R^D$, with $f(x) \in (1+\tau)\blacklozenge K(x)$ for each $x \in E$. Moreover, we guarantee that there exists $F \in \mathcal{C}^m (\R^n, \R^D)$ with norm $\|F\| \leq CM$ such that $F = f$ on $E$.
	\item {\bf No go:} We guarantee that there exists no $F \in \mathcal{C}^m(\R^n, \R^D)$ with norm at most $M$, such that $F(x) \in K(x)$ for all $x \in E$.
\end{itemize}

\begin{algorithm}[!h]
	\TitleOfAlgo{$\mathcal{C}^m$ Selection Algorithm}
	\KwData{Real numbers $M>0$, $\tau\in(0,1)$, an $N-$element set $E\in \R^n$ and a convex polytope $K(x)\subset\R^D$ for each $x\in E$}
	\tcc{We suppose that each $K(x)$ is specified by at most $C$ linear constraints}
	\KwResult{One of the following two outcomes:
		
		{\bf Success}: We return a function $f:E\to\R^D$, with $f(x)\in (1+\tau)\blacklozenge K(x)$ for each $x\in E$. Moreover, we guarantee that there exists $F\in \mathcal{C}^m(\R^n, \R^D)$ with norm $\|F\|\leq CM$ such that $F = f$ on $E$.
	
		{\bf No go}: We guarantee that there exists no $F\in \mathcal{C}^m(\R^n, \R^D)$ with norm at most $M$, such that $F(x) \in K(x)$ for all $x\in E$.
	}
	\label{alg:cmselect}
	\caption{$\mathcal{C}^m$ selection algorithm description.}
\end{algorithm}
In the event of success, we can find the function $F$ by applying to $f$ the Interpolation Algorithm (Algorithm \ref{alg:interp}).

The $\mathcal{C}^m$ Selection Algorithm requires at most $C(\tau)N\log N$ operations, where $C(\tau)$ depends only on $\tau, m, n, D$.

We needn't require the convex sets $K(x)$ to be polytopes. Instead, we suppose that an {\bf Oracle} responds to a query $\tau \in (0,1)$ by producing a family of convex polytopes $K_{\tau}(x)$ ($x\in E$), each defined by at most $C(\tau)$ linear constraints, such that $K(x) \subset K_{\tau}(x) \subset (1+\tau)\blacklozenge K(x)$ for each $x \in E$.

To produce all the $K_{\tau}(x)$ ($x\in E$) for a given $\tau$, the oracle charges us $C(\tau)N$ operations of work. In particular, if each $K(x)$ is already a polytope defined by at most $C$ constraints, then the oracle can simply return $K_{\tau}(x) = K(x)$ for each $x \in E$.

We sketch a few of the ideas behind our algorithm. We oversimplify for ease of understanding. See the sections below for a correct discussion.

The first step is to place the problem in a wider context. Instead of merely examining the values of $F$ at points $x \in E$, we consider the $(m-1)$-rst degree Taylor polynomial of $F$ at $x$, which we denote by $J_x(F)$. We write $\P$ to denote the vector space of all such Taylor polynomials. Instead of families of convex sets $K(x) \subset \R^D$, we consider families of convex sets $\Gamma(x, M, \tau) \subset \P$ ($x\in E, M>0, \tau\in (0,1)$). We want to find $F\in \mathcal{C}^m(\R^n, \R^D)$ with norm at most $CM$, such that $J_x(F) \in \Gamma(x, M, \tau)$ for all $x \in E$.

Under suitable assumptions on the $\Gamma(x, M, \tau)$, we provide the following algorithm.

\begin{algorithm}
	\TitleOfAlgo{Generalized Selection Algorithm}
	\KwData{Real numbers $M>0$, $\tau\in(0,1)$. A suitable family of convex sets $\Gamma(x,M,\tau)$.}
	\KwResult{One of the following two outcomes:
		
		{\bf Success}: We exhibit a polynomial $P^x \in \Gamma(x, CM, C\tau)$ for each $x\in E$. Moreover, we guarantee that there exists $F\in \mathcal{C}^m(\R^n, \R^D)$ with norm $\|F\|\leq CM$ such that $J_x(F) = P^x$ for all $x \in E$.
		
		{\bf No go}: We guarantee that there exists no $F\in \mathcal{C}^m(\R^n, \R^D)$ with norm at most $M$, such that $J_x(F) \in \Gamma(x, M, \tau)$ for all $x\in E$.}
	\label{alg:select}
	\caption{Generalized $\mathcal{C}^m$ selection algorithm description.}
\end{algorithm}

The algorithm requires at most $C(\tau)N\log N$ operations. Our previous $\mathcal{C}^m$ selection algorithm is a special case of the Generalized selection algorithm.

Once we are dealing with $\Gamma$'s, we can take $D=1$ without loss of generality, i.e., we may deal with scalar valued functions $F$. From now on, we suppose $D=1$, and we write $\mathcal{C}^m(\R^n)$ in place of $\mathcal{C}^m(\R^n, \R^D)$.

To produce the Generalized Selection Algorithm we adapt ideas from the proof of the ``finiteness theorem'' in \cite{feffermanSharpFormWhitney2005}. The key ingredients are:
\begin{itemize}
	\item Refinements of $\Gamma$'s.
	\item Local Selection Problems, and
	\item Labels.
\end{itemize}

We provide a brief description of each of these ingredients, then indicate how they are used to produce the Generalized Selection Algorithm.

We begin with refinement of $\Gamma$'s. 

Suppose we are given a collection of convex sets $\Gamma(x, M, \tau) \subset \P$ ($x\in E, M>0, \tau\in(0,1)$). Let $M$ and $\tau$ be given. We want to find $F\in \mathcal{C}^m(\R^n)$ such that 

\begin{itemize}
	\item[\LA{eq:conds}] $\|F\|\leq M$ and $J_x(F) \in \Gamma(x, M, \tau)$ for all $x \in E$.
\end{itemize}

We can define a convex subset $\tg(x, M, \tau) \subset \G(x, M, \tau)$ for each $x\in E$ such that \eqref{eq:conds} implies the seemingly stronger condition

\begin{itemize}
	\item[\LA{eq:conds-ref}] $\|F\|\leq M$ and $J_x(F) \in \tg(x, M, \tau)$ for all $x \in E$.
\end{itemize}

That's because any $F \in \mathcal{C}^m(\R^n)$ with norm at most $M$ satisfies $|\partial^{\alpha}(J_x(F)-J_y(F))(x)|\leq M\|x-y\|^{m-|\alpha|}$ ($|\alpha|\leq m-1$) by Taylor's theorem. Consequently, if $F$ satisfies \eqref{eq:conds} and $J_x(F) = P$, then

\begin{itemize}
	\item[\LA{eq:conds-tayl}] For every $y \in E$ there exists $P' \in \Gamma(y, M, \tau)$ such that $|\partial^{\alpha}(P-P')(x)| \leq M \|x-y\|^{m-|\alpha|}$ ($|\alpha|\leq m-1$).
\end{itemize}

(We can just take $P' = J_y(F)$.)

Thus \eqref{eq:conds} implies \eqref{eq:conds-ref} if we take $\tg(x, M, \tau)$ to consist of all $P \in \G(x, M, \tau)$ satisfying \eqref{eq:conds-tayl}.

In fact, we need a different definition of $\tg$, because the $\tg$ defined by \eqref{eq:conds-tayl} is too expensive to compute. We proceed as in \cite{feffermanFittingSmoothFunction2009a}, using the Well-Separated Pairs Decomposition \cite{callahanDecompositionMultidimensionalPoint1995} from computer science. 

The {\bf first refinement} of the collection of convex sets \\$\G = (\G(x, M, \tau))_{x\in E, M>0, \tau\in(0,1)}$ is defined to be $\tg = (\tg(x, M, \tau))_{x \in E, M>0, \tau\in(0,1)}$. Proceeding by induction on $l \geq 0$, we then define the $l$-th refinement $\G_l = (\G_l(x, M, \tau))_{x\in E, M>0, \tau\in(0,1)}$ by setting $\G_0 = \G$, $\G_{l+1} = $ first refinement of $\G_l$.

We will consider the $l$-th refinement $\G_l$ for $l = 0,\dots,l_{\ast}$, where $l_{\ast}$ is a large enough integer constant determined by $m, n$.

The main properties of $\G_l$ are as follows:

\begin{itemize}
	\item Any $F \in \mathcal{C}^m(\R^n)$ that satisfies \eqref{eq:conds} also satisfies
	\begin{itemize}
		\item[\LA{eq:conds-ref-all}] $J_x(F) \in \G_l(x, M, \tau)$ for all $x\in E$ and $l = 0, \dots, l_{\ast}$
	\end{itemize}
	\item Given $P\in \G_l(x, M, \tau)$ and $y \in E$, there exists
	\begin{itemize}
		\item[\LA{eq:conds-exist-near}] $P' \in \G_{l-1}(y, M, \tau)$ such that $|\partial^{\alpha}(P-P')(x)| \leq M\|x-y\|^{m-|\alpha|}$ for $|\alpha|\leq m-1$.
	\end{itemize}
	\item For a given $(M,\tau)$, the set $\G_l(x, M, \tau)$ may be empty for some $l$, even if all the $\G(x, M, \tau)$ are nonempty. In this case, no $F\in\mathcal{C}^m(\R^n)$ can satisfy \eqref{eq:conds}; that's immediate from \eqref{eq:conds-ref-all}.
\end{itemize}
This concludes our introductory remarks about refinements.

We next discuss {\bf Local Selection Problems} and {\bf Labels}.

Let $\G = (\G(x, M, \tau))_{x\in E, M>0, \tau\in(0,1)}$ as above. Fix $M_0 > 0$, $\tau_0 \in (0,1)$. Suppose we are given a cube $Q_0 \subset \R^n$, a point $x_0 \in E\cap Q_0$, and a polynomial $P_0\in \G(x_0, M_0, \tau_0)$.

The {\bf Local Selection Problem}, denoted $LSP(Q_0, x_0, P_0)$, is to find an $F\in\mathcal{C}^m(Q_0)$ such that
\begin{itemize}
	\item $|\partial^\alpha F| \leq CM_0$ on $Q_0$ for $|\alpha| = m$
	\item $J_{x_0}(F) = P_0$, and
	\item $J_x(F) \in \G(x, CM, C\tau_0)$ for all $x \in E\cap Q_0$.
\end{itemize}

To measure the difficulty of a local selection problem $LSP(Q_0, x_0, P_0)$, we will attach {\bf labels} to it. A ``label'' is a subset $\A$ of the set $\M$ of all multiindices $\alpha = (\alpha_1, \dots, \alpha_n)$ of order $|\alpha| = \alpha_1 + \cdots + \alpha_n \leq m-1$. To decide whether we can attach a given label $\A$ to a problem $LSP(Q_0, x_0 P_0)$ we examine the geometry of the convex set $\G_l(x_0, M_0, \tau_0)$, where $l = l(\A)$ is an integer constant determined by $\A$. Roughly speaking, we attach the label $\A$ to the problem $LSP(Q_0, x_0, P_0)$ if the following condition holds, where $\delta_{Q_0}$ denotes the sidelength of $Q_0$.
\begin{itemize}
	\item[\LA{eq:LSP-cond}] For every $(\xi_\alpha)_{\alpha \in \A}$, with each $\xi_\alpha$ a real number satisfying $|\xi_\alpha|\leq M_0\delta_{Q_0}^{m-|\alpha|}$, there exists $P \in \G_{l(\A)}(x_0, M_0, \tau_0)$ such that $\partial^{\alpha}(P-P_0)(x_0) = \xi_{\alpha}$ for all $\alpha \in \A$.
\end{itemize}
We allow the case $\A = \emptyset$; in that case \eqref{eq:LSP-cond} asserts simply that $P_0 \in \G_{l(\emptyset)}(x_0, M_0, \tau_0)$. A given $LSP(Q_0, x_0, P_0)$ may admit more than on label $\A$.

We impose a total order relation $<$ on labels $\A$. If $\A<\mathcal{B}$ then, roughly speaking, a typical problem $LSP(Q_0, x_0, P_0)$ with label $\A$ is easier than a typical problem $LSP(Q_0, x_0, P_0)$ with label $\mathcal{B}$. If $\mathcal{B} \subset \A$ then $\A< \mathcal{B}$. In particular, the empty set $\emptyset$ is the maximal label with respect to $<$, and the set $\M$ of all multiindices of order at most $(m-1)$ is the minimal label. So $\M$ labels the easiest local selection problems, and $\emptyset$ labels the hardest problems. 

This completes our (oversimplified) introductory explanation of labels.

To make use of refinements, local selection problems and labels, we establish the following result for each label $\A$.

\begin{lemma}[Main Lemma for $\A$ (simplified)]
	Let $\G = (\G(x, M, \tau))_{x\in E, M>0, \tau\in(0,1)}$ be given. Fix $M_0 > 0$ and $\tau_0 \in (0,1)$. Then any local selection problem $LSP(Q_0, x_0, M_0)$ that carries the label $\A$ has a solution $F$. Moreover, such an $F$ can be computed by an efficient algorithm.
\end{lemma}

We prove the above Main Lemma by induction on $\A$, with respect to the order $<$. In the base case $\A = \M$, we can simply take $F = P_0$. This $F$ solves the local selection problem $LSP(Q_0, x_0, P_0)$ because in the base case $\A = \M$, the $\Gamma(x_0, M_0, \tau_0)$ are big enough.

For the induction step, we fix a label $\A\neq \M$, and make the inductive assumption
\begin{itemize}
	\item[\LA{eq:ind-ass}] The Main Lemma for $\A'$ holds for all labels $\A' < \A$.
\end{itemize}
Under this assumption, we then prove the Main Lemma for $\A$. To do so we must solve any given $LSP(Q_0, x_0, P_0)$ that carries the label $\A$. We make a Calder\'{o}n-Zygmund decomposition of $Q_0$ into finitely many subcubes $Q_\nu$. For each $Q_\nu$ we pick a base point $x_\nu \in E$ that lies in or near $Q_\nu$ (our Calder\'{o}n-Zygmund stopping rule guarantees that such an $x_\nu$ exists). If $E\cap Q_{\nu}$ is non-empty, we take $x_\nu \in E\cap Q_{\nu}$.

Because $LSP(Q_0, x_0, P_0)$ carries the label $\A$, we know that $P_0 \in \G_{l(\A)} (x_0, M_0, \tau_0)$. Using the basic property \eqref{eq:conds-exist-near} of the $\G_l$, we find a polynomial $P_\nu \in \G_{l(\A)-1}(x_\nu, M_0, \tau_0)$ for each $\nu$, such that $|\partial^{\alpha}(P_\nu - P_0)(x_0)| \leq M_0\|x_\nu - x_0\|^{m-|\alpha|}$ for $|\alpha|\leq m-1$.

Fix $\nu$, and suppose $E \cap Q_\nu \neq \emptyset$. We then pose the local selection problem $LSP(Q_\nu, x_\nu, P_\nu)$. Our Calder\'{o}n-Zygmund stopping rule guarantees that this problem is either trivial (because $E\cap Q_\nu$ contains only one point), or else carries a label $\A_\nu' < \A$. Consequently, our induction hypothesis \eqref{eq:ind-ass} lets us compute a solution $F_\nu$ to $LSP(Q_\nu, x_\nu, P_\nu)$. This holds if $E \cap Q_\nu \neq \emptyset$. If $E \cap Q_\nu = \emptyset$, then we just set $F_\nu = P_\nu$.

Patching together the above $F_\nu$ by a partition of unity adapted to the Calder\'{o}n-Zygmund decomposition $\{Q_\nu\}$, we obtain a solution $F$ to the given local selection problem $LSP(Q_0, x_0, P_0)$. This completes our induction on $\A$, and thus proves the Main Lemma.

Finally, we apply the above discussion to produce the Generalized $\mathcal{C}^m$ Selection Algorithm. We suppose we are given $\G = (\G(x, M, \tau))_{x\in E, M>0, \tau\in(0,1)}$, together with real numbers $M_0 >0$, $\tau_0 \in (0,1)$. Let $\G_l = (\G_l(x, M, \tau))_{x \in E, M>0, \tau \in (0,1)}$ be the $l$-th refinement of $\G$. We compute the $\G_l(x, M_0, \tau_0)$ for all $x \in E$ and all $l = 0,\dots,l_{\ast}$. If any of these $\G_l (x, M_0, \tau_0)$ are empty, then we produce the outcome {\bf No go} of Algorithm \ref{alg:select}.	Thanks to \eqref{eq:conds-ref-all}, we know that no $F \in \mathcal{C}^m(\R^n)$ with norm at most $M$ can satisfy $J_x(F) \in \G(x, M_0, \tau_0)$ for all $x \in E$.

On the other hand, suppose $\G_l(x, M_0, \tau_0)$ is non-empty for each $x \in E$. Let $Q_0$ be a cube of sidelength $1$ containing a point $x_0 \in E$. Then we can find a polynomial $P_0 \in \G_l(x_0, M_0, \tau_0)$ with $l = l(\emptyset)$. The local selection problem $LSP(Q_0, x_0, P_0)$ carries the label $\emptyset$, thanks to the remark immediately after \eqref{eq:LSP-cond}. The Main Lemma for the label $\emptyset$ allows us to compute a function $F_{Q_0}\in \mathcal{C}^m(Q_0)$ with $\mathcal{C}^m$ norm at most $CM_0$, such that $J_x(F_{Q_0}) \in \G(x, CM_0, C\tau_0)$ for all $x \in E\cap Q_0$. 

Covering $E$ by cubes $Q_0$ of unit length, and patching together the above $F_{Q_0}$ using a partition of unity, we obtain a function $F \in \mathcal{C}^m(\R^n)$ with norm at most $CM$, such that $J_x(F) \in \G(x, CM_0, C\tau_0)$ for all $x \in E$.

Thus, we have produced the outcome {\bf Success} for the Generalized $C^m$ Selection Algorithm. This concludes our sketch of that algorithm.

So far, we've omitted all mention of the assumptions we have to impose on our inputs $\G(x, M, \tau)$. One of those assumptions is that
\begin{itemize}
	\item[\LA{eq:blob-intro}] $(1+\tau)\blacklozenge \G(x, M, \tau) \subset \G(x, M', \tau')$ for $M' \geq CM$, $\tau' \geq C\tau$.
\end{itemize}
This allows us to ``simplify'' many convex sets $G \subset \P$ that arise in executing the Generalized $\mathcal{C}^m$ Selection Algorithm (\ref{alg:select}). More precisely, without harm, we may replace $G$ by a convex polytope $G_\tau$ defined by at most $C(\tau)$ linear constraints, such that $G \subset G_\tau \subset (1+\tau)\blacklozenge G$.

This prevents the complexity of the relevant convex polytopes from growing uncontrollably as we execute Algorithm \ref{alg:select}. 

We close our introduction by again warning the reader that we have oversimplified matters. The sections that follow give the correct results. Therefore, even the basic notation and definitions are to be taken from subsequent sections, not from this introduction. 

We are grateful to the participants of several workshops on Whitney Problems for valuable comments. We thank the National Science Foundation, the Air Force Office of Scientific Research, the US-Israel Binational Science Foundation, the Fulbright Commission (Spain) and Telefonica for generous financial support. We also thank Kevin Luli for his remarks regarding the application of these algorithms to the interpolation of non-negative functions. 

\section{Notation and Preliminaries\label{notation-and-preliminaries}}

Fix $m$, $n\geq 1$. We will work with cubes in $\mathbb{R}^{n}$; all our
cubes have sides parallel to the coordinate axes. If $Q$ is a cube, then $%
\delta _{Q}$ denotes the sidelength of $Q$. For real numbers $A>0$, $AQ$
denotes the cube whose center is that of $Q$, and whose sidelength is $%
A\delta _{Q}$. Note that, for general convex sets $K$ we define $AK = \{Av: v\in K\}$. It will always be clear in context which of these two conventions are in effect. 

A \underline{dyadic} cube is a cube of the form $I_{1}\times I_{2}\times
\cdots \times I_{n}\subset \mathbb{R}^{n}$, where each $I_{\nu }$ has the
form $[2^{k}\cdot i_{\nu },2^{k}\cdot \left( i_{\nu }+1\right) )$ for
integers $i_{1},\cdots ,i_{n}$, $k$. Each dyadic cube $Q$ is contained in
one and only one dyadic cube with sidelength $2\delta _{Q}$; that cube is
denoted by $Q^{+}$.

We write $\mathcal{P}$ to denote the vector space of all real-valued
polynomials of degree at most $\left( m-1\right) $ on $\mathbb{R}^{n}$. If $%
x\in \mathbb{R}^{n}$ and $F$ is a real-valued $C^{m-1}$ function on a
neighborhood of $x$, then $J_{x}\left( F\right) $ (the \textquotedblleft
jet" of $F$ at $x$) denotes the $\left( m-1\right) ^{rst}$ order Taylor
polynomial of $F$ at $x$, i.e.,

\begin{equation*}
J_{x}\left( F\right) \left( y\right) =\sum_{\left\vert \alpha \right\vert
	\leq m-1}\frac{1}{\alpha !}\partial ^{\alpha }F\left( x\right) \cdot \left(
y-x\right) ^{\alpha }.
\end{equation*}%
Thus, $J_{x}\left( F\right) \in \mathcal{P}$. Note that for all convex sets $K \in \P$, the convention of $AK = \{Av: v\in K\}$ will apply.

For each $x\in \mathbb{R}^{n}$, there is a natural multiplication $\odot
_{x} $ on $\mathcal{P}$ (\textquotedblleft multiplication of jets at $x$")
defined by setting%
\begin{equation*}
P\odot _{x}Q=J_{x}\left( PQ\right) \text{ for }P,Q\in \mathcal{P}\text{.}
\end{equation*}%
We write $C^{m}\left( \mathbb{R}^{n}\right) $ to denote the Banach space of
real-valued locally $C^{m}$ functions $F$ on $\mathbb{R}^{n}$ for which the
norm 
\begin{equation*}
\left\Vert F\right\Vert _{C^{m}\left( \mathbb{R}^{n}\right) }=\sup_{x\in 
	\mathbb{R}^{n}}\max_{\left\vert \alpha \right\vert \leq m}\left\vert
\partial ^{\alpha }F\left( x\right) \right\vert
\end{equation*}%
is finite. Similarly, for $D\geq 1$, we write $C^{m}\left( \mathbb{R}^{n},%
\mathbb{R}^{D}\right) $ to denote the Banach space of all $\mathbb{R}^{D}$%
-valued locally $C^{m}$ functions $F$ on $\mathbb{R}^{n}$, for which the
norm 
\begin{equation*}
\left\Vert F\right\Vert _{C^{m}\left( \mathbb{R}^{n},\mathbb{R}^{D}\right)
}=\sup_{x\in \mathbb{R}^{n}}\max_{\left\vert \alpha \right\vert \leq
	m}\left\Vert \partial ^{\alpha }F\left( x\right) \right\Vert
\end{equation*}%
is finite. Here, we use the Euclidean norm on $\mathbb{R}^{D}$.

If $F$ is a real-valued function on a cube $Q$, then we write $F\in
C^{m}\left( Q\right) $ to denote that $F$ and its derivatives up to $m$-th
order extend continuously to the closure of $Q$. For $F\in C^{m}\left(
Q\right) $, we define 
\begin{equation*}
\left\Vert F\right\Vert _{C^{m}\left( Q\right) }=\sup_{x\in
	Q}\max_{\left\vert \alpha \right\vert \leq m}\left\vert \partial ^{\alpha
}F\left( x\right) \right\vert .
\end{equation*}%
Similarly, if $F$ is an $\mathbb{R}^{D}$-valued function on a cube $Q$,
then we write $F\in C^{m}\left( Q,\mathbb{R}^{D}\right) $ to denote that $F$
and its derivatives up to $m$-th order extend continuously to the closure of 
$Q$. For $F\in C^{m}\left( Q,\mathbb{R}^{D}\right) $, we define 
\begin{equation*}
\left\Vert F\right\Vert _{C^{m}\left( Q,\mathbb{R}^{D}\right) }=\sup_{x\in
	Q}\max_{\left\vert \alpha \right\vert \leq m}\left\Vert \partial ^{\alpha
}F\left( x\right) \right\Vert \text{,}
\end{equation*}%
where again we use the Euclidean norm on $\mathbb{R}^{D}$.

If $F\in C^{m}\left( Q\right) $ and $x$ belongs to the boundary of $Q$, then
we still write $J_{x}\left( F\right) $ to denote the $\left( m-1\right)
^{rst}$ degree Taylor polynomial of $F$ at $x$, even though $F$ isn't
defined on a full neighborhood of $x\in \mathbb{R}^{n}$.

We write $\mathcal{M}$ to denote the set of all multiindices $\alpha =\left(
\alpha _{1},\cdots ,\alpha _{n}\right) $ of order $\left\vert \alpha
\right\vert =\alpha _{1}+\cdots +\alpha _{n} \leq m-1$.

We define a (total) order relation $<$ on $\mathcal{M}$, as follows. Let $%
\alpha =\left( \alpha _{1},\cdots ,\alpha _{n}\right) $ and $\beta =\left(
\beta _{1},\cdots ,\beta _{n}\right) $ be distinct elements of $\mathcal{M}$%
. Pick the largest $k$ for which $\alpha _{1}+\cdots +\alpha
_{k}\not=\beta _{1}+\cdots +\beta _{k}$. (There must be at least one such $k$%
, since $\alpha $ and $\beta $ are distinct). Then we say that $\alpha
<\beta $ if $\alpha _{1}+\cdots +\alpha _{k}<\beta _{1}+\cdots +\beta _{k}$.

We also define a (total) order relation $<$ on subsets of $\mathcal{M}$, as
follows. Let $\mathcal{A},\mathcal{B}$ be distinct subsets of $\mathcal{M}$,
and let $\gamma $ be the least element of the symmetric difference $\left( 
\mathcal{A\setminus B}\right) \cup \left( \mathcal{B\setminus A}\right) $
(under the above order on the elements of $\mathcal{M}$). Then we say that $%
\mathcal{A}<\mathcal{B}$ if $\gamma \in \mathcal{A}$.

One checks easily that the above relations $<$ are indeed total order
relations. Note that $\mathcal{M}$ is minimal, and the empty set $\emptyset $
is maximal under $<$. A set $\mathcal{A}\subseteq \mathcal{M}$ is called 
\underline{monotonic} if, for all $\alpha \in \mathcal{A}$ and $\gamma \in 
\mathcal{M}$, $\alpha +\gamma \in \mathcal{M}$ implies $\alpha +\gamma \in 
\mathcal{A}$. We make repeated use of a simple observation:

Suppose $\mathcal{A}\subseteq \mathcal{M}$ is monotonic, $P\in \mathcal{P}$
and $x_{0}\in \mathbb{R}^{n}$. If $\partial ^{\alpha }P\left( x_{0}\right)
=0 $ for all $\alpha \in \mathcal{A}$, then $\partial ^{\alpha }P\equiv 0$
on $\mathbb{R}^{n}$ for all $\alpha \in \A$.

This follows by writing $\partial ^{\alpha }P\left( y\right)
=\sum_{\left\vert \gamma \right\vert \leq m-1-\left\vert \alpha \right\vert }%
\frac{1}{\gamma !}\partial ^{\alpha +\gamma }P\left( x_{0}\right) \cdot
\left( y-x_{0}\right) ^{\gamma }$ and noting that all the relevant $\alpha
+\gamma $ belong to $\mathcal{A}$, hence $\partial ^{\alpha +\gamma }P\left(
x_{0}\right) =0$.

For finite sets $X$, we write $\#\left( X\right) $ to denote the numbers of
elements in $X$.

If $\lambda =\left( \lambda _{1},\cdots ,\lambda _{n}\right) $ is an $n$%
-tuple of positive real numbers, and if $\beta =\left( \beta _{1},\cdots
,\beta _{n}\right) \in \mathbb{Z}^{n}$, then we write $\lambda ^{\beta }$ to
denote 
\begin{equation*}
\lambda _{1}^{\beta _{1}}\cdots \lambda _{n}^{\beta _{n}}.
\end{equation*}%
We write $B_{n}\left( x,r\right) $ to denote the open ball in $\mathbb{R}%
^{n} $ with center $x$ and radius $r$, with respect to the Euclidean metric. 

\part{Convex Sets}
\label{part:cvx}
\section{Approximating Convex Sets}
  Given a convex set $K \subset \R^D$, we define $(1+\epsilon)\blacklozenge K = K + \frac{\epsilon}{2}K - \frac{\epsilon}{2}K$ and we want to approximate $K$ by a polytope described by  $k(D,\epsilon)$ half-spaces ($\xi_i\cdot v \leq b_i$) such that
  \begin{align}
    K\subset \{v: \xi_i\cdot v \leq b_i\forall 1\leq i \leq k(D,\epsilon)\} \subset (1+\epsilon)\blacklozenge K.    
  \end{align}

\begin{remark}
	If $K$ is not bounded, then it could be that $(1+\tau)\blacklozenge K$ is $\R^D$ for every $\tau > 0$ (for example if $K$ is a half-space).
\end{remark}

\begin{lemma}\label{lem1}
Let $K\subset\R^D$ be closed, convex, nonempty, bounded. Let $\hat{e}_1,\dots,\hat{e}_D$ be an orthonormal basis for $\R^D$, let $\lambda_1,\dots,\lambda_D$ be nonnegative real numbers, and let $w_0\in\R^D$ be given. Let $C_0>0$ be a real number. Assume:
\begin{enumerate}
  \item $w_0+\lambda_l\hat{e}_l$ and $w_0-\lambda_l\hat{e}_l$ belong to $K$ for each $l$.
  \item For each $l$, $\|w^+-w^-\|\leq C_0\lambda_l$ for all $w^+,w^-\in K$ s.t. $w^+-w^- \bot \hat{e}_{l'}$ for all $l'<l$.
\end{enumerate}
Then:
\begin{enumerate}
  \item[3.] $\{v\in\R^D: |(v-w_0)\cdot \hat{e}_l|\leq c_1\lambda_l,\,1\leq l \leq D\}\subset K$ and\\ $K \subset \{v\in\R^D: |(v-w_0)\cdot \hat{e}_l|\leq C_1\lambda_l,\,1\leq l\leq D\}$.
\end{enumerate}
\end{lemma}
\begin{proof}
  Assume, WLOG, that $w_0 = 0$ and $\hat{e}_1, \dots, \hat{e}_D$ are the usual unit vectors in $\R^D$, then 1. and 2. imply:
  \begin{enumerate}
    \item[4.] $(v_1,\dots,v_D)\in\R^D$ belongs to $K$ provided $|v_l|\leq c\lambda_l$ for each $l$.
    \item[5.] For each $l$, the following holds. Let $w^+, w^-$ be two points in $K$, s.t. $w_j^+ = w_j^-$ for all $j<l$. Then $\| w^+ - w^- \| \leq C_0 \lambda_l$
  \end{enumerate}
  Then by the following induction on $l$ one proves that if $v = (v_1, \dots, v_D)\in K$ then $|v_l|\leq C_l\lambda_l$ with $C_l$ determined by $C_0, D$. 
  
  We define $w^+ = (cv_1,\dots,cv_{l-1},0,\dots,0)$ and $w^-=cv$ ($c<1$). 
  
  By the induction step and 4., $w^+$ belongs in $K$, and $w^-$ also. Applying 5., we learn that $c|v_l|\leq \|w^+-w^-\| \leq C_0\lambda_l$. 
\end{proof}
\begin{definition}{1}
Fix a dimension $D$. A \textbf{descriptor} is an object of the form
\begin{equation}
  \Delta = [(\xi_i)_{i=1,\dots,I},(b_i)_{i=1,\dots,I}]
\end{equation}
where each $\xi_i$ is a vector in $\R^D$ and each $b_i$ is a real number. We call $I$ the \emph{length} of the descriptor $\Delta$ and we denote the length by $|\Delta|$.

If $\Delta$ is a descriptor, then we define:
\begin{equation}
  K(\Delta) = \{v\in\R^D :  \xi_i \cdot  v \leq b_i, \; i=1,\dots,I\}
\end{equation}
\end{definition}
We use \textbf{Megiddo's Algorithm} \cite{megiddoLinearProgrammingLinear1984} to give a solution (or say it's unbounded or unfeasible) to the problem:
\begin{equation*}
\begin{aligned}
& \underset{v\in\R^D}{\text{minimize}}
& & - \xi\cdot  v  \\
& \text{subject to}
& &  \xi_i\cdot  v  \leq b_i, \; i = 1, \ldots, I.
\end{aligned}
\end{equation*}
The work and storage are linear in $|\Delta|$, with constants depending only on $D$.

\begin{lemma}\label{lem:descsub} Given a descriptor $\Delta$ for which $K(\Delta)\subset \R^D$ is nonempty and bounded, and given a subspace $H\subset\R^D$ of dimension $L\geq 1$, there exists an algorithm producing vectors $v^+,v^-,\hat{e}$ and a scalar $\lambda$ s.t.:
  \begin{enumerate}
  \item[\refstepcounter{equation}\text{(\theequation)}\label{ds1}] $v^+,v^- \in K(\Delta)$ and $v^+-v^-\in H$.
  \item[\refstepcounter{equation}\text{(\theequation)}\label{ds2}] If $w^+,w^-$ are other vectors with property \eqref{ds1}, then $\|w^+-w^-\|\leq D^{1/2}\|v^+-v^-\|$.
  \item[\refstepcounter{equation}\text{(\theequation)}\label{ds3}] $\hat{e}\in H$, $\lambda\geq 0$, $\|\hat{e}\|=1$ and $v^+ - v^- = \lambda\hat{e}$.
  \end{enumerate}
The total work and storage required by the algorithm are at most $C|\Delta|$ where $C$ depends only on $D$.
\end{lemma}
\begin{algorithm}
  \TitleOfAlgo{Find diameter in subspace}
  \KwData{$\Delta$ such that $K(\Delta)$ is nonempty and bounded, and $\tilde{e}_1,\dots,\tilde{e}_L$ orthonormal basis for $H\subset\R^D$ different from $\{0\}$}
  \KwResult{vectors $v^+, v^-, \hat{e}$ and a scalar $\lambda$ as in Lemma \ref{lem:descsub}}

  \For{$l=1,\dots,L$}{
    Using \textbf{Megiddo's Algorithm}, solve the problem:
    \begin{equation*}
      \begin{array}{ll@{}ll}
        \text{maximize}_{v_l^+, v_l^-}  & \mu_l &&\\
        \text{subject to}& \mu_l = &(v^+_l - v^-_l)\cdot\tilde{e}_l &\\
                                        & & v_l^+-v_l^- \in H&\\
                                        & & v_l^+, v_l^- \in K(\Delta)&
      \end{array}
    \end{equation*}
  }
  $\hat{l} = \text{argmax}_l \mu_l$\;
  $v^+ = v_{\hat{l}}^+$, $v^- = v_{\hat{l}}^-$\;
  \uIf{$v^+ \neq v^-$}{
    $\lambda = \|v^+-v^-\|$\;
    $\hat{e} = \frac{v^+-v^-}{\lambda}$\;
  }\uElse{
    $\lambda=0$\;
    $\hat{e}$ any unit vector in $H$\;
  }
  \KwRet{$v^+, v^-, \hat{e}, \lambda$}
  \caption{Find diameter in subspace}
  \label{alg:descsub}
\end{algorithm}
 
\emph{Explanation for Algorithm \ref{alg:descsub}}:
If $w^+, w^-$ as in \eqref{ds2}, then $(w^+-w^-)\cdot \tilde{e}_l \leq \mu_l$; also $(w^- - w^+)\cdot \tilde{e}_l \leq \mu_l$. Thus $|(w^+-w^-)\cdot\tilde{e}_l|\leq \mu_l$.

Picking $\hat{l}$ to maximize $\mu_{\hat{l}}$, we see that any $w^+,w^-$ satisfying \eqref{ds2} satisfy also
\begin{equation*}
  \|w^+-w^-\|\leq D^{1/2}\mu_{\hat{l}} = D^{1/2}|(v_{\hat{l}}^+-v_{\hat{l}}^-)\cdot\tilde{e}_{\hat{l}}|\leq D^{1/2}\|v_{\hat{l}}^+ - v_{\hat{l}}^-\|.
\end{equation*}
Thus $v^+ = v_{\hat{l}}^+$ and $v^-= v_{\hat{l}}^-$ satisfy \eqref{ds1} and \eqref{ds2}.

For $\lambda, \hat{e}$, cases:
\begin{itemize}
  \item $v^+\neq v^-$, then $\lambda=\|v^+-v^-\|$ and $\hat{e}=\frac{v^+-v^-}{\lambda}$.
  \item Else, $\lambda=0$ and $\hat{e}$ any unit vector in $H$.
\end{itemize}
\qed

\begin{lemma}\label{lem:box}Given a descriptor $\Delta$ for which $K(\Delta)$ is nonempty and bounded, we produce vectors $\hat{e}_1,\dots,\hat{e}_D, w_0$ and scalars $\lambda_1,\dots,\lambda_D$ satisfying the hypotheses of Lemma \ref{lem1} for $K=K(\Delta)$, with some $C_0$ depending only on $D$. The work and storage are at most $C|\Delta|$ where $C$ depends only on $D$.
\end{lemma}
\begin{algorithm}
  \TitleOfAlgo{Produce Box}
  \KwData{$\Delta$ such that $K(\Delta)$ is nonempty and bounded}
  \KwResult{Vectors $\hat{e}_1,\dots,\hat{e}_D, w_0$ and scalars $\lambda_1,\dots,\lambda_D$ satisfying hypotheses of Lemma \ref{lem1}}
  $E =\{0\}$\;
  \For{$l=1,\dots,D$}{
    $H = \langle E \rangle ^\bot$\;
    $\tilde{e}_1,\dots,\tilde{e}_L$ orthonormal basis of $H$\;
    $v_l^+, v_l^-, \hat{e}_l, \hat{\lambda}_l = $ result of applying Algorithm \ref{alg:descsub} to $\Delta, H$\;
    $E = E \cup \{\hat{e}_l\}$\;
    $\lambda_l = \frac{1}{2D}\hat{\lambda}_l$\;
  }
  $w_0 = \frac{1}{2D}\sum_l (v_l^+ + v_l^-)$\;
  \KwRet{$\hat{e}_1,\dots,\hat{e}_D, w_0, \lambda_1,\dots,\lambda_D$}
  \caption{Produce Box from descriptor}
  \label{alg:box}
\end{algorithm}

\emph{Explanation of Algorithm \ref{alg:box}}: For $l=1,\dots,D$ we will produce vectors $v_l^+, v_l^-, \hat{e}_l$ and a scalar $\hat{\lambda}_l$ s.t.:
\begin{enumerate}
  \item $v_l^+,v_l^-\in K(\Delta)$
  \item $v_l^+-v_l^-\bot \hat{e}_{l'}$ for $l'<l$
  \item if $w_l^+,w_l^-\in K(\Delta)$ are other vectors such that $w_l^+-w_l^-\bot \hat{e}_{l'}$ for $l'<l$, then $\|w_l^+-w_l^-\|\leq D^{1/2}\|v_l^+-v_l^-\|$.
  \item $\hat{e}_l\bot\hat{e}_{l'}$ for all $l'<l$
  \item $\hat{\lambda}_l\geq 0$ and $v_l^+-v_l^- = \hat{\lambda}_l\hat{e}_l$.
\end{enumerate}
To do so, we proceed by induction on $l$. Given that we have constructed these for $l'<l$ then we compute the next by applying Algorithm 1 with $H$ the orthocomplement of $\text{span}\{\hat{e}_{l'}, l'<l\}$. At the end we compute:
\begin{equation*}
  w_0 = \frac{1}{2D}\sum_l (v_l^+ + v_l^-)
\end{equation*}
and $\lambda_l = \frac{\hat{\lambda}_l}{2D}$. These satisfy the hypotheses of Lemma 1 for $K(\Delta)$.
\qed

We will work with a small parameter $\tau>0$. We write $c(\tau), C(\tau),...$ to denote constants depending only on $m,n,\tau$.  Recall that if $\Gamma\subset\P$ is a nonempty bounded convex set, we write $(1+\tau)\blacklozenge\Gamma$ to denote the convex set $\Gamma -\frac{\tau}{2}\Gamma + \frac{\tau}{2}\Gamma$. 

\begin{lemma}
	\label{lem:dilation}
	Let $\Gamma = w_0 + \sigma$ where $\Gamma, \sigma \subset \P$ are convex sets, $A^{-1}B \subset \sigma \subset AB$ for the Euclidean unit ball $B\subset \R^D$, some $A>1$ and $w_0 \in \P$. Then:
	\begin{equation}
		(1+\tau)A^{-2}\sigma \subset \Gamma-\frac{\tau}{2}\Gamma + \frac{\tau}{2}\Gamma - w_0 \subset (1+\tau)A^2\sigma
	\end{equation}
\end{lemma}
\begin{proof}
	Assume, WLOG, that $w_0 = 0$.
	
	Let $P = P_0 + \frac{\tau}{2}P_1 - \frac{\tau}{2}P_2 \in \Gamma-\frac{\tau}{2}\Gamma + \frac{\tau}{2}\Gamma$, with $P_0, P_1, P_2 \in \Gamma$.
	
	Examining $\|P\|$, we see that $\|P\| \leq (1+\tau)A$. Therefore $\Gamma -\frac{\tau}{2}\Gamma + \frac{\tau}{2}\Gamma - w_0 \subset (1+\tau)AB \subset (1+\tau)A^2\sigma$.
	
	On the other hand if $P \in (1+\tau)A^{-2}\sigma$ then $P = (1+\tau)P' = P' + \frac{\tau}{2}P' - \frac{\tau}{2}(-P')$, $P' \in A^{-2}\sigma$. Since $A^{-2}\sigma \subset A^{-1}B \subset \sigma$ we have $-P' \in A^{-1}B \subset \sigma$ and thus $P \in \sigma + \frac{\tau}{2}(\sigma - \sigma)$. In conclusion, $(1+\tau)A^{-2}\sigma \subset \Gamma -\frac{\tau}{2}\Gamma + \frac{\tau}{2}\Gamma$.
\end{proof}
	
\begin{lemma}\label{lem2} Let $\Lambda$ be a $\tau$-net in the Euclidean unit ball $B\subset\R^D$, and let $K\subset\R^D$ be a closed convex set satisfying $A^{-1} B \subset K \subset AB$ for some given $A>1$. Let $0 < \tau \leq \frac{1}{2}$.

  Define $K_{\tau} = \{v\in \R^D:  \xi\cdot v \leq \max_{w\in K}  \xi\cdot w \;\forall \xi \in \Lambda\}$.

  Then $K\subset K_{\tau} \subset (1+6A^2\tau)\blacklozenge K$

\end{lemma}
\begin{proof} Obviously $K\subset K_\tau$.\\
  Let $v\in K_\tau$, and $\xi\in\R^D$ with $\|\xi\| = 1$. Pick $\eta\in\Lambda$ such that $\xi-\eta \in \tau B$. Then:
  \begin{flalign*}
     \xi\cdot v &\leq  \eta\cdot v + \tau\|v\|  &&\\
                          &\leq \max_{w\in K}  \eta\cdot w + \tau\|v\| &&\text{[Because } v\in K_\tau\text{]}\\
                          &\leq \max_{w\in K}  \xi\cdot  w  +\|\xi-\eta\|\max_{w\in K} \|w\| + \tau \|v\| &&\\
                          &\leq \max_{w\in K}  \xi\cdot  w + A\tau + \tau\|v\| &&
  \end{flalign*}
  Also, $\eta\cdot v \leq \max_{w\in K} \eta\cdot w\leq \|\eta\|\max_{w\in K}\|w\|\leq A$, hence the above inequalities show that $\xi\cdot v \leq A + \tau \|v\|$, for any $\xi \in \R^D$ with $\|\xi\| = 1$. Thus $\|v\|\leq 2A$ and therefore $\xi\cdot v\leq\max_{w\in K}\xi\cdot w + 3A \tau $.

  On the other hand,
  \begin{equation}
    \max_{w\in K} \xi\cdot w \geq \max_{w\in A^{-1}B}  \xi\cdot w = A^{-1}
  \end{equation}
  Therefore
  \begin{equation}
     \xi \cdot  v \leq \max_{w\in K}  \xi\cdot  w + 3A^2\tau\max_{w\in K}\xi\cdot w\;\forall v\in K_\tau,\|\xi\|=1
  \end{equation}
  Because $K$ is compact, convex and $0\in K$, it follows $\frac{v}{1+3A^2\tau}\in K$ for any $v\in K_\tau$. That is, any $v \in K_{\tau}$ also is in $(1+3A^2\tau)K$ and we can write it as $v = (1+3A^2\tau)v'$ for $v' \in K$. Therefore it is $v = v' +\frac{6A^2 \tau}{2} v' - \frac{6A^2 \tau}{2} (0)$ so $v \in K + \frac{3A^2\tau}{2} K - \frac{3A^2\tau}{2} K$ so $v \in (1+6A^2\tau)\blacklozenge K$.

\end{proof}
\begin{lemma}\label{lem:approx-tau}Given $\tau>0$ and given a descriptor $\Delta$ for which $K(\Delta)$ is nonempty and bounded, there is an algorithm that produces a vector $w_0$ and a descriptor $\tilde{\Delta}$ with the following properties:
\begin{enumerate}
  \item $|\tilde{\Delta}|$ is bounded by a constant determined by $\tau$ and $D$.
  \item $K(\tilde{\Delta})\subset K(\Delta)-w_0\subset (1+\tau) \blacklozenge K(\tilde{\Delta})$.
\end{enumerate}
The work and storage used are at most $C(\tau)|\Delta|$, where $C(\tau)$ is determined by $\tau$ and $D$.
\end{lemma}
\begin{algorithm}
  \TitleOfAlgo{Approximating Polytopes}
  \KwData{$\tau >0$, $\Lambda$ a $\frac{\tau}{C(A)}$-net in the Euclidean unit ball $B\subset\R^D$ and $\Delta$ such that $K(\Delta)$ is nonempty and bounded. Here $A$ is a constant depending only on $D$.}
  \KwResult{vector $w_0$ and descriptor $\tilde{\Delta}$ with the properties of Lemma \ref{lem:approx-tau}}
  $w_0, \hat{e}_1,\dots,\hat{e}_D,\lambda_1,\dots,\lambda_D = $ result of applying Algorithm \ref{alg:box} to $\Delta$\;
  Apply a linear transformation $T$ to $K(\Delta)$ (using $w_0, \hat{e}_1,\dots,\hat{e}_D,\lambda_1,\dots,\lambda_D$) to obtain a $\hat{\Delta}$ such that $K(\hat{\Delta}) = T(K(\Delta))$ and $\{v\in\R^{D'}: |v_l|\leq A^{-1} \;\forall l\} \subset K(\hat{\Delta}) \subset \{v\in\R^{D'}: |v_l|\leq A\;\forall l\}$\;
  $\hat{\Delta}_\tau = [\emptyset$]\;
  \ForEach{$\xi \in \Lambda$}{
  	
    $\mu_\xi = \max_{w\in K(\hat{\Delta})} \xi\cdot w$ (use Megiddo's Algorithm)\;
    $\hat{\Delta}_\tau = \Delta_\tau \cup \{[\xi, \mu_\xi] \}$\;
  }
  Apply the inverse linear transformation $T^{-1}$ to $K(\hat{\Delta}_{\tau})$ to obtain $\Delta_\tau$ such that $T(K(\Delta_{\tau})) = K(\hat{\Delta}_\tau)$\;
  \KwRet{$\Delta_\tau$}
  \caption{Approximating Polytopes}
  \label{alg:approx-tau}  
\end{algorithm}

\emph{Explanation:}
Suppose first that we know that
\begin{equation}
  \{v\in\R^D: |v_l|\leq A^{-1} \;\forall l\} \subset K(\Delta) \subset \{v\in\R^D: |v_l|\leq A\;\forall l\}
\end{equation}
for some given constant $A$. By applying Lemma \ref{lem2}, together with Megiddo's algorithm to compute $\max_{w\in K}\xi\cdot w$ for each $\xi\in\Lambda$ (as in Lemma 2), we can compute, using work and storage at most $C(A,\tau)|\Delta|$ a descriptor $\Delta_\tau$ such that $|\Delta_\tau|\leq C(A,\tau)$ and
\begin{equation}
  K(\Delta)\subset K(\Delta_\tau) \subset (1+6A^2\tau)\blacklozenge K(\Delta)
\end{equation}

Next, suppose that we know that
\begin{equation}
  \{v\in\R^D: |v_l|\leq \lambda_lA^{-1} \;\forall l\} \subset K(\Delta) \subset \{v\in\R^D: |v_l|\leq \lambda_lA\;\forall l\}
\end{equation}
for known positive numbers $A,\lambda_1,\dots,\lambda_D$. We can trivially reduce the problem to the previous case (rescaling). If instead of assuming that all $\lambda_l$ are positive, we assume that they are nonnegative, we can reduce the problem to a lower dimensional one.

Next, if we have vectors $w_0,\hat{e}_1,\dots,\hat{e}_D$ and scalars $\lambda_1,\dots,\lambda_D\geq 0$ such that the $\hat{e}_l$ form an orthonormal basis of $\R^D$ and
\begin{equation}
  \{v\in\R^D: |( v- w_0)\cdot \hat{e}_l|\leq \lambda_lA^{-1} \forall l\} \subset K(\Delta) \subset \{v\in\R^D: |( v- w_0)\cdot \hat{e}_l|\leq \lambda_lA\forall l\}
\end{equation}
we can compute a descriptor $\Delta_\tau$ s.t. $|\Delta_\tau|\leq C(A,\tau)$ and $w_0+K(\Delta_\tau)\subset K(\Delta) \subset w_0 + (1+\tau)\blacklozenge K(\Delta_\tau)$.

Finally, given a descriptor $\Delta$ we apply Algorithm \ref{alg:box} to find $w_0,\hat{e}_1,\dots,\hat{e}_l$, $\lambda_1,\dots,\lambda_D$ with $A$ depending only on $D$. We get the desired descriptor from there.

\begin{remark}
	Note that a $\tau$-net of the unit ball contains $C\tau^{-D}$ points. That is both the number of Linear Programming Problems that will be solved, and the size of the resulting descriptor. During the rest of the document, we recommend to the reader that they read $C(\tau)$ as $C\tau^{-D}$ to gauge the size of the constants appearing in the runtimes and space requirements of the algorithm.
\end{remark}

We end this subsection with a result that will be used later in the specific application to the smooth selection problem.

\begin{lemma}
	\label{lem:twicetau}
	Let $K$ be a convex set. Then $(1+\tau) \blacklozenge ((1+\tau) \blacklozenge K) = (1+(2+\tau)\tau)\blacklozenge K$.
\end{lemma}

\begin{proof}
	Let $x \in (1+\tau) \blacklozenge ((1+\tau) \blacklozenge K)$. Then $x = x_0 + \frac{\tau}{2}x_1 - \frac{\tau}{2}x_2$ for $x_i \in (1+\tau)\blacklozenge K$. In turn each $x_i = x_{i,0} + \frac{\tau}{2}x_{i,1} - \frac{\tau}{2}x_{i,2}$ where $x_{i,j} \in K$.
	
	Therefore
	\begin{equation*}
	x = x_{0,0} + \frac{\tau}{2}(x_{0,1}-x_{0,2}) + \frac{\tau}{2}(x_{1,0}-x_{2,0}) + \left( \frac{\tau}{2} \right)^2 (x_{1,1}-x_{1,2}) + \left( \frac{\tau}{2} \right)^2 (x_{2,1}-x_{2,2})
	\end{equation*}
	
	Each of the summands (except $x_{0,0}$ which belongs to $K$) is a member of $K-K$, a symmetric convex set. Therefore, $x \in K + \frac{\tau}{2}(K-K) + \frac{\tau}{2}(K-K) + \left( \frac{\tau}{2} \right)^2 (K-K) + \left( \frac{\tau}{2} \right)^2 (K-K)$. Because $K-K$ is symmetric we can group these Minkowski sums, so 
	
	\begin{flalign*}
	K + \frac{\tau}{2}(K-K) + \frac{\tau}{2}(K-K) + &\\\left( \frac{\tau}{2} \right)^2 (K-K) + \left( \frac{\tau}{2} \right)^2 (K-K) &= K + (\frac{\tau}{2} + \frac{\tau}{2} + \left( \frac{\tau}{2} \right)^2  + \left( \frac{\tau}{2} \right)^2) (K-K)\\
	&= K + (\tau+ \frac{\tau^2}{2})(K-K)\\
	&= K + (\tau+ \frac{\tau^2}{2})K - (\tau+ \frac{\tau^2}{2})K \\
	&= (1+(2+\tau)\tau)\blacklozenge K
	\end{flalign*}
	
	The reverse inclusion proceeds similarly. Let $x \in (1+(2+\tau)\tau) \blacklozenge K$. Therefore $x = x_0 + \frac{(2+\tau)\tau}{2}x_1 - \frac{(2+\tau)\tau}{2}x_2$ where $x_i \in K$. Now we can reverse the above operations and see that
	\begin{flalign*}
	x &= x_0 + \frac{(2+\tau)\tau}{2}x_1 - \frac{(2+\tau)\tau}{2}x_2 \\
	  &= x_0 + (\frac{\tau}{2}+\frac{\tau}{2} +\left(\frac{\tau}{2}\right)^2 + \left(\frac{\tau}{2}\right)^2)(x_1-x_2)\\
	  &= \underbrace{x_0 + \frac{\tau}{2}(x_1-x_2)}_{\tilde{x}_0 \in (1+\tau)\blacklozenge K} + \frac{\tau}{2}\underbrace{(x_1+\frac{\tau}{2}(x_1-x_2))}_{\tilde{x}_1\in (1+\tau)\blacklozenge K} - \frac{\tau}{2}\underbrace{(x_2+\frac{\tau}{2}(x_2-x_1))}_{\tilde{x}_2\in (1+\tau) \blacklozenge K} \\
	  &= \tilde{x}_0 +\frac{\tau}{2}\tilde{x}_1 - \frac{\tau}{2}\tilde{x}_2
	\end{flalign*}
	belongs to $(1+\tau) \blacklozenge ((1+\tau) \blacklozenge K)$.
	
\end{proof}

\section{Approximate Minkowski Sums}
Let $\Boxx=\{v\in\R^D: |v\cdot\hat{e}_i|\leq\lambda_i, i = 1,\dots,D\}$ and
$\Boxx'=\{v\in\R^D: |v\cdot\hat{e}'_i|\leq\lambda'_i, i=1,\dots,D\}$ where $\hat{e}_i$ and $\hat{e}_i'$ ($i=1,\dots,D$) are orthonormal bases for $\R^D$ and $\lambda_i, \lambda'_i$ are nonnegative numbers.

We will say here that two symmetric convex sets $K_1, K_2$ are "comparable" if $cK_1\subset K_2\subset CK_1$ for $c,C$ depending only on $D$.

Let $I=\{i:\lambda_i\neq 0\}$ and $I'=\{i:\lambda'_i\neq 0\}$. Let $V=\text{span}\{\hat{e}_i: i\in I\}$ and $V'=\text{span}\{\hat{e}_i':i\in I'\}$.

A box \Boxx can be written equivalently as $\Boxx=\{v\in V: |v\cdot\hat{e}_i|\leq\lambda_i, i \in I \} \subset \R^D$. It is comparable to an \textbf{Ellipsoid} $E = \{v\in V: q(v) = \sum_{i\in I} (\frac{v\cdot\hat{e}_i}{\lambda_i})^2 \leq 1\}$.

\begin{algorithm}
	\TitleOfAlgo{BoxAMS (Box Approximate Minkowski Sum)}
	\KwData{Two nonempty boxes, $\Boxx$ and $\Boxx'$.}
	\KwResult{A box $\overline{\Boxx}$ comparable to $\Boxx+ \Boxx'$.}
	$I = \{i:\lambda_i \neq 0\}$, $I' = \{i:\lambda'_i \neq 0\}$\;
	$V = \text{span}\{\hat{e}_i: i\in I\}$, $V' = \text{span}\{\hat{e}_i':i\in I'\}$\;
	$D' = \text{dim}(V+V')$\;
	Define $Q(w) = \min\limits_{\substack{v+v'=w \\ v\in V, v'\in V'}} q(v) + q'(v')$ for $w \in V+V'$\;
	Diagonalize $Q$ to obtain an orthonormal basis $\tilde{e}_1,\dots,\tilde{e}_L$ for $V+V'$ and positive numbers $\mu_1,\dots,\mu_L$\;
	Complete orthonormal basis to $\R^D$, with $\mu_i = 0$ for $i > L$\;
	\KwRet{$\overline{\Boxx} = \{v\in\R^D: |v\cdot\tilde{e}_i|\leq\mu_i, i=1,\dots,D\}$}
	\caption{BoxAMS}
	\label{alg:BoxAMS}
\end{algorithm}

We will compute a box comparable to the Minkowski sum $\Boxx+\Boxx'$.
We know $\Boxx$ is comparable to $\textbf{Ellipsoid}$  and $\Boxx'$ is comparable to $\textbf{Ellipsoid}'$.
Then $\Boxx+\Boxx'$ is comparable to
\begin{align}
  \textbf{Ellipsoid}+\textbf{Ellipsoid}'=\{w\in V+V':\min_{\substack{v+v'=w \\ v\in V, v'\in V'}} \max\{q(v),q'(v')\}\leq 1\}
\end{align}

which in turn is comparable to $\{w\in V+V': \min\limits_{\substack{v+v'=w \\ v\in V, v'\in V'}} q(v) + q'(v') \leq 1\} $ .
The minimum here may be expressed as $Q(w)$ for a positive definite quadratic form $Q$ on $V+V'$. By diagonalizing $Q$ we find an orthonormal basis $\tilde{e}_1,\dots,\tilde{e}_L$ for $V+V'$ and positive numbers $\mu_1,\dots,\mu_L$
such that $\Boxx+\Boxx'$ is comparable to
\begin{align}
  \{w\in V+V' : \sum_{i=1,\dots,L}(\frac{w\cdot \tilde{e}_i}{\mu_i})^2 \leq 1\}
\end{align}
Completing $\tilde{e}_1,\dots,\tilde{e}_L$ to an orthonormal basis $\tilde{e}_1,\dots,\tilde{e}_D$ of $\R^D$,
 and setting $\mu_i = 0$ for $i=L+1,\dots,D$ we see that $\Boxx+\Boxx'$ is comparable to $\{w\in\R^D: |w\cdot \tilde{e}_i|\leq\mu_i, i=1,\dots,D\}$. Algorithm \ref{alg:BoxAMS} describes this process, and the total work and storage to compute this box is at most $C(D)$, a constant depending only on $D$.

 \begin{algorithm}[H]
   \TitleOfAlgo{AMS (Approximate Minkowski Sum)}
   \KwData{Two nonempty bounded convex polytopes $K=K(\Delta)$ and $K'=K(\Delta')$ in $\R^D$,  $\tau>0$}
   \KwResult{Convex polytope $\tilde{K} = K(\tilde{\Delta})$ with $|\tilde{\Delta}|\leq C(\tau)$ such that $K+K' \subset \tilde{K} \subset (1+\tau)\blacklozenge(K+K')$}
   \uIf{$K == \emptyset$ or $K' == \emptyset$}{
     \KwRet{$\emptyset$}
   }
   $w_0, \hat{e}_1, \dots, \hat{e}_D, \lambda_1,\dots, \lambda_D = $ result of applying Algorithm \ref{alg:box} to $\Delta$\;
   $w_0', \hat{e}_1', \dots, \hat{e}_D', \lambda_1',\dots, \lambda_D' = $ result of applying Algorithm \ref{alg:box} to $\Delta'$\;
   $\Boxx = \{v\in\R^D: |v\cdot\hat{e}_i|\leq\lambda_i, i = 1,\dots,D\}$\;
   $\Boxx' = \{v\in\R^D: |v\cdot\hat{e}'_i|\leq\lambda'_i, i = 1,\dots,D\}$\;
   $\widetilde{\Boxx} = \text{BoxAMS}(\Boxx, \Boxx')$.
   $I = \{i:\lambda_i \neq 0\}$, $I' = \{i:\lambda'_i \neq 0\}$\;
   $V = \text{span}\{\hat{e}_i: i\in I\}$, $V' = \text{span}\{\hat{e}_i':i\in I'\}$\;
   $D' = \text{dim}(V+V')$\;
   $\Lambda$ a $\frac{\tau}{C(A')}$-net on $B(V+V')$, $A'$ a constant depending only on $D'$\;
   Rescale and recenter both $K$ and $K'$ as in Algorithm \ref{alg:approx-tau} with $\tilde{e_i}$ and $\mu_i$ from $\widetilde{\Boxx}$\;
   $\tilde{\hat{\Delta}} = \{\emptyset\}$\;
   \ForEach{$\xi \in \Lambda$}{
     $\mu_\xi = \max_{w\in \hat{K}} \xi\cdot w + \max_{w'\in \hat{K}'}\xi\cdot w'$\;
     $\tilde{\hat{\Delta}} = \tilde{\hat{\Delta}} \cup \{[\xi, \mu_\xi]\}$\;
   }
   Rescale and recenter $\tilde{\hat{\Delta}}$ as in Algorithm \ref{alg:approx-tau} to produce $\tilde{\Delta}$\;
   \KwRet{$\tilde{\Delta}$}\;
   \caption{AMS}
   \label{alg:AMS}
\end{algorithm}

\emph{Explanation of Algorithm \ref{alg:AMS}}: We write $C, C'$, etc. to denote constants depending only on $D$. By an earlier algorithm we can find points $w\in K$, $w'\in K'$ and rectangular boxes

\begin{align}
  \Boxx = \{ v\in \R^D : |v\cdot\hat{e}_i|\leq\lambda_i, i=1,\dots,D\}\\
  \Boxx' = \{ v\in \R^D: |v\cdot\hat{e}_i'|\leq\lambda'_i, i=1,\dots,D\}
\end{align}

such that $\Boxx\subset K-w\subset C\Boxx$ and $\Boxx'\subset K'-w' \subset C\Boxx'$. Without loss of generality we may assume $w=w'=0$. We then apply the algorithm immediately preceding this one, to compute a rectangular box
$\widetilde{\Boxx}\subset\R^D$ such that $\Boxx+\Boxx'\subset \widetilde{\Boxx}\subset C(\Boxx+\Boxx')$, and therefore
\begin{align}
  c\widetilde{\Boxx}\subset K+K' \subset C \widetilde{\Boxx}
  \label{tildebox}
\end{align}
By applying an invertible linear map to $\R^D$ we may assume that (\ref{tildebox}) holds with
\begin{align}
  \widetilde{\Boxx} = \{v=(v_1,\dots,v_D)\in\R^D: |v_i|\leq 1, i=1,\dots,I, v_i=0, i>I\}
\end{align}
for some $I$. We may regard $K, K'$ as subsets of $\R^I$. We may now apply Algorithm 3 but maximizing over $K+K'$ instead of a single $K$. To compute it we simply compute
\begin{align}
  \max_{w\in K+K'} \xi\cdot w = \max_{w\in K} \xi\cdot w+\max_{w\in K'} \xi \cdot w
\end{align}
The work used to do the above is at most $C(\tau)[|\Delta|+|\Delta'|]$.

\section{Approximate Intersections}
In this section, we present an algorithm to compute an approximation of the intersection of $k$ nonempty, bounded convex sets $K_1 = K(\Delta_1),\dots,K_k = K(\Delta_k)$. We use the tools and algorithms from previous sections. The algorithm uses work and storage at most $C(\tau)\sum_l |\Delta_l|$ with $C(\tau)$ determined by $\tau, D$.

\begin{remark}
	\label{rem:intersect}
	The intersection of $k$ non-empty convex sets $K_1, \dots, K_k$ given by $k$ descriptors $\Delta_1, \dots, \Delta_k$ is described by the union $\cup_l \Delta_l$. The algorithm is needed to keep the size of the descriptor controlled even when $k$ is very large.
\end{remark}
 \begin{algorithm}
   \TitleOfAlgo{AI (Approximate Intersection)}
   \KwData{$k$ nonempty bounded convex polytopes $K_l=K(\Delta_l)$ in $\R^D$,  $\tau>0$}
   \KwResult{Convex polytope $\tilde{K} = K(\tilde{\Delta})$ with $|\tilde{\Delta}|\leq C(\tau)$ such that $\cap_{l}K_l \subset \tilde{K} \subset (1+\tau)\blacklozenge(\cap_lK_l)$}
   $\hat{\Delta}= \{\emptyset\}$\;
   \For{$l = 1,\dots,k$}{
     \uIf{$K_l == \emptyset$}{
       \KwRet{$\emptyset$}
     }
     $\hat{\Delta} = \hat{\Delta}\cup\Delta_l$\;
   }
   $\tilde{\Delta} = $ result of applying Algorithm \ref{alg:approx-tau} to $\hat{\Delta}$ and $\tau$\;
   \KwRet{$\tilde{\Delta}$}
   \caption{Algorithm: AI}
   \label{alg:AI}
 \end{algorithm}
 \clearpage

\part{Blob Fields and Their Refinements}

\section{Finding Critical Delta}\label{sec:find-critical-delta}

In this section we work in $\P$, the vector space of polynomials of degree less than or equal to $m-1$ on $\mathbb{R}^n$. We denote (possibly empty) convex sets of polynomials by $\G$. Let $D = \dim \P$. Constants $c,C,C',$etc. depend only on $m, n$ unless we say otherwise.

Recall from \cite{feffermanFittingSmoothFunction2009a}.

\begin{lemma}{"Find Critical Delta" in Symmetric Case.}

Let $\xi_1,\dots,\xi_D$ be linear functionals on $\P$, and let $\lambda_1,\dots,\lambda_D$ be nonnegative real numbers. Let $\A\subset\M$ and let $x_0\in\R^n$, let $A\geq 1$.
There exists an algorithm that given the above produces $\hat{\delta}\in[0,\infty]$ for which the following hold:
\begin{enumerate}
  \item[(I)] Given $0<\delta<\hat{\delta}$ there exist $P_{\alpha}\in\P$ ($\alpha\in\A$) such that:
  \begin{enumerate}
    \item[(A)] $\partial^\beta P_\alpha(x_0) = \delta_{\beta\alpha}$ for $\beta,\alpha\in\A$.
    \item[(B)] $|\partial^\beta P_\alpha(x_0)| \leq CA\delta^{|\alpha|-|\beta|}$ for $\beta\in\M$, $\alpha\in\A$, $\beta\geq\alpha$.
    \item[(C)] $|\xi_{l}(\delta^{m-|\alpha|}P_\alpha)|\leq CA\lambda_l$ for $\alpha\in \A$, $l=1,\dots,D$.
  \end{enumerate}
  \item[(II)] Suppose $0<\delta<\infty$ and $P_\alpha\in\P$ ($\alpha\in\A$) satisfy
  \begin{enumerate}
    \item[(A)] $\partial^\beta P_\alpha(x_0) = \delta_{\beta\alpha}$ for $\beta,\alpha\in\A$.
    \item[(B)] $|\partial^\beta P_\alpha(x_0)| \leq cA\delta^{|\alpha|-|\beta|}$ for $\beta\in\M$, $\alpha\in\A$, $\beta\geq\alpha$.
    \item[(C)] $|\xi_{l}(\delta^{m-|\alpha|}P_\alpha)|\leq cA\lambda_l$ for $\alpha\in \A$, $l=1,\dots,D$.
  \end{enumerate}
  Then $0<\delta<\hat{\delta}$.
\end{enumerate}
The work and storage used to compute $\hat{\delta}$ are at most $C$ (see Lemma 1 in section 8 of "Fitting II" \cite{feffermanFittingSmoothFunction2009a}).
\end{lemma}

We study the case in which $\Gamma = K(\Delta)$, the compact convex polytope arising from a descriptor $\Delta$. Recall that we can use the results from Part \ref{part:cvx} to compute $P_w\in\Gamma$, linear functionals $\xi_1,\dots,\xi_D$ on $\P$, and nonnegative real numbers $\lambda_1,\dots,\lambda_D$ such that:
\begin{equation}
  \{P\in\P: |\xi_l(P-P_w)|\leq \lambda_l\}\subset\Gamma\subset\{p\in\P:|\xi_l(P-P_w)|\leq C\lambda_l\}
\end{equation}
If we set $\sigma = \{P\in\P: |\xi_l(P)|\leq \lambda_l \forall l\}$ then it follows that $\Gamma+c\tau\sigma\subset (1+\tau)\blacklozenge\Gamma$.
\begin{lemma} Find Critical Delta, General Case.
\label{lem:find-crit-delta}
Given $\emptyset\neq\A\subset\M$, $x_0\in\R^n$, $A\geq 1$, $M\geq 1$, $1>\tau>0$, $\Gamma_{in}=K(\Delta_{in})\subset\Gamma=K(\Delta)\subset\P$ with $\Gamma_{in},\Gamma$ non-empty, compact; we compute $\tilde{\delta}\in\left[0,\infty \right)$ with the following properties:
\begin{enumerate}
  \item[(I)] There exist $P_w\in\Gamma_{in}$ and $P_\alpha\in\P$ ($\alpha\in\A$), that we compute as well, such that:
  \begin{enumerate}
    \item[(A)] $\partial^\beta P_\alpha (x_0) =\delta_{\beta\alpha}$ for $\beta,\alpha\in\A$.
    \item[(B)] $|\partial^\beta P_\alpha(x_0)| \leq CA\tilde{\delta}^{|\alpha|-|\beta|}$ for $\alpha\in\A$,$\beta\in\M$, $\beta\geq\alpha$.
    \item[(C)] $P_w \pm \frac{M\tilde{\delta}^{m-|\alpha|}P_\alpha}{CA}\in(1+\tau)\blacklozenge\Gamma$
  \end{enumerate}
  \item[(II)] Suppose $0<\delta<\infty$ and $P_w \in \Gamma_{in}$, $P_\alpha\in\P$ ($\alpha\in \A$) satisfy:
  \begin{enumerate}
    \item[(A)] $\partial^\beta P_\alpha (x_0) =\delta_{\beta\alpha}$ for $\beta,\alpha\in\A$.
    \item[(B)] $|\partial^\beta P_\alpha(x_0)| \leq cA\delta^{|\alpha|-|\beta|}$ for $\alpha\in\A$,$\beta\in\M$, $\beta\geq\alpha$.
    \item[(C)] $P_w \pm \frac{M\delta^{m-|\alpha|}P_\alpha}{cA}\in(1+\tau)\blacklozenge\Gamma$
  \end{enumerate}
\end{enumerate}
Then $0<\delta\leq\tilde{\delta}$.

The work and storage used are at most a constant determined by $|\Delta_{in}|$, $|\Delta|, \tau, m, n$.
\end{lemma}
\begin{algorithm}
	\label{alg:find-crit-delta}
	\TitleOfAlgo{Find Critical Delta}
	\KwData{$\emptyset\neq\A\subset\mathcal{M}$, $x_0\in\R^n$, $A\geq 1$, $M\geq 1$, $\tau>0$, $\Gamma_{in} = K(\Delta_{in}) \subset \Gamma = K(\Delta) \subset\P$ with $\Gamma_{in},\Gamma$ non-empty, compact.}
	\KwResult{$\tilde{\delta}$, $P_w \in \Gamma_{in}$, $P_\alpha \in \P$ ($\alpha \in \A$) as in Lemma \ref{lem:find-crit-delta}}
	$\xi_1,\dots,\xi_D,\tilde{P}_w,\tilde{\lambda}_1,\dots,\tilde{\lambda}_D$ = result of applying Algorithm \ref{alg:box} to $\Delta$.\;
	$\lambda_i = \frac{\tilde{\lambda}_i}{M}$ for all $i = 1,\dots, D$\;
	$\sigma =\{P\in\P: |\xi_l(P)|\leq \lambda_l,\, l=1,\dots,L\}$\;
	$\hat{\delta} = $ result of applying Algorithm "Find Critical Delta in Symmetric Case" to $\sigma$\;
	\uIf{$\hat{\delta}>0$}{
		Produce $[\tau\hat{\delta} = \delta_1, \delta_2,\dots, \delta_{\nu_{\max}} = \hat{\delta}]$ such that $\delta_{\nu +1}\leq 2\delta_\nu$ and $\nu_{\max} \leq C\log\frac{10}{\tau}$\;
		\For{$\nu = 1,\dots, \nu_{\max}$}{
			Use Megiddo's Algorithm to solve 
		   \begin{equation*}
			\begin{array}{ll@{}ll}
			\mathop{\text{maximize}}\limits_{\substack{P_w \in \Gamma_{in}\\ P_\alpha \in P\\ Q_{\nu\alpha}, Q_{\nu\alpha}', Q_{\nu\alpha}''\in \Gamma}}  & 1 &&\\
			\text{subject to}& \partial^\beta P_\alpha(x_0) &=\delta_{\beta\alpha} & \alpha,\beta\in\A\\
			&|\partial^\beta P_\alpha(x_0)| &\leq CA\delta_\nu^{|\alpha|-|\beta|}&\alpha \in \A, \beta\geq\alpha\\
			&P_w \pm \frac{M\delta_\nu^{m-|\alpha|}P_{\alpha}}{CA} &= Q_{\nu\alpha} - \frac{\tau}{2}Q_{\nu\alpha}' + \frac{\tau}{2}Q_{\nu\alpha}''& \alpha\in\A
			\end{array}
			\end{equation*}
		}
		$\tilde{\delta} = \delta_{\nu^{'}}$ the max $\nu$ such that the above linear programming problem has a solution, and $P_w, P_{\alpha}$ the corresponding polynomials\;
	}\uElse{
		We set $\tilde{\delta}=0$\;
		We find $P_w \in \Gamma_{in}$ using Megiddo's Algorithm.\;
		We set $P_\alpha(x) = \frac{1}{\alpha !}(x-x_0)^\alpha$\;
	}
	\KwRet{$\tilde{\delta}, P_w, P_\alpha (\alpha\in\A)$}
	\caption{Find Critical Delta}
\end{algorithm}

\emph{Explanation:} By applying Algorithm \ref{alg:box} and Lemma \ref{lem1} from a previous section, and dividing by $M$, we compute a vector $\tilde{P}_w \in \Gamma$ and a symmetric "box":
\begin{equation}
  \sigma=\{P\in\P: |\xi_l(P)|\leq \lambda_l,\; l=1,\dots,L,\; L\leq D\}
\end{equation}
such that
\begin{equation}
  \tilde{P}_w+M\sigma \subset \Gamma \subset \tilde{P}_w+CM\sigma
\end{equation}
Here, the $\xi_l$ are linear functionals on $\P$, the $\lambda_l$ are non-negative real numbers, and we need not have $\tilde{P}_w \in \Gamma_{in}$.
Next we apply the algorithm "Find Critical Delta in Symmetric Case" to the box $\sigma$, the point $x_0$, the set $\A\subset\M$ and the number $A$. We obtain $\hat{\delta}\in[ 0, \infty ] $ for which the following hold.
\begin{enumerate}
  \item[(I)] There exist $P_\alpha\in\P$ ($\alpha\in\A$) such that
  \begin{enumerate}
    \item[(A)] $\partial^\beta P_\alpha(x_0) = \delta_{\beta\alpha}$ for $\beta,\alpha\in \A$
    \item[(B)]$|\partial^\beta P_\alpha(x_0)|\leq CA\hat{\delta}^{|\alpha|-|\beta|}$ for $\alpha\in\A$, $\beta\in\M$, $\beta\geq\alpha$.
    \item[(C)]$\frac{\hat{\delta}^{m-|\alpha|}P_\alpha}{CA}\in\sigma$ for $\alpha\in\A$.
  \end{enumerate}
  \item[(II)] There do not exist $P_\alpha\in\P$ ($\alpha\in\A$) such that
  \begin{enumerate}
    \item[(A)] $\partial^\beta P_\alpha(x_0) = \delta_{\beta\alpha}$ for $\beta,\alpha\in \A$
    \item[(B)]$|\partial^\beta P_\alpha(x_0)|\leq cA\hat{\delta}^{|\alpha|-|\beta|}$ for $\alpha\in\A$, $\beta\in\M$, $\beta\geq\alpha$.
    \item[(C)]$\frac{\hat{\delta}^{m-|\alpha|}P_\alpha}{cA}\in\sigma$ for $\alpha\in\A$.
  \end{enumerate}
\end{enumerate}

Note that we cannot have $\hat{\delta} = \infty$ because that would contradict the fact that $\sigma$ is bounded. Indeed for any $\delta>0$ there would exist $P_\alpha\in\P$ ($\alpha\in\A\neq\emptyset$) such that $\partial^\beta P_\alpha(x_0) = \delta_{\beta\alpha}$ for $\beta,\alpha\in\A$ and $\delta^{m-|\alpha|}P_\alpha \in CA\sigma$. Therefore, we cannot have $\tilde{\delta} = \infty$.

If $\hat{\delta} \neq 0,\infty$ we compute a point $P_w\in\Gamma_{in}\subset\Gamma$. Letting $P_\alpha$ ($\alpha\in\A$) be as in (I), we note that
\begin{equation}
  \frac{(\tau\hat{\delta})^{m-|\alpha|}P_\alpha}{CA}\in\tau\sigma\;(0<\tau\leq1)
\end{equation}
therefore,
\begin{equation}
  \pm\frac{M(\tau\hat{\delta})^{m-|\alpha|}P_\alpha}{CA}\in\frac{\tau}{2}[(M\sigma+\tilde{P}_w)-(-M\sigma+\tilde{P}_w)]\subset\frac{\tau}{2}(\Gamma-\Gamma)
\end{equation}
and consequently
\begin{equation}
  P_w \pm\frac{M(\tau\hat{\delta})^{m-|\alpha|}P_\alpha}{CA}\in(1+\tau)\blacklozenge\Gamma \; \text{ for }\alpha\in\A
\end{equation}
Also, $\partial^\beta P_\alpha(x_0) = \delta_{\beta\alpha}$ for $\beta,\alpha \in \A$ and $|\partial^\beta P_\alpha(x_0)| \leq CA\hat{\delta}^{|\alpha|-|\beta|}\leq CA(\tau\hat{\delta})^{|\alpha|-|\beta|}$ for $\alpha\in\A$, $\beta\in\M$, $\beta\geq\alpha$.

So, for $\delta = \tau\hat{\delta}$ there exist $P_w\in\Gamma_{in}$ and $P_\alpha\in\P$ $\alpha\in\A$ such that
\begin{align}
  & P_w \pm\frac{M\delta^{m-|\alpha|}P_\alpha}{CA}\in(1+\tau)\blacklozenge\Gamma \; \text{ for }\alpha\in\A\\
  & \partial^\beta P_\alpha(x_0) = \delta_{\beta\alpha}\;(\beta,\alpha\in\A)\\
  & |\partial^\beta P_\alpha(x_0)| \leq C A \delta^{|\alpha|-|\beta|} (\beta\in\M, \alpha\in\A, \beta\geq\alpha)
\end{align}
On the other hand, suppose $0<\delta<\infty$ and suppose there exist $P_w\in\Gamma_{in}$ and $P_\alpha \in \P$ ($\alpha\in\A$) such that:
\begin{align}
  & P_w \pm\frac{M\delta^{m-|\alpha|}P_\alpha}{c_1A}\in(1+\tau)\blacklozenge\Gamma \; \text{ for }\alpha\in\A\\
  & \partial^\beta P_\alpha(x_0) = \delta_{\beta\alpha}\;(\beta,\alpha\in\A)\\
  & |\partial^\beta P_\alpha(x_0)| \leq c_1 A \delta^{|\alpha|-|\beta|} (\beta\in\M, \alpha\in\A, \beta\geq\alpha)
\end{align}
for $c_1$ small enough.

Then,
\begin{equation}
  \frac{2M\delta^{m-|\alpha|}P_\alpha}{c_1A} \in (1+\tau)\blacklozenge\Gamma - (1+\tau)\blacklozenge\Gamma \subset MC'\sigma
\end{equation}
with $C'$ independent of our choice of $c_1$ (and $C'>1$). Therefore $\frac{\delta^{m-|\alpha|}P_\alpha}{(C'c_1)A}\in\sigma$ ($\alpha\in\A$), $\partial^\beta P_\alpha(x_0) = \delta_{\beta\alpha}$ ($\beta,\alpha\in\A$) and $|\partial^\beta P_\alpha(x_0)|\leq (C'c_1)A\delta^{|\alpha|-|\beta|}$ ($\alpha\in\A, \beta\in\M, \beta\geq\alpha$).

Taking $c_1$ small enough, and recalling the defining condition for $\hat{\delta}$ we conclude that $\delta<\hat{\delta}$

Now we produce a list $\delta_\nu$ ($\nu = 1,\dots,\nu_{\max}$) of real numbers starting at $\tau\hat{\delta}$ and ending at $\hat{\delta}$ with, for example $\delta_{\nu+1}\leq2\delta_\nu$, and $\nu_{\max}\leq C\log{\frac{10}{\tau}}$.

For each $\delta_\nu$ we check whether there exist $P_w\in\Gamma_{in}$, $P_\alpha \in \P$ ($\alpha\in\A$) such that
\begin{align}
  & \partial^\beta P_\alpha(x_0) = \delta_{\beta\alpha}&(\beta,\alpha\in\A)\\
  & |\partial^\beta P_\alpha(x_0)| \leq C A \delta_\nu^{|\alpha|-|\beta|} &(\beta\in\M, \alpha\in\A, \beta\geq\alpha)\\
  & P_w \pm\frac{M\delta_\nu^{m-|\alpha|}P_\alpha}{CA}=Q_{\nu\alpha}-\frac{\tau}{2}Q'_{\nu\alpha}+\frac{\tau}{2}Q''_{\nu\alpha} &Q_{\nu\alpha},Q'_{\nu\alpha},Q''_{\nu\alpha}\in\Gamma,\alpha\in\A.
\end{align}
Here, $C$ is the same as in the case $\delta = \tau\hat{\delta}$.

This is a linear program and we can solve it using Megiddo's algorithm. We know such $P_w, P_\alpha$ exist for $\delta_1 = \tau\hat{\delta}$. Let $\tilde{\delta}$ be the largest of the $\delta_\nu$ for which such $P_w,P_\alpha$ exist.

Therefore we have found $P_w\in\Gamma_{in}$, $P_\alpha\in\P$ ($\alpha\in\A$) such that
\begin{align}
  & P_w \pm\frac{M\tilde{\delta}^{m-|\alpha|}P_\alpha}{CA}\in(1+\tau)\blacklozenge\Gamma &\alpha\in\A\\
  & \partial^\beta P_\alpha(x_0) = \delta_{\beta\alpha}&(\beta,\alpha\in\A)\\
  & |\partial^\beta P_\alpha(x_0)| \leq C A \tilde{\delta}^{|\alpha|-|\beta|} &(\beta\in\M, \alpha\in\A, \beta\geq\alpha)
\end{align}

Suppose now there exist $P_w\in\Gamma_{in}$, $P_\alpha\in\P$ ($\alpha\in\A$) such that
\begin{align}
  & P_w \pm\frac{M\tilde{\delta}^{m-|\alpha|}P_\alpha}{c_1 A}\in(1+\tau)\blacklozenge\Gamma &\alpha\in\A\\
  & \partial^\beta P_\alpha(x_0) = \delta_{\beta\alpha}&(\beta,\alpha\in\A)\\
  & |\partial^\beta P_\alpha(x_0)| \leq c_1 A \tilde{\delta}^{|\alpha|-|\beta|} &(\beta\in\M, \alpha\in\A, \beta\geq\alpha)
\end{align}
with $c_1$ small enough, to be picked below. We know in that case $\tilde{\delta} < \hat{\delta} = \delta_{\nu_{\max}}$ and therefore it makes sense to speak of $\delta_{\nu+1}$ where $\tilde{\delta}=\delta_\nu$.
Furthermore we have $\tilde{\delta} < \delta_{\nu+1} \leq 2\tilde{\delta}$.

Therefore, our $P_w \in \Gamma_{in}$ and $P_\alpha\in\P$ ($\alpha\in\A$) satisfy:
\begin{align}
  & P_w \pm\frac{M\delta_{\nu+1}^{m-|\alpha|}P_\alpha}{2^m c_1 A}\in(1+\tau)\blacklozenge\Gamma &\alpha\in\A\\
  & \partial^\beta P_\alpha(x_0) = \delta_{\beta\alpha}&(\beta,\alpha\in\A)\\
  & |\partial^\beta P_\alpha(x_0)| \leq 2^m c_1 A \delta_{\nu+1}^{|\alpha|-|\beta|} &(\beta\in\M, \alpha\in\A, \beta\geq\alpha)
\end{align}
If we pick $c_1$ small enough that $2^m c_1< C$ (same as in the first case) then the above $P_w, P_\alpha$ violate the maximality of the $\delta_\nu$.

Therefore there do not exist $P_w\in\Gamma_{in}$, $P_\alpha\in\P$ ($\alpha\in\A$) such that
\begin{align}
  & P_w \pm\frac{M\tilde{\delta}^{m-|\alpha|}P_\alpha}{c_1 A}\in(1+\tau)\blacklozenge\Gamma &\alpha\in\A\\
  & \partial^\beta P_\alpha(x_0) = \delta_{\beta\alpha}&(\beta,\alpha\in\A)\\
  & |\partial^\beta P_\alpha(x_0)| \leq c_1 A \tilde{\delta}^{|\alpha|-|\beta|}& (\beta\in\M, \alpha\in\A, \beta\geq\alpha).
\end{align}

These conditions are the properties of $\tilde{\delta}$ asserted in Algorithm Find Critical Delta, General Case in the case $\hat{\delta}\in(0,\infty)$.

Suppose $\hat{\delta}=0$ Then for any $\delta>0$ there do not exist $P_\alpha\in\P$ ($\alpha\in\A$) such that:
\begin{align}
  & \partial^\beta P_\alpha(x_0) = \delta_{\beta\alpha}&(\beta,\alpha\in\A)\\
  & |\partial^\beta P_\alpha(x_0)| \leq c A \delta^{|\alpha|-|\beta|} &(\beta\in\M, \alpha\in\A, \beta\geq\alpha)\\
  & \delta^{m-|\alpha|}P_\alpha\in cA\sigma & \alpha\in\A
\end{align}

We set $\tilde{\delta}=0$. We use Megiddo's Algorithm to find $P_w \in \Gamma_{in}$. So (I) is satisfied.

Regarding (II), suppose there exist $0<\delta<\infty$, $P_w\in\Gamma_{in}$, $P_\alpha\in\P$ ($\alpha\in\A$) such that
\begin{align}
  & P_w \pm\frac{M\delta^{m-|\alpha|}P_\alpha}{c_1 A}\in(1+\tau)\blacklozenge\Gamma &\alpha\in\A\\
  & \partial^\beta P_\alpha(x_0) = \delta_{\beta\alpha}&(\beta,\alpha\in\A)\\
  & |\partial^\beta P_\alpha(x_0)| \leq c_1 A \delta^{|\alpha|-|\beta|} &(\beta\in\M, \alpha\in\A, \beta\geq\alpha)
\end{align}
with $c_1$ small enough.

Then,
\begin{equation}
  \frac{2M\delta^{m-|\alpha|}P_\alpha}{c_1A} \in (1+\tau)\blacklozenge\Gamma - (1+\tau)\blacklozenge\Gamma \subset MC'(1+\tau)\sigma
\end{equation}
with $C'$ independent of our choice of $c_1$ (choose $C'>1$). Therefore $\frac{\delta^{m-|\alpha|}P_\alpha}{(C'c_1)A(1+\tau)}\in\sigma$ ($\alpha\in\A$), $\partial^\beta P_\alpha(x_0) = \delta_{\beta\alpha}$ ($\beta,\alpha\in\A$) and $|\partial^\beta P_\alpha(x_0)|\leq (C'c_1)A\delta^{|\alpha|-|\beta|}$ ($\alpha\in\A, \beta\in\M, \beta\geq\alpha$).

If we pick $c_1$ small enough, then we get a contradiction. Therefore (II) holds with $\tilde{\delta}=0$. This settles all cases except $\A=\emptyset$, which we ruled out. This completes the explanation of the Algorithm.
\qed

We will use the above algorithm with:
\begin{align}
  \Gamma_{in}=\{&P\in\Gamma: \partial^\beta(P-P_{\text{given}})\equiv 0 ,\beta\in\A,\\
   &|\partial^\beta(P-P_{\text{given}})(x_\text{given})|\leq M_\text{given}\delta^{m-|\beta|}_\text{given}, \beta\in\M\}
\end{align}
Where $P_\text{given}\in\P$, $M_\text{given}$, $\delta_\text{given}$ are given.

\section{Blobs}

Recall from \cite{feffermanSharpFormWhitney2005} that a family of convex sets $(\Gamma(x, M))_{x\in E,M > 0}$ in a finite dimensional vector space is a shape field if for all $x \in E$ and $0 < M' \leq M \leq \infty$, $\Gamma(x, M)$ is a possibly empty convex set and  $\Gamma(x, M') \subset \Gamma(x, M)$.
 
A family of convex sets $\Gamma(M,\tau)$ in a finite dimensional vector space (possibly empty), parameterized by $M>0$ and $\tau\in(0,\tau_{\max}]$ is a \textbf{blob} with \textbf{blob constant} $C$ if it satisfies:

\begin{itemize}
\item[\refstepcounter{equation}\text{(\theequation)}\label{blob1}]$ (1+\tau)\blacklozenge\G(M,\tau) \subset \G(M',\tau ')\ \text{for}\ M'\geq CM, \frac{\tau_{\max}}{C}\geq \tau ' \geq C\tau$.
\end{itemize}

A blob field with blob constant $C$ is a family of convex sets $\Gamma(x,M,\tau)\subset\P$ parameterized by $x\in E$, $M,\tau$ as above, such that for each $x\in E$, the family $(\Gamma(x,M,\tau))_{\substack{M>0\\ \tau\in(0,\tau_{\max}]}}$ is a blob with blob constant $C$.

\subsection{Specifying a blob field}
\label{sec:specifying}
Recall that $N = \# E$. In order to develop algorithms that compute the jet of an interpolant, we need to explain how to specify a blob field. We will use an Oracle that gives us the needed descriptors of a blob field in $O(N \log N)$ work.

\begin{definition}
	\label{def:oracle}
	A Blob Field is specified by an Oracle $\Omega$. We query $\Omega$ with an $M>0$ and a $\tau < \tau_{\max}$ and, after charging $O(N\log N)$ work, $\Omega$ returns a list $(\Delta(\Gamma(x, M, \tau)))_{x\in E}$ with the descriptors of $\Gamma(x, M, \tau)$ for each $x$. Moreover, the sum of all lengths $|\Delta(\Gamma(x, M, \tau))|$ over all $x \in E$ is assumed to be at most $CN$.
\end{definition}

\begin{remark}
	Without loss of generality, we can assume that for each $x$, the length of the descriptor $\Delta(\Gamma(x,M,\tau))$ is at most $C(\tau)$. We can approximate each of the descriptors using Algorithm \ref{alg:approx-tau} if that was not the case.
\end{remark}

Not every blob field can be specified by an oracle, because $\Gamma(x, M, \tau)$ needn't be a polytope. However, when we perform computations, we will deal only with blob fields that can be specified by an oracle.

\subsection{Operations with blobs and blob fields}

The Minkowski sum of blobs $\vec{\Gamma}=(\Gamma(M,\tau))_{\substack{M>0\\\tau\in(0,\tau_{\max}]}}$ and $\vec{\Gamma}'=(\Gamma'(M,\tau))_{\substack{M>0\\\tau\in(0,\tau_{\max}]}}$ is the family of convex sets
 $(\Gamma(M,\tau)+\Gamma'(M,\tau))_{\substack{M>0\\\tau\in(0,\tau_{\max}]}}$.

One checks easily that the Minkowski sum is again a blob; its blob constant can be taken to be the maximum of the blob constant of $\vec{\Gamma}$ and that of $\vec{\Gamma'}$. Here we use the fact that $(1+\tau)\blacklozenge K + (1+\tau)\blacklozenge K' = (1+\tau)\blacklozenge(K+K')$.

The intersection of blobs $\vec{\Gamma}$ and $\vec{\Gamma}'$ above is given by $(\Gamma(M,\tau)\cap\Gamma'(M,\tau))_{\substack{M>0\\\tau\in(0,\tau_{\max}]}}$.

Again, one checks easily that this is again a blob with blob constant less than or equal to the maximum of the blob constants of $\vec{\Gamma},\vec{\Gamma}'$. Here we use the fact that $(1+\tau)\blacklozenge(K\cap K') \subset(1+\tau)\blacklozenge K  \cap (1+\tau)\blacklozenge K'$ for convex $K, K'$. From now on we write $\vec{\Gamma}+\vec{\Gamma}'$ and $\vec{\Gamma}\cap\vec{\Gamma}'$ to denote the Minkowski sum and intersection.

The same applies for blob fields.

\subsection{C-equivalent blobs}
Two blobs $\vec{\Gamma}$ and $\vec{\Gamma}'$ are called $C-$equivalent if
\begin{align}
  \Gamma(M,\tau)\subset\Gamma'(M',\tau') \\
  \Gamma'(M,\tau)\subset\Gamma(M',\tau')
\end{align}
for $M'\geq CM$ and $\frac{\tau_{\max}}{C}\geq\tau'\geq C\tau$. Similarly for blob fields.

\begin{lemma}
  \label{lem:c-eq-bl}
  Suppose $\vec{\Gamma}$ is a blob with blob constant $C_1$ and suppose $\vec{\Gamma}'$ is a collection of convex sets $\Gamma'(M,\tau)\subset\P$ indexed by $M>0$, $\tau\in(0,\tau_{\max}]$ such that
  \begin{align}
    \Gamma(M,\tau)\subset\Gamma'(M,\tau)\subset (1+\tau)\blacklozenge\Gamma(C_2 M, C_2\tau)
  \end{align}
  for $M>0$, $0 < \tau < \frac{\tau_{\max}}{C_2}$.
  
  Then $\vec{\Gamma}'$ is a blob, with blob constant determined by $C_1,C_2$. Moreover the blobs are $C-$equivalent, with $C$ determined by $C_1$ and $C_2$.
\end{lemma}
\begin{proof}
  Since $\vec{\Gamma}$ is a blob, we know
  \begin{align}
    (1+\tau)\blacklozenge\Gamma(M, \tau)\subset\Gamma(M',\tau')
  \end{align}
  for $M'\geq C_1M$ and $\frac{\tau_{\max}}{C_1} \geq \tau'\geq C_1\tau$. We have
  \begin{align}
    (1+\tau)\blacklozenge\Gamma'(M,\tau) \subset (1+\tau)\blacklozenge[(1+\tau)\blacklozenge\Gamma(C_2M, C_2\tau)]
  \end{align}

  and, applying the blob property twice, $(1+\tau)\blacklozenge\Gamma'(M,\tau) \subset \Gamma(C_2 M'', C_2 \tau'') \subset \Gamma'(C_2 M'', C_2 \tau'')$ for $M'' \geq C_1^2 M$ and $\frac{\tau_{\max}}{C_1^2}\geq \tau'' \geq C_1^2 \tau$.
  Therefore, $\vec{\Gamma}'$ is a blob with blob constant $C_1^2 C_2$.
  The proof also shows the $C_1C_2$-equivalence of both blobs.  
\end{proof}

\subsection{$(C_w, \delta_{\max} )$-convexity}
\label{sec:cwconvexity}

A blob $\vec{\Gamma} = (\Gamma(M,\tau))_{M\geq 0, \tau\in(0,\tau_{\max}]}$ is called $(C_w, \delta_{\max})$-convex at $x\in\R^n$ if the following holds:

Let $0<\delta\leq\delta_{\max}$, $M>0$, $\tau\in(0,\tau_{\max}]$, $P_1,P_2\in\Gamma(M,\tau)$, $Q_1,Q_2\in\P$. Assume
\begin{itemize}
  \item $|\partial^\beta(P_1-P_2)(x)|\leq M\delta^{m-|\beta|}$ for $\beta\in\M$ and
  \item$|\partial^\beta Q_i(x)|\leq\delta^{-|\beta|}$ for $\beta\in\M$ and $i=1,2$.
\end{itemize}
Assume also that $\sum_{i=1}^2 Q_i \odot_x Q_i = 1$. Then $\sum_{i=1}^2Q_i\odot_x Q_i \odot_x P_i\in \Gamma(C_wM, C_w\tau)$.

A blob field $(\Gamma(x,M,\tau))$ is $(C_w,\delta_{\max})$-convex if for each $x\in E$, the blob $(\Gamma(x,M,\tau))$ is $(C_w,\delta_{\max})$-convex at $x$.

\begin{remark}
	The intersection of blobs $(C_w, \delta_{\max})$-convex at $x$ is also a $(C_w, \delta_{\max})$-convex blob at $x$.
\end{remark}

We write $\mathscr{B}(x,\delta) = \{P\in\P: |\partial^\beta P(x)| \leq \delta^{m-|\beta|}\; \text{for }\beta\in\M\}$.

\begin{lemma}[Hopping Lemma]
  \label{hopping-lemma}
  
  Let $\vec{\Gamma}=(\Gamma(M,\tau))$ be a blob with blob constant $C_0$. Assume $\vec{\Gamma}$ is $(C_w,\delta_{\max})$-convex at $y$. Let $\|x-y\|\leq\tilde{\delta}\leq \delta_{\max}$.

  Then $\vec{\Gamma}' =(\Gamma'(M,\tau))_{M,\tau}= (\Gamma(M,\tau) + M\mathscr{B}(x,\tilde{\delta}))_{M,\tau}$ is a blob, and that blob is $(C_w', \delta_{\max})$-convex at $x$, where $C_w'$ depends only on $C_w, C_0, m,n$. The blob constant for $\vg'$ depends only on $C_0, m, n$.
\end{lemma}
\begin{proof}
	
	First, note that $(M\mathscr{B}(x,\tilde{\delta}))$ is a blob if we consider it as a function $(M,\tau) \to M\mathscr{B}(x, \tilde{\delta})$, with blob constant $1+\tau_{\max}$. Therefore $\vec{\Gamma}'$ is a blob. Its blob constant is the maximum between the blob constant $C_0$ and $(1+ \tau_{\max})$.

  Let $0<\delta\leq \delta_{\max}$, $M>0$, $\tau\in(0,\tau_{\max}]$, $P_1',P_2'\in\Gamma'(M,\tau)$, and $Q_1',Q_2'\in\P$. Assume:
  \begin{enumerate}
    \item $|\partial^\beta(P_1'-P_2')(x)|\leq M\delta^{m-|\beta|}$ for $\beta\in\M$.
    \item $|\partial^\beta Q_i'(x)|\leq\delta^{-|\beta|}$ for $\beta \in \M$.
    \item $\sum_{i=1,2}Q_i'\odot_x Q_i' = 1$.
  \end{enumerate}
  We write $P_i' = P_i + MP_{bi}$ where $P_i \in \Gamma(M,\tau)$ and $|\partial^\beta P_{bi}(x)| \leq \tilde{\delta}^{m-\beta}$ for $\beta \in \M$.

  We want to prove there exists a $P\in \Gamma(CM,C\tau)$ such that
  \begin{align}
    |\partial^\beta(\sum_{i=1,1}Q_i'\odot_x Q_i' \odot_x P_i' - P)(x)| \leq CM\tilde{\delta}^{m-|\beta|}\text{ for }\beta \in \M.
  \end{align}

  We define
  \begin{flalign*}
    \theta_i = \frac{Q'_i}{(Q_1^{'2} + Q_2^{'2})^{\frac{1}{2}}}\text{ on } B_n(x, c_0\delta) 
  \end{flalign*}
  for a $c_0 <1$ small enough so that $\theta_i$ is well defined and $|\partial^\beta\theta_i|\leq C\delta^{-|\beta|}$ on $B_n(x, c_0\delta)$. (Note that $\theta_1^2 + \theta_2^2 = 1$ on $B_n(x, c_0\delta)$ and $J_x(\theta_i)=Q_i'$ .)

  We divide the proof in two cases:
  \paragraph{Case 1:} Suppose $\tilde{\delta}\leq c_0\delta$.

  Then
  \begin{flalign*}
    |\partial^\beta(P_1 - P_2)(x)| &\leq |\partial^\beta(P_1 - P'_1)(x)| + |\partial^\beta(P'_1 - P'_2)(x)| + |\partial^\beta(P'_2 - P_2)(x)| &\\
    &\leq M\tilde{\delta}^{m-|\beta|} + M\delta^{m-|\beta|} + M\tilde{\delta}^{m-|\beta|} \leq CM\delta^{m-|\beta|}&
  \end{flalign*}
  for $|\beta|\leq m-1$. Consequently,
  \begin{flalign*}
    |\partial^\beta(P_1-P_2)| \leq CM\delta^{m-\beta}\;\text{on }B_n(x, c_0\delta)\text{ for }|\beta|\leq m
  \end{flalign*}
  In particular,
  \begin{flalign*}
    |\partial^\beta(P_1-P_2)(y)| \leq CM\delta^{m-\beta}
  \end{flalign*}

  Let $Q_i = J_y(\theta_i)$, we know $|\partial^\beta Q_i (y)|\leq C\delta^{-|\beta|}$. We know $\vec{\Gamma}$ is $(C_w,\delta_{\max})$-convex at $y$, therefore
  \begin{flalign*}
    P = J_y(\theta_1^2 P_1 + \theta_2^2P_2) \in \Gamma(CM, C\tau)
  \end{flalign*}
  for $\tau_{\max} \geq C\tau$, where $C$ depends on $C_w, C_0, m, n$. We propose $P$ as a candidate for seeing
  \begin{flalign*}
    |\partial^\beta(\theta_1^2 P'_1 + \theta_2^2 P'_2 - P)(x)|\leq CM\tilde{\delta}^{m-|\beta|}
  \end{flalign*}
  Since $\theta_1^2 + \theta_2^2 = 1$ and $J_y P_1 = P_1$:
  \begin{flalign*}
    (\theta_1^2 P'_1 + \theta_2^2 P'_2 - J_y(\theta_1^2 P_1 + \theta_2^2 P_2)) &= &\theta_1^2(P'_1 - P_1) + \theta_2^2(P'_2-P_2) +\\
    &&+[\theta_1^2 P_1 + \theta_2^2 P_2 - J_y(\theta_1^2 P_1 + \theta_2^2 P_2)]\\
    &= &\theta_1^2(P'_1 - P_1) + \theta_2^2(P'_2-P_2) +\\
    &&+[\theta_1^2 P_1 + \theta_2^2 P_2 - P_1 - J_y(\theta_2^2(P_2-P_1))]\\
    &= &\theta_1^2(P'_1-P_1)+\theta_2^2(P'_2-P_2) +\\
    &&+\theta_2^2(P_2-P_1) - J_y(\theta_2^2(P_2-P_1)).
  \end{flalign*}
  Now, on one hand we know $|\partial^\beta[\theta_i^2(P'_i - P_i)](x)|\leq CM\tilde{\delta}^{m-|\beta|}$ (apply the product rule and properties of $\theta_i$ and $P'_i - P_i$, and remember that $\delta^{-|\beta|}\leq C\tilde{\delta}^{-|\beta|}$).

  On the other hand,
  \begin{flalign*}
    |\partial^\beta[\theta_2^2(P_1-P_2)](z)|\leq CM\delta^{m-|\beta|}\text{ for all }z\in B_n(x,c_0\delta), |\beta|\leq m
  \end{flalign*}
  In particular $|\partial^\beta [\theta_2^2 (P_1-P_2)]|\leq CM$ on $B_n(x,c_0\delta)$ for $|\beta|=m$. Applying Taylor's theorem, we find:
  \begin{flalign*}
    |\partial^\beta[\theta_2^2(P_1-P_2)-J_y(\theta_2^2(P_1-P_2))](x)|\leq CM\|x-y\|^{m-|\beta|}\text{ for }|\beta|\leq m-1
  \end{flalign*}
  which implies by our assumption $\|x-y\|\leq \tilde{\delta}$:
  \begin{flalign*}
    |\partial^\beta[\theta_2^2(P_1-P_2)-J_y(\theta_2^2(P_1-P_2))](x)|\leq CM\tilde{\delta}^{m-|\beta|}\text{ for }|\beta|\leq m-1
  \end{flalign*}
  
  \paragraph{Case 2:} Suppose now $\tilde{\delta}>c_0\delta$.

  Then we have
  \begin{flalign*}
    P' &= Q'_1\odot_x Q'_1 \odot_x P'_1 + Q'_2\odot_x Q'_2\odot_x P'_2\\
    &= P'_1 + Q'_2\odot_x Q'_2\odot_x (P'_2 - P'_1).
  \end{flalign*}
  From our assumptions for $Q'_2$ and $P'_2 - P'_1$ we have
  \begin{flalign*}
    |\partial^\beta (P' - P'_1)(x)| \leq CM\delta^{m-|\beta|}\text{ for }|\beta|\leq m-1
  \end{flalign*}
  and since $\delta \leq C \tilde{\delta}$, we have $|\partial^\beta(P'-P'_1)(x)|\leq CM\tilde{\delta}^{m-|\beta|}$. We know that there exists $P_1 \in \Gamma(M,\tau)$ such that
  \begin{flalign*}
    |\partial^\beta(P'_1-P_1)(x)|\leq M\tilde{\delta}^{m-|\beta|}\text{ for }|\beta|\leq m-1
  \end{flalign*}
  which allows us to conclude that
  \begin{flalign*}
    |\partial^\beta(P'-P_1)(x)|\leq CM\tilde{\delta}^{m-|\beta|}\text{ for }|\beta|\leq m-1.
  \end{flalign*}
  This concludes our proof for Lemma \ref{hopping-lemma}.
\end{proof}

\begin{lemma}
  \label{lem:c-eq-conv}
  Let $\vg = (\Gamma(M,\tau))$ and $\vg' = (\G'(M,\tau))$ be two $C-$equivalent blobs. Assume $\vg$ is $(C_w,\delta_{\max})-$convex at $x$. Then $\vg'$ is $(C'_w,\delta_{\max})-$convex at $x$, where $C'_w$ depends only on $C$ and $C_w$.
\end{lemma}
\begin{proof}
Let $0<\delta\leq\delta_{\max}$, $M>0$, $\tau\in(0,\tau_{\max}]$, $P_1,P_2\in\Gamma'(M,\tau)$, $Q_1,Q_2\in\P$. Assume
\begin{itemize}
  \item $|\partial^\beta(P_1-P_2)(x)|\leq M\delta^{m-|\beta|}$ for $\beta\in\M$ and
  \item$|\partial^\beta Q_i(x)|\leq\delta^{-|\beta|}$ for $\beta\in\M$ and $i=1,2$.
\end{itemize}
Assume also that $\sum_{i=1}^2 Q_i \odot_x Q_i = 1$.

Because $\G',\G$ are $C-$equivalent we know $P_1,P_2 \in \Gamma(M',\tau')$ for $M' \geq CM$ and $\tau' \geq C\tau$. Then because $\G$ is $(C_w,\delta_{\max})$-convex at $x$, we have $\sum_{i=1}^2Q_i\odot_x Q_i \odot_x P_i\in \Gamma(C_wM', C_w\tau')$. Again applying $C-$equivalence, $\sum_{i=1}^2Q_i\odot_x Q_i \odot_x P_i\in \Gamma(C_wC^2M, C_wC^2\tau)$.
\end{proof}

We recover some lemmas from \cite{feffermanFinitenessPrinciplesSmooth2016}. We refer the reader to \cite{feffermanFinitenessPrinciplesSmooth2016} for the proofs, which have to be trivially modified to account for $\tau$.
\begin{lemma}
\label{lemma-wsf1} Suppose $\vec{\Gamma}=\left( \Gamma \left( x,M,\tau\right)
\right) _{x\in E, M>0, \tau\in(0,\tau_{\max}]}$ is a $\left( C_{w},\delta _{\max }\right) $-convex
blob field with blob constant $C_{\G}$. Let

\begin{itemize}
\item[\refstepcounter{equation}\text{(\theequation)}\label{6}] $0<\delta
\leq \delta_{\max}$, $x \in E$, $M>0$, $P_1,P_2,Q_1,Q_2 \in \mathcal{P}$ and 
$A^{\prime },A^{\prime \prime }>0$.
\end{itemize}

Assume that

\begin{itemize}
\item[\refstepcounter{equation}\text{(\theequation)}\label{7}] $P_1,P_2 \in
\Gamma(x,A^{\prime }M, A'\tau)$ with $A'\tau \leq \tau_{\max}$;

\item[\refstepcounter{equation}\text{(\theequation)}\label{8}] $\left\vert
\partial ^{\beta }\left( P_{1}-P_{2}\right) \left( x\right) \right\vert \leq
A^{\prime}M \delta^{  m-\left\vert \beta \right\vert }$ for $\left\vert \beta
\right\vert \leq m-1$;

\item[\refstepcounter{equation}\text{(\theequation)}\label{9}] $\left\vert
\partial ^{\beta }Q_{i}\left( x\right) \right\vert \leq A^{\prime \prime 
}\delta^{-\left\vert \beta \right\vert }$ for $\left\vert \beta \right\vert \leq m-1$
and $i=1,2$;

\item[\refstepcounter{equation}\text{(\theequation)}\label{10}] $Q_{1}\odot
_{x}Q_{1}+Q_{2}\odot _{x}Q_{2}=1$.

\item[\refstepcounter{equation}\text{(\theequation)}\label{ctau}] $C\tau \leq \tau_{\max}$ for a constant $C$ determined by $A^{\prime }$, $A^{\prime \prime }$, $C_{w}$, $C_{\Gamma}$, $m$, and $n$.
\end{itemize}

Then

\begin{itemize}
\item[\refstepcounter{equation}\text{(\theequation)}\label{11}] $%
P:=Q_{1}\odot _{x}Q_{1}\odot _{x}P_{1}+Q_{2}\odot _{x}Q_{2}\odot
_{x}P_{2}\in \Gamma \left( x,CM, C\tau\right) $ with $C$ determined by $A^{\prime }$%
, $A^{\prime \prime }$, $C_{w}$, $C_{\Gamma}$, $m$, and $n$.
\end{itemize}
\end{lemma}

\begin{lemma}
\label{lemma-wsf2} Suppose $\vec{\Gamma}=\left( \Gamma \left( x,M,\tau\right)
\right) _{x\in E, M>0, \tau\in(0,\tau_{\max}]}$ is a $\left( C_{w},\delta _{\max }\right) $-convex
blob field with blob constant $C_{\G}$. Let

\begin{itemize}
\item[\refstepcounter{equation}\text{(\theequation)}\label{12}] $0<\delta
\leq \delta _{\max }$, $x\in E$, $M>0,A^{\prime },A^{\prime \prime }>0$, $%
P_{1},\cdots P_{k},Q_{1},\cdots ,Q_{k}\in \mathcal{P}$. 
\end{itemize}

Assume that

\begin{itemize}
\item[\refstepcounter{equation}\text{(\theequation)}\label{13}] $P_{i}\in
\Gamma \left( x,A^{\prime }M, A'\tau\right) $ for $i=1,\cdots ,k$ ($A'\tau \leq \tau_{\max}$);

\item[\refstepcounter{equation}\text{(\theequation)}\label{14}] $\left\vert
\partial ^{\beta }\left( P_{i}-P_{j}\right) \left( x\right) \right\vert \leq
A^{\prime}M\delta^{  m-\left\vert \beta \right\vert }$ for $\left\vert \beta
\right\vert \leq m-1$, $i,j=1,\cdots ,k$;

\item[\refstepcounter{equation}\text{(\theequation)}\label{15}] $\left\vert
\partial ^{\beta }Q_{i}\left( x\right) \right\vert \leq A^{\prime \prime }\delta^{
-\left\vert \beta \right\vert }$ for $\left\vert \beta \right\vert \leq m-1$
and $i=1,\cdots ,k $;

\item[\refstepcounter{equation}\text{(\theequation)}\label{16}] $%
\sum_{i=1}^{k}Q_{i}\odot _{x}Q_{i}=1$.

\item[\refstepcounter{equation}\text{(\theequation)}\label{ctau2}] $C\tau \leq \tau_{\max}$ for a constant $C$ determined by $A^{\prime }$, $A^{\prime \prime }$, $C_{w}$, $C_{\Gamma}$, $m$, $n$ and $k$.

\end{itemize}

Then

\begin{itemize}
\item[\refstepcounter{equation}\text{(\theequation)}\label{17}] $%
\sum_{i=1}^kQ_{i}\odot _{x}Q_{i}\odot _{x}P_{i}\in \Gamma \left( x,CM,C\tau\right)
, $ with $C$ determined by $A^{\prime }$, $A^{\prime \prime }$, $C_{w}$, $C_{\Gamma}$, $m$%
, $n$, $k$.
\end{itemize}
\end{lemma}

\section{Refinements}
\label{sec:ref}
Say $\vec{\vec{\Gamma}} = (\Gamma(x, M, \tau))_{\substack{x\in E\\ M>0 \\ \tau\in(0,\tau_{\max}]}}$ is a blob field, $\#(E) = N$. We define a new blob field called the first refinement of $\vec{\vec{\Gamma}}$. To do so, we imitate (\cite{feffermanFittingSmoothFunction2009}). We use a Well Separated Pairs Decomposition $E\times E-\text{Diag} = \cup_{1\leq \nu\leq \nu_{\max}}E'_{\nu}\times E''_{\nu}$ with $\nu_{\max}\leq CN$. Additionally each $E'_{\nu}$ has the form $E'_{\nu}=E\cap Q'_{\nu}$ and each $E''_{\nu} = E\cap Q''_{\nu}$ where $Q'_{\nu},Q''_{\nu}$ are boxes. See \cite{feffermanFittingSmoothFunction2009} for more details.

  Moreover, each $E'_{\nu}$ and each $E''_{\nu}$ may be decomposed as a disjoint union of at most $C\log N$ dyadic intervals $I'_{\nu i}$ ($i=1,\dots,i'_{\max}(\nu)$) and $I''_{\nu i}$ ($i=1, \dots, i''_{\max}(\nu)$) in $E$, respectively, with respect to an order relation on $E$. We say that the $I'_{\nu i}$ appear in $E'_{\nu}$ and that the $I''_{\nu i}$ appear in $E''_{\nu}$.
  
For a subset $A\subset \R^n$ we define

\begin{equation*} 
\text{diam}_\infty (A) = \sqrt{n}\sup_{(x_1, \dots, x_n),(y_1, \dots, y_n)\in A}\max_{1\leq i \leq n} |x_i - y_i|,
\end{equation*}
the $l_\infty$ diameter of $A$.
\begin{enumerate}
\item[Step 1:] For each dyadic interval $I$ in $E$ we fix a representative $x_I \in I$ and define:
  \begin{flalign*}
    \Gamma_{\text{step 1}}(I, M, \tau) = \cap_{y\in I}[\Gamma(y,M,\tau)+M\mathscr{B}(x_I, \text{diam}_{\infty} I)]
  \end{flalign*}

\item[Step 2:] For each $E''_{\nu}$ define a representative $x''_{\nu} \in E''_{\nu}$ and define:
  \begin{flalign*}
    \Gamma_{\text{step 2}}(E''_{\nu}, M, \tau) = \cap_{I\text{ appears in }E''_{\nu}}[\Gamma_{\text{step 1}}(I, M, \tau) + M\mathscr{B}(x''_{\nu}, \text{diam}_{\infty} E''_{\nu})]
  \end{flalign*}
\item[Step 3:] For each $E'_{\nu}$ we fix a representative $x'_{\nu}$ and define:
  \begin{flalign*}
    \Gamma_{\text{step 3}}(E'_{\nu}, M, \tau) = \Gamma_{\text{step 2}}(E''_{\nu}, M, \tau) + M\mathscr{B}(x'_{\nu}, \|x'_{\nu}-x''_{\nu}\|)
  \end{flalign*}
\item[Step 4:] For each dyadic interval $I$, define:
  \begin{flalign*}
    \Gamma_{\text{step 4}}(I, M, \tau) = \cap_{\substack{\text{all }E'_{\nu} s.t. \\ I\text{ appears in }E'_{\nu}}} \Gamma_{\text{step 3}}(E'_{\nu}, M, \tau)
  \end{flalign*}
\item[Step 5:] For each $x\in E$, define:
  \begin{flalign*}
    \Gamma_{\text{step 5}}(x, M, \tau) = [\cap_{I\ni x} \Gamma_{\text{step 4}}(I,M,\tau)]\cap \Gamma(x,M,\tau)
  \end{flalign*}
\end{enumerate}

All of these are blobs, with blob constants controlled by the blob constant of $\vec{\vec{\Gamma}}$. 
\begin{lemma}
	\label{lem:cw-convex-ref}
If $\vec{\vec{\Gamma}}$ is $(C_w,\delta_{\max})$-convex, then:
\begin{enumerate}
  \item[(I)] $(\Gamma_{\text{step 1}}(I,M,\tau))_{\substack{M>0\\ \tau\in(0,\tau_{\max}]}}$ is $(C',\delta_{\max})$-convex at $x_I$.
  \item[(II)] $(\Gamma_{\text{step 2}}(E''_{\nu},M,\tau))_{\substack{M>0\\ \tau\in(0,\tau_{\max}]}}$ is $(C',\delta_{\max})$-convex at $x''_{\nu}$ (by Lemma \ref{hopping-lemma} and intersection properties).
  \item[(III)] $(\Gamma_{\text{step 3}}(E'_{\nu},M,\tau))_{\substack{M>0\\ \tau\in(0,\tau_{\max}]}}$ is $(C',\delta_{\max})$-convex at any point of $E'_{\nu}$ .
  \item[(IV)] $(\Gamma_{\text{step 4}}(I,M,\tau))_{\substack{M>0\\ \tau\in(0,\tau_{\max}]}}$ is $(C',\delta_{\max})$-convex at any point of $I\subset E$.
  \item[(V)] $(\Gamma_{\text{step 5}}(x,M,\tau))_{\substack{M>0\\ \tau\in(0,\tau_{\max}]}}$ is $(C',\delta_{\max})$-convex at $x$ .
\end{enumerate}
\end{lemma}
    \begin{proof}
    	\begin{enumerate}
    		\item[(I)] By Lemma \ref{hopping-lemma} and intersection properties.
    		\item[(II)] By Lemma \ref{hopping-lemma} and intersection properties.
    		\item[(III)]       We proceed as in Lemma \ref{hopping-lemma}. Obviously $\Gamma_{\text{step 3}}$ is $(C_w,\delta_{\max})-$convex at $x_{\nu}'$ (by Lemma \ref{hopping-lemma}), but we need to prove it for every $x \in E'_\nu$.
    		Let $0<\delta\leq \delta_{\max}$, $M>0$, $\tau\in(0,\tau_{\max}]$, $P_1',P_2'\in\Gamma_{\text{step 3}}(E'_{\nu},M,\tau)$, and $Q_1',Q_2'\in\P$. Let $x\in E'_{\nu}$. Assume:
    		\begin{enumerate}
    			\item $|\partial^\beta(P_1'-P_2')(x)|\leq M\delta^{m-|\beta|}$ for $\beta\in\M$.
    			\item $|\partial^\beta Q_i'(x)|\leq\delta^{-|\beta|}$.
    			\item $\sum_{i=1,2}Q_i'\odot_x Q_i' = 1$.
    		\end{enumerate}
    		We write $P_i' = P_i + MP_{bi}$ where $P_i \in \Gamma_{\text{step 2}}(E''_{\nu},M,\tau)$ and $|\partial^\beta P_{bi}(x_{\nu}')| \leq \|x_{\nu}' - x_{\nu}''\|^{m-\beta}$ for $\beta \in \M$.
    		
    		We want to prove there exists a $P\in \Gamma_{\text{step 2}}(E''_{\nu},CM,C\tau)$ such that
    		\begin{align}
    		\label{eq:cwmaxref}
    		|\partial^\beta(\sum_{i=1,1}Q_i'\odot_x Q_i' \odot_x P_i' - P)(x_{\nu}')| \leq CM\|x'_{\nu}-x''_{\nu}\|^{m-|\beta|}\text{ for }\beta \in \M
    		\end{align}
    		
    		We proceed exactly as in Lemma \ref{hopping-lemma} and divide in two cases $\|x-x'_\nu\| \leq c_0\delta$ or $\|x-x'_\nu\| > c_0\delta$. In both cases, proceeding as in Lemma \ref{hopping-lemma}, we would arrive at the inequality we want to see except we would have $\textit{left hand side of \eqref{eq:cwmaxref}} \leq CM \|x-x'_\nu\|^{m-|\beta|}$ . By the Well Separated Pairs Composition, $\|x-x'_\nu\| \leq \kappa\|x'_\nu - x''_\nu \|$ for all $x \in E'_\nu$. Therefore, \eqref{eq:cwmaxref} follows from the analogous inequality with $x'_\nu$ replaced by $x$. That is how we would prove the $(C_w, \delta_{\max})$-convexity at every point in $E'_\nu$.
    		\item[(IV)] By intersection properties.
    		\item[(V)] By intersection properties.
    	\end{enumerate}

  This concludes our proof.

    \end{proof}

Given $P_5\in\Gamma_{\text{step 5}}(x,M,\tau)$ and $I'\ni x$, we have $P_5 \in \Gamma_{\text{step 4}}(I', M,\tau)$. Given any $E'_{\nu}\ni x$, we have $I'\ni x$ for some $I'$ appearing in $E'_{\nu}$, hence $P_5 \in \Gamma_{\text{step 3}}(E'_{\nu}, M, \tau)$. We then have some $P_2 \in \Gamma_{\text{step 2}}(E''_{\nu}, M, \tau)$ such that
\begin{align}
  |\partial^{\beta}(P_2-P_5)(x_{\nu}')|\leq M\|x'_{\nu}-x''_{\nu}\|^{m-|\beta|}\text{ for }\beta\in\M
\end{align}

Since $\|x - x_{\nu}'\| \leq \kappa \|x_{\nu}'' - x_\nu'\|$ we have

\begin{align}
  |\partial^{\beta}(P_2-P_5)(x)|\leq CM\|x'_{\nu}-x''_{\nu}\|^{m-|\beta|}\text{ for }\beta\in\M
\end{align}
Given $I''$ appearing in $E''_{\nu}$ there exists $P_1 \in \Gamma_{\text{step 1}}(I'', M, \tau)$ such that:
\begin{align}
  |\partial^\beta(P_1-P_2)(x_\nu'')|\leq M(\text{diam}_{\infty} E''_{\nu})^{m-|\beta|}\text{ for }\beta\in\M
\end{align}
and because $\text{diam}_{\infty} E''_{\nu} \leq C\|x'_\nu - x''_{\nu}\|$ (with $C$ depending only on $n, m, \kappa$), we can substitute $x_\nu''$ with $x_\nu'$ and then $x_\nu'$ with $x$, so we have

 \begin{align}
 |\partial^\beta(P_1-P_2)(x)|\leq CM(\|x'_\nu - x''_\nu\|)^{m-|\beta|}\text{ for }\beta\in\M.
 \end{align}

Finally, given $y\in I''$ there exists $P\in\Gamma(y,M,\tau)$ such that $|\partial^{\beta}(P-P_1)(x_{I''})|\leq CM(\text{diam}_\infty I'')^{m-|\beta|}\text{ for }\beta\in\M$, and we can repeat the previous substitutions. Moreover, every $y\in E''_{\nu}$ belongs to some $I''$ appearing on $E''_{\nu}$.

Therefore, given $(x,y)\in E\times E-\text{Diag}$ , and given $P_5\in\Gamma_{\text{step 5}}(x,M,\tau)$ there exists $P\in \Gamma(y,M,\tau)$ such that:
\begin{align}
  |\partial^\beta(P_5-P)(x)|\leq CM\|x'_\nu-x''_\nu\|^{m-|\beta|}\text{ for }\beta\in \M
\end{align}
where $x'_\nu, x''_\nu$ are the representatives of the $E'_\nu, E''_\nu$ that correspond to $(x, y)$. Since $c\|x-y\| \leq \|x'_\nu - x''_\nu\| \leq C\|x-y\|$, we finally have
\begin{align}
|\partial^\beta(P_5-P)(x)|\leq CM\|x-y\|^{m-|\beta|}\text{ for }\beta\in \M
\end{align}
which corresponds to the refinements defined in \cite{feffermanSharpFormWhitney2005}.

If $x=y$, we can just take $P=P_5$.

Next, let $F\in\mathcal{C}_{\text{loc}}^m(\R^n)$ such that $|\partial^\beta F|\leq cM$ on $\R^n$ for all $|\beta|=m$ and $J_x(F) \in \Gamma(x,M,\tau)$ for all $x\in E$. Then:
\begin{flalign}
  J_{x_I}(F) &\in \Gamma_{\text{step 1}}(I, M, \tau) &\text{ for all }I\\
  J_{x''_{\nu}}(F) &\in \Gamma_{\text{step 2}}(E''_{\nu}, M, \tau) &\text{ for all }\nu \\
  J_{x}(F) &\in \Gamma_{\text{step 3}}(E'_{\nu}, M, \tau) &\text{ for all }x\in E'_{\nu},\text{ any }\nu \\
  J_{x}(F) &\in \Gamma_{\text{step 4}}(I, M, \tau) &\text{ for all }x\in I,\text{ any }I \\
  J_{x}(F) &\in \Gamma_{\text{step 5}}(x, M, \tau) &\text{ for all }x\in E .
\end{flalign}

We define $\vec{\vec{\Gamma_1}} = (\Gamma_{\text{step 5}}(x, M,\tau))_{\substack{x\in E\\ M>0\\ \tau\in(0,\tau_{\max}]}}$ to be the first refinement of $\vec{\vec{\Gamma}}$ . The above discussion shows that:
\begin{itemize}
\item[\LA{br0}] $\vec{\vg}_1$ is a blob field with blob constant determined by that of $\vec{\vg}$, together with $m, n$ and $\tau_{\max}$.
\item[\LA{br1}] If $\vec{\vec{\Gamma}}$ is $(C_w, \delta_{\max})$-convex, then $\vec{\vec{\Gamma_1}}$ is $(C',\delta_{\max})$-convex, with $C'$ determined by $C_w, m, n$ and the blob constant for $\vec{\vg}$.
\item[\LA{br2}] Given $P\in\Gamma_{\text{step 5}}(x, M, \tau)$ and given $y\in E$, there exists $P' \in \Gamma(y,M,\tau)$ such that $|\partial^\beta(P-P')(x)|\leq CM\|x-y\|^{m-|\beta|}$ for $\beta\in\M$. Note that for $y=x$ the result also is true since $\Gamma_{\text{step 5}}(y, M, \tau) \subset \Gamma(y, M, \tau)$.
\item[\LA{br3}] If $F\in\mathcal{C}^m_{\text{loc}}(\R^n)$ satisfies $|\partial^\beta F|\leq cM$ on $\R^n$ for $|\beta|=m$ and $J_x (F)\in\Gamma(x,M,\tau)$ for all $x\in E$, then also $J_x(F)\in\Gamma_{\text{step 5}}(x,M,\tau)$ for all $x\in E$.
\end{itemize}

Now we define the $l^{th}$ refinement of $\vec{\vg}$ by recursion: $\vec{\vg}_0 = \vec{\vg}$, $\vec{\vg}_{l+1} = \vec{\vg}_{l, \text{step 5}}$. 

\subsection*{Computing the blobs}
\label{sec:computing}

Suppose that our initial blob field $\vec{\vg} = \vvg_0$ is given by an oracle $\Omega$ as in Section \ref{sec:specifying}.

We won't compute $\vec{\vec{\Gamma_1}}$; instead, we compute a $C-$equivalent approximation, using the following algorithms.
\begin{itemize}
\item Approximate Minkowski Sum. See \textbf{Algorithm 6}, and note that the approximate sum for each $M, \tau$ is contained in a $\Gamma(x, CM, C\tau)$ by the definition of blobs.
\item Approximate Intersection: For each $M,\tau$ we concatenate the descriptors for all convex sets if all of them are nonempty, run the Megiddo Algorithm to know if the intersection is non-empty, and then apply \textbf{Algorithm 3}. 
\end{itemize}
We note that these computations will give convex sets that are contained in a blob $\Gamma(x, CM, C\tau)$ for a constant $C$ depending only on $n,m$. This means that the properties explained in section \ref{sec:ref} still hold true, except that in each refinement we replace $M$ and $\tau$ by $CM, C\tau$ respectively. This will determine our initial choice for $\tau$ so that $C^{l}\tau < \tau_{\max}\text{.}$ 

More precisely, let $\tg_0$ be a blob field specified by an Oracle which is known to be $C-$equivalent (for $C$ depending only on $n,m$) to $\G_0$ and for each $x \in E$ let $\tg_l(x, M, \tau)$ be the $l-$th refinement using the approximate Minkowski sum and approximate intersection algorithms. Then we know that $\G_l(x, M, \tau) \subset \tg_l(x,M,\tau) \subset (1+\tau)\blacklozenge\G_l(x, CM, C\tau)$ where $C$ depends on the blob constant of $\vvg_0$, together with $l, m, n, \tau_{\max}$. By Lemma \ref{lem:c-eq-bl} they are $C-$equivalent for some $C$ depending on the blob constant of $\vvg$. Therefore, by Lemma \ref{lem:c-eq-conv} $\tg_l(x, M, \tau)$ have the same $(C_w,\delta_{\max})-$convexity properties as $\Gamma_l (x, M, \tau)$.
The above discussion shows that:
\begin{itemize}
\item[\LA{abr0}] If $\tg_0$ is a blob field with blob constant $C$, then $\tg_l$ is also a blob field with blob constant $C'$ depending only on $l, C_\G, m, n, \tau_{\max}$.
\item[\LA{abr1}] If $\tg_0$ is $(C_w,\delta_{\max})$ convex, then $\tg_l$ is $(C',\delta_{\max})$ convex, with $C'$ depending only on $l, C_w, m, n, C_\G, \tau_{\max}$.
\item[\LA{abr2}] Given $P \in \tg_l(x,M,\tau)$ and given $y \in E$, there exists $P'\in \G_{l-1}(y, CM, C\tau)$ such that $|\partial^{\beta}(P-P')(x)|\leq C'M\|x-y\|^{m-|\beta|}$ for $\beta \in \M$. $C, C'$ depend only on $C_\G, m, n, C_w, \tau_{\max}, l$.
\item[\LA{abr3}] If $F\in\mathcal{C}^m_{\text{loc}}(\R^n)$ satisfies $|\partial^\beta F|\leq cM$ on $\R^n$ for $c$ depending on $n, m, C_\G, l, \tau_{\max}$ and for $|\beta|=m$ and $J_x (F)\in\Gamma(x,M,\tau)$ for all $x\in E$, then also $J_x(F)\in\tg_l(x,CM,C\tau)$ for all $x\in E$.
\end{itemize}

\begin{remark}
  Note that the only difference is between \eqref{br2}, \eqref{abr2}, \eqref{abr3}. The other properties (\eqref{abr0}, \eqref{abr1}) are conserved and only the size of constants $C, C'$ changes (but they still do not depend on $N$ or $\tau$ or $M$).
\end{remark}
\begin{remark}
  All of the proofs from \cite{feffermanFinitenessPrinciplesSmooth2016} will work with our $\vvg$ (just a few minor changes are needed but the proof remains the same). For that reason, the rest of the document will focus on $\tg$. Furthermore, even for $\tg$ the proofs remain the same until Lemma \ref{lemma-transport} of section \ref{transport-lemma}. 
\end{remark}

Recall from \cite{feffermanFittingSmoothFunction2009} and previous sections in this paper that up until now we don't need more than $C(\tau)N\log N$ operations to call the Blob oracle and to create the refinements. Indeed, calling the original blob oracle that returns the whole blob field for a given $M, \tau$ costs $C(\tau)N\log N$, while the approximate Minkowski sum of two convex sets $K(\Delta), K'(\Delta')$ takes $C(\tau)[|\Delta|+|\Delta'|]$ operations but $|\Delta| \leq C(\tau, m)$. The work used to compute the intersection of $k$ convex sets is $kC(\tau,D)$ for the same reasons. Therefore the amount of work in step 1 and 5 is bounded by
\begin{equation}
  \label{eq:1}
  \sum_{I\text{ dyadic interval}} |I| < C(\tau)N\log N
\end{equation}
by (7) from Section 5 in \cite{feffermanFittingSmoothFunction2009}. Step 2 requires no more than $C(\tau)N\log N$ operations, step 3 takes $C(\tau)N$ operations, and step 4 takes $C(\tau)N\log N$ operations. In total, the number of operations is no more than $C(\tau)N\log N$ and the storage is bounded by $C'(\tau)N$. A new Oracle is therefore produced that for a given $M, \tau$ returns all the first refinements in $C(\tau)N\log N)$.

The main theorem of this paper is

\begin{theorem}
	\label{bigTh}For a large enough $l_{\ast }=l_{\ast }\left( m,n\right) $, the
	following holds. Let $\tg _{0}=\left( \tg_{0}\left( x,M, \tau\right)
	\right) _{x\in E,M>0,\tau\in(0,\tau_{\max}]}$ be a $\left( C_{w},\delta _{\max }\right) $-convex
	blob field with blob constant $C_{\Gamma}$, and for $l\geq 1$, let $\tg _{l}=\left( \tg_{l}\left(
	x,M, \tau \right) \right) _{x\in E,M>0, \tau\in(0,\tau_{\max}]}$ be its approximate $l^{th}$ refinement. Suppose we are
	given a cube $Q_{\max }$ of sidelength $\delta _{\max }$, a point $x_{0}\in
	E\cap Q_{\max }$, a number $M_{0}>0$, and a polynomial $P_{0}\in \tg
	_{l_{\ast }}\left( x_{0},M_{0},\tau_{0}\right) $. Then there exists $F\in C^{m}\left( 
	\mathbb{R}^{n}\right) $ such that 
	
	\begin{itemize}
		\item[\LA{mt1}]$J_{x}(F)\in \tg_{0}(x,C_{\ast }M_{0}, C_{\ast} \tau_0)$ for all $x\in Q_{\max
		}\cap E$, and
		
		\item[\LA{mt2}] $|\partial ^{\beta }(F-P_{0})(x)|\leq C_{\ast }M_{0}\delta _{\max
		}^{m-\left\vert \beta \right\vert }$ for all $x\in Q_{\max }$, $|\beta| \leq
		m$.
	\end{itemize}
	
	Here, $C_{\ast }$ depends only on $m$, $n$, $C_{w}$, $C_{\Gamma}$.
\end{theorem}

In this paper, we also implement an algorithm to compute for such an $F$, the jet $J_x(F)$ efficiently at each point $x\in E$.

\section{Polynomial bases}

We adapt some definitions from \cite{feffermanFinitenessPrinciplesSmooth2016}. Let $\vg = (\G(x, M,\tau))_{x\in E, M>0, \tau\in (0,\tau_{\max}]}$ be a blob field with blob constant $C$. Let $x_0 \in E$, $M_0 > 0$, $0 < \tau_0 \leq \frac{\tau_{\max}}{C}$, $P^0\in\P$, $\A\subset\M$, $P_{\alpha}\in\P$ for $\alpha\in\A$, $C_B >0$, $\delta>0$ be given. Then we say that $(P_{\alpha})_{\alpha\in\A}$ forms an $(\A, \delta, C_B)$-basis for $\vg
$ at $(x_0, M_0, \tau_0, P^0)$ if the following conditions are satisfied:
\begin{itemize}
\item[\refstepcounter{equation}\text{(\theequation)}\label{pb1}] $P^{0}\in
\Gamma \left( x_{0},C_{B}M_{0}, C_{B}\tau_0\right) $.
\end{itemize}

\begin{itemize}
\item[\refstepcounter{equation}\text{(\theequation)}\label{pb2}] $P^{0}+%
\frac{M_{0}\delta ^{m-\left\vert \alpha \right\vert }}{C_{B}}P_{\alpha }$, $%
P^{0}-\frac{M_{0}\delta ^{m-\left\vert \alpha \right\vert }}{C_{B}}P_{\alpha
}\in \Gamma \left( x_{0},C_{B}M_{0},C_{B}\tau_0\right) $ for all $\alpha \in \mathcal{A}$%
.
\end{itemize}

\begin{itemize}
\item[\refstepcounter{equation}\text{(\theequation)}\label{pb3}] $\partial
^{\beta }P_{\alpha }\left( x_{0}\right) =\delta _{\alpha \beta }$ (Kronecker
delta) for $\beta ,\alpha \in \mathcal{A}$.
\end{itemize}

\begin{itemize}
\item[\refstepcounter{equation}\text{(\theequation)}\label{pb4}] $\left\vert
\partial ^{\beta }P_{\alpha }\left( x_{0}\right) \right\vert \leq
C_{B}\delta ^{\left\vert \alpha \right\vert -\left\vert \beta \right\vert }$
for all $\alpha \in \mathcal{A}$, $\beta \in \mathcal{M}$.
\end{itemize}

We say that $(P_{\alpha})_{\alpha\in\A}$ forms a weak $(\A, \delta, C_B)$-basis for $\vg$ at $(x_0, M_0, \tau_0, P^0)$ if conditions \eqref{pb1}, \eqref{pb2}, \eqref{pb3} hold as stated and condition \eqref{pb4} holds for $\alpha\in\A,\beta\in\M,\beta\geq\alpha$.

We make a few obvious remarks.

\begin{itemize}
\item[\refstepcounter{equation}\text{(\theequation)}\label{pb5}] Any $(%
\mathcal{A}, \delta, C_B)$-basis for $\vec{\Gamma}$ at $(x_0,M_0,\tau_0,P^0)$ is
also an $(\mathcal{A}, \delta, C_B^{\prime })$-basis for $\vec{\Gamma}$ at $%
(x_0,M_0,\tau_0,P^0)$, whenever $C^{\prime }_B \geq C_B$.
\end{itemize}

\begin{itemize}
\item[\refstepcounter{equation}\text{(\theequation)}\label{pb6}] Any $(%
\mathcal{A}, \delta, C_B)$-basis for $\vec{\Gamma}$ at $(x_0,M_0,\tau_0,P^0)$ is
also an $(\mathcal{A}, \delta^{\prime }, C_B\cdot[\max\{\frac{\delta^{\prime
}}{\delta},\frac{\delta}{\delta^{\prime }} \}]^m)$-basis for $\vec{\Gamma}$
at $(x_0,M_0,\tau_0,P^0)$, for any $\delta^{\prime }>0$.
\end{itemize}

\begin{itemize}
\item[\refstepcounter{equation}\text{(\theequation)}\label{pb7}] {Any weak $(%
\mathcal{A}, \delta, C_B)$-basis for $\vec{\Gamma}$ at $(x_0,M_0,\tau_0,P^0)$ is
also a weak $(\mathcal{A}, \delta^{\prime }, C_B^{\prime })$-basis for $\vec{%
\Gamma}$ at $(x_0,M_0,\tau_0,P^0)$, whenever $0<\delta^{\prime }\leq \delta$ and $%
C^{\prime }_B \geq C_B$. }
\end{itemize}

Note that \eqref{pb1} need not follow from \eqref{pb2}, since $\mathcal{A}$
may be empty.

\begin{itemize}
\item[\refstepcounter{equation}\text{(\theequation)}\label{pb7a}] If $%
\mathcal{A}=\emptyset$, then the existence of an $(\mathcal{A},\delta,C_B)$%
-basis for $\vec{\Gamma}$ at $(x_0,M_0,\tau_0,P^0)$ is equivalent to the assertion
that $P^0 \in \Gamma(x_0, C_BM_0, C_B\tau_0)$.
\end{itemize}

The main result of this section is Lemma \ref{lemma-pb2}. 

\begin{lemma}[Relabeling Lemma]
\label{lemma-pb2} Let $\vg = \blobdef$ be a $(C_w,\delta_{\max})$-convex blob field with blob constant $C_\G$. Let $x_{0}\in E
$, $M_{0}>0$, $0<\tau_{0}\leq\tau_{\max}$, $0<\delta \leq \delta _{\max }$, $C_{B}>0$, $P^{0}\in \Gamma
\left( x_{0},M_{0},\tau_{0}\right) $, $\mathcal{A}\subseteq \mathcal{M}$. Suppose $%
\left( P_{\alpha }^{00}\right) _{\alpha \in \mathcal{A}}$ is a weak $\left( 
\mathcal{A},\delta ,C_{B}\right) $-basis for $\vg$ at $\left(
x_{0},M_{0},\tau_{0},P^{o}\right) $. Then, for some monotonic $\hat{\mathcal{A}}\leq 
\mathcal{A}$, $\vec{\Gamma}$ has an $(\hat{\mathcal{A}},\delta
,C_{B}^{\prime })$-basis at $(x_{0},M_{0},\tau_{0},P^{0})$, with $C_{B}^{\prime }$
determined by $C_{B}$, $C_{w}$, $C_\G$, $m$, $n$. Moreover, if $\max_{\alpha \in 
\mathcal{A},\beta \in \mathcal{M}}\delta ^{|\beta |-|\alpha |}|\partial
^{\beta }P_{\alpha }^{00}(x_{0})|$ exceeds a large enough constant
determined by $C_{B}$, $C_{w}$, $m$, $n$, then we can take $\hat{\mathcal{A}}%
<\mathcal{A}$ (strict inequality).
\end{lemma}
\begin{proof}
  To prove Lemma \ref{lemma-pb2} we proceed as in \cite{feffermanFinitenessPrinciplesSmooth2016}, with trivial changes in the proof and statement of the required technical lemmas.
\end{proof}

The next result is a consequence of the Relabeling Lemma (Lemma \ref%
{lemma-pb2}).

\begin{lemma}[Control $\Gamma $ Using Basis]
\label{lemma-pb3} Let $\vg=\blobdef$ be a $\left( C_{w},\delta _{\max }\right) $-convex blob
field with blob constant $C_{\Gamma}$. Let $x_{0}\in E$, $M_{0}>0$, $0<\tau_0\leq\tau_{\max}$, $0<\delta \leq \delta _{\max }$, $C_{B}>0
$, $\mathcal{A\subseteq M}$, and let $P$, $P^{0}\in \mathcal{P}$. Suppose $\vec{%
\Gamma}$ has an $\left( \mathcal{A},\delta ,C_{B}\right) $-basis at $\left(
x_{0},M_{0},\tau_0, P^{0}\right) $. Suppose also that

\begin{itemize}
\item[\refstepcounter{equation}\text{(\theequation)}\label{pb71}] $P\in
\Gamma \left( x_{0},C_{B}M_{0},C_{B}\tau_0\right) $,
\end{itemize}

\begin{itemize}
\item[\refstepcounter{equation}\text{(\theequation)}\label{pb72}] $\partial
^{\beta }\left( P-P^{0}\right) \left( x_{0}\right) =0$ for all $\beta \in 
\mathcal{A}$, and
\end{itemize}

\begin{itemize}
\item[\refstepcounter{equation}\text{(\theequation)}\label{pb73}] $%
\max_{\beta \in \mathcal{M}}\delta ^{\left\vert \beta \right\vert
}\left\vert \partial ^{\beta }\left( P-P^{0}\right) \left( x_{0}\right)
\right\vert \geq M_{0}\delta ^{m}$.
\end{itemize}

Then there exist $\hat{\mathcal{A}} \subseteq \mathcal{M}$ and $\hat{P}^0
\in \mathcal{P}$ with the following properties.

\begin{itemize}
\item[\refstepcounter{equation}\text{(\theequation)}\label{pb74}] $\hat{%
\mathcal{A}}$ is monotonic.
\end{itemize}

\begin{itemize}
\item[\refstepcounter{equation}\text{(\theequation)}\label{pb75}] $\hat{%
\mathcal{A}} < \mathcal{A}$ (strict inequality).
\end{itemize}

\begin{itemize}
\item[\refstepcounter{equation}\text{(\theequation)}\label{pb76}] $\vec{%
\Gamma}$ has an $(\hat{\mathcal{A}},\delta,C_B^{\prime })$-basis at $%
(x_0,M_0,\tau_0, \hat{P}^0)$, with $C_B^{\prime }$ determined by $C_B$, $C_\G$ $C_w$, $m$, $%
n$.
\end{itemize}

\begin{itemize}
\item[\refstepcounter{equation}\text{(\theequation)}\label{pb77}] $\partial
^{\beta }\left( \hat{P}^{0}-P^{0}\right) \left( x_{0}\right) =0$ for all $%
\beta \in \mathcal{A}$.
\end{itemize}

\begin{itemize}
\item[\refstepcounter{equation}\text{(\theequation)}\label{pb78}] $%
\left\vert \partial ^{\beta }\left( \hat{P}^{0}-P^{0}\right) \left(
x_{0}\right) \right\vert \leq M_{0}\delta ^{m-\left\vert \beta \right\vert }$
for all $\beta \in \mathcal{M}$.
\end{itemize}
\end{lemma}

\begin{proof}
  The proof of Lemma \ref{lemma-pb3} is the same as for the Lemma ``Control $\G$ Using Basis'' in \cite{feffermanFinitenessPrinciplesSmooth2016}.
\end{proof}

\section{The Transport Lemma}

\label{transport-lemma}

In this section, we present the following result.
\begin{lemma}[Transport Lemma]
\label{lemma-transport}Let $\tg_{0}=\blobdef[0]$ be a blob field with blob constant $C_{\G}$. For $l\geq 1$, let $\tg_{l}=\blobdef[l]$ be the approximate $l$-th refinement of $\tg_{0}$.

\begin{itemize}
\item[\refstepcounter{equation}\text{(\theequation)}\label{t1}] Suppose $%
\mathcal{A\subseteq M}$ is monotonic and $\hat{\mathcal{A}} \subseteq 
\mathcal{M}$ (not necessarily monotonic).
\end{itemize}

Let $x_{0}\in E$, $M_{0}>0$, $l_{0}\geq 1$, $\delta >0$, $C_{B}$, $\hat{C}%
_{B}$, $C_{DIFF}>0$. Let $P^{0}$, $\hat{P}^{0}\in \mathcal{P}$.
Assume that the following hold.

\begin{itemize}
\item[\refstepcounter{equation}\text{(\theequation)}\label{t2}] $\tg_{l_{0}}$ has an $\left( \mathcal{A},\delta ,C_{B}\right) $-basis at $\left(
x_{0},M_{0},\tau_{0},P^{0}\right) $, and an $\left(\hat{\mathcal{A}} ,\delta ,\hat{C}%
_{B}\right) $-basis at $\left(x_{0}, M_{0},\tau_0,\hat{P}^{0}\right) $.
\end{itemize}

\begin{itemize}
\item[\refstepcounter{equation}\text{(\theequation)}\label{t3}] $%
\partial^\beta(P^0-\hat{P}^0)\equiv 0$ for $\beta \in \mathcal{A}$.
\end{itemize}

\begin{itemize}
\item[\refstepcounter{equation}\text{(\theequation)}\label{t4}] $%
|\partial^\beta (P^0 - \hat{P}^0)(x_0)|\leq C_{DIFF}M_0\delta^{m-|\beta|}$
for $\beta \in \mathcal{M}$.
\end{itemize}

Let $y_0 \in E$, and suppose that

\begin{itemize}
\item[\refstepcounter{equation}\text{(\theequation)}\label{t5}] $%
|x_0-y_0|\leq \epsilon_0\delta $,
\end{itemize}
where $\epsilon_0$ is a a small enough constant determined by $C_B$, $\hat{C}%
_B$, $C_{DIFF}$, $m$, $n$ and the blob constant $C_{\G}$. Then there exists $\hat{P}^\# \in \mathcal{P}$
with the following properties.

\begin{itemize}
\item[\refstepcounter{equation}\text{(\theequation)}\label{t6}] $\tg%
_{l_0-1}$ has both an $(\mathcal{A},\delta,C_B^{\prime })$-basis and an $(%
\hat{\mathcal{A}},\delta,C_B^{\prime })$-basis at $(y_0,M_0,\tau_0,\hat{P}^\#)$,
with $C_B^{\prime }$ determined by $C_B$, $\hat{C}_B$, $C_{DIFF}$, $m$, $n$ and the blob constant $C_{\G}$.
\end{itemize}
\begin{itemize}
\item[\refstepcounter{equation}\text{(\theequation)}\label{t7}] $%
\partial^\beta(\hat{P}^\#-P^0)\equiv0$ for $\beta \in \mathcal{A}$.
\end{itemize}

\begin{itemize}
\item[\refstepcounter{equation}\text{(\theequation)}\label{t8}] $%
|\partial^\beta(\hat{P}^\#-P^0)(x_0)|\leq C^{\prime }M_0\delta^{m-|\beta|}$
for $\beta \in \mathcal{M}$, with $C^{\prime }$ determined by $C_B$, $\hat{C}%
_B$, $C_{DIFF}$, $m$, $n$ and the blob constant $C_{\G}$.
\end{itemize}
\end{lemma}

\begin{remark}
Note that $\mathcal{A}$ and $\hat{\mathcal{A}}$ play different roles here;
see \eqref{t1}, \eqref{t3}, and \eqref{t7}.
\end{remark}

\begin{proof}[Proof of the Transport Lemma]
The proof is the same as in \cite{feffermanFinitenessPrinciplesSmooth2016}. The constant introduced in the approximate refinements can be hidden into $C_B'$.
\end{proof}

For future reference, we state the special case of the Transport Lemma in
which we take $\hat{\mathcal{A}} = \mathcal{A}$, $\hat{P}^0=P^0$.

\begin{corollary}
\label{cor-to-transport} Let $\tg_{0}=\left( \Gamma _{0}\left(
x,M\right) \right) _{x\in E, M>0, \tau\in(0,\tau_{\max}]}$ be a blob field with blob constant $C_{\G}$. For $l\geq 1$, let $\tg_{l}=\left( \Gamma _{l}\left( x,M\right) \right) _{x\in E, M>0, \tau\in(0,\tau_{\max}]}$ be
the approximate $l$-th refinement of $\tg_{0}$. Suppose

\begin{itemize}
\item[\refstepcounter{equation}\text{(\theequation)}\label{t62}] $\mathcal{A}
\subseteq \mathcal{M}$ is monotonic.
\end{itemize}

Let $x_{0}\in E$, $M_{0}>0$, $0<\tau_0\leq \tau_{\max}$, $l_{0}\geq 1$, $\delta >0,C_{B}>0$; and let $%
P^{0}\in \mathcal{P}$. Assume that

\begin{itemize}
\item[\refstepcounter{equation}\text{(\theequation)}\label{t63}] $\tg_{l_{0}}$ has an $\left( \mathcal{A},\delta ,C_{B}\right) $-basis at $%
\left( x_{0},M_{0},\tau_0,P^{0}\right) $.
\end{itemize}

Let $y_0 \in E$, and suppose that

\begin{itemize}
\item[\refstepcounter{equation}\text{(\theequation)}\label{t64}] $\left\vert
x_{0}-y_{0}\right\vert \leq \epsilon _{0}\delta $, where $\epsilon _{0} $ is
a small enough constant determined by $C_{B}$, $m$, $n$ and the blob constant $C_{\G}$.
\end{itemize}

Then there exists $\hat{P}^{\#}\in \mathcal{P}$ with the following
properties.

\begin{itemize}
\item[\refstepcounter{equation}\text{(\theequation)}\label{t65}] $\tg_{l_{0}-1}$ has an $\left( \mathcal{A},\delta ,C_{B}^{\prime }\right) $%
-basis at $\left( y_{0},M_{0},\tau_0, \hat{P}^{\#}\right) $, with $C_{B}^{\prime }$
determined by $C_{B}$, $m$, $n$ and the blob constant $C_{\G}$.
\end{itemize}

\begin{itemize}
\item[\refstepcounter{equation}\text{(\theequation)}\label{t66}] $\partial
^{\beta }\left( \hat{P}^{\#}-P^{0}\right) \equiv 0$ for $\beta \in \mathcal{A%
}$.
\end{itemize}

\begin{itemize}
\item[\refstepcounter{equation}\text{(\theequation)}\label{t67}] $\left\vert
\partial ^{\beta }\left( \hat{P}^{\#}-P^{0}\right) \left( x_{0}\right)
\right\vert \leq C^{\prime }M_{0}\delta ^{m-\left\vert \beta \right\vert }$
for all $\beta \in \mathcal{M}$, with $C^{\prime }$ determined by $C_{B}$, $m
$, $n$ and the blob constant $C_{\G}$.
\end{itemize}
\end{corollary}
\begin{remark}We will need to find the polynomial $\hat{P}^{\#}$ in the main algorithm. This can be done by solving a linear programming problem with dimension and number of constraints bounded by a constant depending on $n,m$; and we know a solution exists.
\end{remark}
\part{The Main Lemma}
\section{Statement of the Main Lemma}

\label{statement-of-the-main-lemma}

For $\mathcal{A}\subseteq \mathcal{M}$ monotonic, we define

\begin{itemize}
\item[\refstepcounter{equation}\text{(\theequation)}\label{m1}] $l\left( 
\mathcal{A}\right) =1+3\cdot \#\left\{ \mathcal{A}^{\prime }\subseteq 
\mathcal{M}:\mathcal{A}^{\prime }\text{ monotonic, }\mathcal{A}^{\prime }<%
\mathcal{A}\right\} $.
\end{itemize}

Thus,

\begin{itemize}
\item[\refstepcounter{equation}\text{(\theequation)}\label{m2}] $l\left( 
\mathcal{A}\right) -3\geq l\left( \mathcal{A}^{\prime }\right) $ for $%
\mathcal{A}^{\prime },\mathcal{A}\subseteq \mathcal{M}$ monotonic with $%
\mathcal{A}^{\prime }<\mathcal{A}$.
\end{itemize}

By induction on $\mathcal{A}$ (with respect to the order relation $<$), we
will prove the following result.

\begin{lemma}[Main Lemma for $\mathcal{A}$]\label{lem:main-lemma}
	Let $\tg_{0}=\blobdef[0]$ be a $\left( C_{w},\delta _{\max }\right) $-convex blob field with blob constant $C_{\G}$, and for $l\geq 1$, let $\tg_{l}=\blobdef[l]$ be the approximate $l$-th refinement of $\tg_{0}$. Fix a dyadic cube $Q_{0}\subset \mathbb{R}^{n}$. Let $E_0 = E \cap \frac{65}{64}Q_0$, and assume it is not empty. Fix a point $x_{0}\in E_0$ and a polynomial $P^{0}\in \mathcal{P}$, as well as positive real numbers $M_{0}$, $0<\tau_{0}\leq\tau_{\max}$, $\epsilon $, $C_{B}$.  We make the following assumptions.
	\begin{itemize}
		\item[(A1)] $\tg_{l\left( \mathcal{A}\right) }$ has an $\left( \mathcal{A},\epsilon ^{-1}\delta _{Q_{0}},C_{B}\right) $-basis at $\left( x_{0},M_{0},\tau_{0},P^{0}\right) $.
		\item[(A2)] $\epsilon ^{-1}\delta _{Q_{0}}\leq \delta _{\max }$.
		\item[(A3)] (\textquotedblleft Small $\epsilon $ Assumption")\ $\epsilon $ is less than a small enough constant determined by $C_{B}$, $C_{w}$, $m$, $n$ and the blob constant $C_{\G}$.
	\end{itemize}
	Then there exists $F\in C^{m}\left( \frac{65}{64}Q_{0}\right)$ satisfying the following conditions.

	\begin{itemize}
		\item[(C1)] $\left\vert \partial ^{\beta }\left(F-P^{0}\right)\right\vert \leq C\left( \epsilon \right) M_{0}\delta _{Q_{0}}^{m-\left\vert\beta \right\vert }$ on $\frac{65}{64}Q_{0}$ for $\left\vert \beta\right\vert \leq m$, where $C\left( \epsilon \right) $ is determined by $\epsilon $, $C_{B}$, $C_{w}$, $m$, $n$, $C_{\G}$.
		
		\item[(C2)] $J_{z}\left( F\right) \in \Gamma _{0}\left(z,C^{\prime}\left( \epsilon \right) M_{0}, C'(\epsilon)\tau_0 \right) $ for all $z\in E_0$, where $C^{\prime }\left( \epsilon \right) $ is determined by $\epsilon $, $C_{B}$, $C_{w}$, $m$, $n$, $C_{\G}$.
	\end{itemize}
\end{lemma}
\begin{remark}
We state the Main Lemma only for monotonic $\mathcal{A}$.
\end{remark}

Note that we do not assert that $J_{x_0}(F)=P^0$.

\section{The Base Case}

\label{the-base-case} The base case of our induction on $\mathcal{A}$ is the
case $\mathcal{A}= \mathcal{M}$.

In this section, we prove the Main Lemma for $\mathcal{M}$. The hypotheses
of the lemma are as follows:

\begin{itemize}
\item[\refstepcounter{equation}\text{(\theequation)}\label{b1}] $\tg%
_{0}=\left( \Gamma _{0}\left( x,M,\tau\right) \right) _{x\in E, M>0, \tau\in(0,\tau_{\max}]}$ is a $\left(
C_{w},\delta _{\max }\right) $-convex blob field with blob constant $C_{\G}$.
\end{itemize}

\begin{itemize}
\item[\refstepcounter{equation}\text{(\theequation)}\label{b2}] $\tg%
_{1}=\left( \Gamma _{1}\left( x,M,\tau\right) \right) _{x\in E, M>0, \tau\in(0,\tau_{\max}]}$ is the first
approximate refinement of $\tg_{0}$.
\end{itemize}

\begin{itemize}
\item[\refstepcounter{equation}\text{(\theequation)}\label{b3}] $\tg%
_{1}$ has an $\left( \mathcal{M},\epsilon ^{-1}\delta _{Q_{0}},C_{B}\right) $%
-basis at $\left( x_{0},M_{0},\tau_0,P^{0}\right) $.
\end{itemize}

\begin{itemize}
\item[\refstepcounter{equation}\text{(\theequation)}\label{b4}] $\epsilon
^{-1}\delta _{Q_{0}}\leq \delta _{\max }$.
\end{itemize}

\begin{itemize}
\item[\refstepcounter{equation}\text{(\theequation)}\label{b5}] $\epsilon $
is less than a small enough constant determined by $C_{B}$, $C_{w}$, $m$, $n$,$C_{\G}$%
.
\end{itemize}

\begin{itemize}
\item[\refstepcounter{equation}\text{(\theequation)}\label{b6}] $x_{0}\in
E_0$.
\end{itemize}

We write $c$, $C$, $C^{\prime }$, etc., to denote constants determined by $%
C_{B}$, $C_{W}$, $m$, $n$, $C_{\G}$. These symbols may denote different constants in
different occurrences.

\begin{itemize}
\item[\refstepcounter{equation}\text{(\theequation)}\label{b7}] Let $z\in
E\cap \frac{65}{64}Q_{0}$.
\end{itemize}

Then \eqref{b6}, \eqref{b7} imply that

\begin{itemize}
\item[\refstepcounter{equation}\text{(\theequation)}\label{b8}] $\left\vert
z-x_{0}\right\vert \leq C\delta _{Q_{0}}=C\epsilon \cdot \left( \epsilon
^{-1}\delta _{Q_{0}}\right) $.
\end{itemize}

From \eqref{b1}, \eqref{b2}, \eqref{b3}, \eqref{b5}, \eqref{b8}, and
Corollary \ref{cor-to-transport} in Section \ref{transport-lemma}, we obtain
a polynomial $\hat{P}^{\#}\in \mathcal{P}$ such that

\begin{itemize}
\item[\refstepcounter{equation}\text{(\theequation)}\label{b9}] $\tg%
_{0}$ has an $\left( \mathcal{M},\epsilon ^{-1}\delta _{Q_{0}},C^{\prime
}\right) $-basis at $\left( z,M_{0},\tau_{0},\hat{P}^{\#}\right) $, and
\end{itemize}

\begin{itemize}
\item[\refstepcounter{equation}\text{(\theequation)}\label{b10}] $\partial
^{\beta }\left( \hat{P}^{\#}-P^{0}\right) =0$ for $\beta \in \mathcal{M}$.
\end{itemize}

From \eqref{b9}, we have $\hat{P}^\# \in \Gamma_0(z,C^{\prime }M_0,C'\tau_0)$, while %
\eqref{b10} tells us that $\hat{P}^\# = P^0$. Thus,

\begin{itemize}
\item[\refstepcounter{equation}\text{(\theequation)}\label{b11}] $P^{0}\in
\Gamma _{0}\left( z,C^{\prime }M_{0},C'\tau_0\right) $ for all $z\in E_0$.
\end{itemize}

Consequently, the function $F:=P^0$ on $\frac{65}{64}Q_0$ satisfies the
conclusions (C1), (C2) of the Main Lemma for $\mathcal{M}$.

This completes the proof of the Main Lemma for $\mathcal{M}$.   $ \blacksquare $

\section{Setup for the Induction Step}

\label{setup-for-the-induction-step}

Fix a monotonic set $\mathcal{A}$ strictly contained in $\mathcal{M}$, and
assume the following

\begin{itemize}
\item[\refstepcounter{equation}\text{(\theequation)}\label{s1}] \underline{%
Induction Hypothesis}: The Main Lemma for $\mathcal{A}^{\prime }$ holds for
all monotonic $\mathcal{A}^{\prime }< \mathcal{A}$.
\end{itemize}

Under this assumption, we will prove the Main Lemma for $\mathcal{A}$. Thus,
let $\tg_{0}$, $\tg_{l}$ $(l\geq 1)$, $C_{\G}$ $C_{w}$, $\delta
_{\max }$, $Q_{0}$, $E_0$, $x_{0}$, $P^{0}$, $M_{0}$, $\epsilon $, $C_{B}$ be as in
the hypotheses of the Main Lemma for $\mathcal{A}$. Our goal is to prove the
existence of $F\in C^{m}(\frac{65}{64}Q_{0})$ satisfying conditions (C1) and
(C2). To do so, we introduce a constant $A\geq 1$, and make the following
additional assumptions.

\begin{itemize}
\item[\refstepcounter{equation}\text{(\theequation)}\label{s2}] \underline{%
Large $A$ assumption:} $A$ exceeds a large enough constant determined by $C_B
$, $C_w$, $m$, $n$, $C_{\G}$.
\end{itemize}

\begin{itemize}
\item[\refstepcounter{equation}\text{(\theequation)}\label{s3}] \underline{%
Small $\epsilon$ assumption:} $\epsilon$ is less than a small enough
constant determined by $A$, $C_B$, $C_w$, $m$, $n$, $C_{\G}$.
\end{itemize}

We write $c$, $C$, $C^{\prime }$, etc., to denote constants determined by $%
C_{B}$, $C_{w}$, $m$, $n$, $C_{\G}$. Also we write $c(A)$, $C(A)$, $C^{\prime }(A)$,
etc., to denote constants determined by $A$, $C_{B}$, $C_{W}$, $m$, $n$, $C_{\G}$.
Similarly, we write $C\left( \epsilon \right) $, $c\left( \epsilon \right) $%
, $C^{\prime }\left( \epsilon \right) $, etc., to denote constants
determined by $\epsilon $, $A$, $C_{B}$, $C_{w}$, $m$, $n$, $C_{\G}$. These symbols
may denote different constants in different occurrences.

In place of (C1), (C2), we will prove the existence of a function $F\in
C^{m}\left( \frac{65}{64}Q_{0}\right) $ satisfying

\begin{itemize}
\item[(C*1)] $\left\vert \partial ^{\beta }\left( F-P^{0}\right) \right\vert
\leq C\left( \epsilon \right) M_{0}\delta _{Q_{0}}^{m-\left\vert \beta
\right\vert }$ on $\frac{65}{64}Q_{0}$ for $|\beta |\leq m$; and

\item[(C*2)] $J_{z}\left( F\right) \in \Gamma _{0}\left( z,C\left( \epsilon
\right) M_{0}, C(\epsilon)\tau_{0}\right) $ for all $z\in E_{0}$.
\end{itemize}

Conditions (C*1), (C*2) differ from (C1), (C2) in that the constants in
(C*1), (C*2) may depend on $A$.

Once we establish (C*1) and (C*2), we may fix $A$ to be a constant
determined by $C_{B}$, $C_{w}$, $m$, $n$, $C_\G$, large enough to satisfy the Large $%
A$ Assumption \eqref{s2}. The Small $\epsilon $ Assumption \eqref{s3} will
then follow from the Small $\epsilon $ Assumption (A3) in the Main Lemma for 
$\mathcal{A}$; and the desired conclusions (C1), (C2) will then follow from
(C*1), (C*2).

Thus, our goal is to prove the existence of $F\in C^{m}\left( \frac{65}{64}%
Q_{0}\right) $ satisfying (C*1) and (C*2), assuming \eqref{s1}, \eqref{s2}, %
\eqref{s3} above, along with hypotheses of the Main Lemma for $\mathcal{A}$.
This will complete our induction on $\mathcal{A}$ and establish the Main
Lemma for all monotonic subsets of $\mathcal{M}$.

\section{Calder\'on-Zygmund Decomposition}

\label{cz-decomposition}

We place ourselves in the setting of Section \ref%
{setup-for-the-induction-step}. Let $Q$ be a dyadic cube. We say that $Q$ is
``OK'' if \eqref{cz1} and \eqref{cz2} below are satisfied.

\begin{itemize}
\item[\refstepcounter{equation}\text{(\theequation)}\label{cz1}] $%
5Q\subseteq 5Q_{0}$.

\item[\refstepcounter{equation}\text{(\theequation)}\label{cz2}] Either $%
\#(E_0\cap 5Q)\leq 1$ or there exists $\hat{\A}<\mathcal{A}$ (strict
inequality) for which the following holds:

\item[\refstepcounter{equation}\text{(\theequation)}\label{cz3}] For each $y\in E_0 \cap 5Q$, Algorithm \ref{alg:find-crit-delta} with data $y$, $\hat{\mathcal{A}}$, $A$, $M_0$, $\tau_0$, $\Gamma_{in} = \tg_{l\left( \mathcal{A}\right) -3}(y, AM_0, A\tau_0) \cap \P^0$, $\Gamma = \tg_{l\left(\mathcal{A}\right)-3}(y, AM_0, A\tau_0)$ where
  \begin{align*}
    \P^0 = \{P\in\P : |&\partial^{\beta}(P-P^0)(x_0)| \leq AM_0(\epsilon^{-1}\delta_{Q_0})^{m-|\beta|}&\;\forall\beta\in\M\\
                      &\partial^{\beta}(P-P^0) \equiv 0&\;\forall\beta\in\mathcal{A}\}
  \end{align*}
  produces a $\tilde{\delta}$ such that $\tilde{\delta} \geq \epsilon^{-1}\delta_Q$.
\end{itemize}

\begin{remark}
The argument in this section and the next will depend sensitively on several
details of the above definition. Note that \eqref{cz3} involves $\tg_{l(%
\mathcal{A})-3}$ rather than $\tg_{l(\hat{\mathcal{A}})}$, and that
the set $\P^0$ of \eqref{cz3} involves $x_0$, $\delta_{Q_0}$ rather than $y$, $\delta_Q$. Note also
that the set $\hat{\mathcal{A}}$ in \eqref{cz2}, \eqref{cz3} needn't be
monotonic.
\end{remark}
We prove now two Lemmas relating the OK-ness of a cube with a weak basis.
\begin{lemma}
	\label{lem:cz-equiv}
	We place ourselves in the setting of Section \ref{setup-for-the-induction-step}. Let $Q$ be a dyadic cube. Suppose:
	\begin{itemize}
		\item[\refstepcounter{equation}\text{(\theequation)}\label{cz1-orig}] $%
		5Q\subseteq 5Q_{0}$. 
	
		\item[\refstepcounter{equation}\text{(\theequation)}\label{cz2-orig}] Either $%
		\#(E_0 \cap 5Q)\leq 1$ or there exists $\hat{A}<\mathcal{A}$ (strict
		inequality) for which the following holds:
	
		\item[\refstepcounter{equation}\text{(\theequation)}\label{cz3-orig}] For each $%
		y\in E_0 \cap 5Q$ there exists $\hat{P}^{y}\in \mathcal{P}$ satisfying
		
		\begin{itemize}
			\item[(\ref{cz3-orig}a)] $\tg_{l\left( \mathcal{A}\right) -3}$ has a weak $%
			\left( \hat{\mathcal{A}},\epsilon ^{-1}\delta _{Q},C\right) $-basis at $%
			\left( y,M_{0},\tau_{0},\hat{P}^{y}\right) $.
			
			\item[(\ref{cz3-orig}b)] $\left\vert \partial ^{\beta }\left( \hat{P}^{y}-P^{0}\right)
			\left( x_{0}\right) \right\vert \leq CM_{0}\left( \epsilon ^{-1}\delta
			_{Q_{0}}\right) ^{m-\left\vert \beta \right\vert }$ for all $\beta \in 
			\mathcal{M}$.
			
			\item[(\ref{cz3-orig}c)] $\partial ^{\beta }\left( \hat{P}^{y}-P^{0}\right) \equiv 0$ for 
			$\beta \in \mathcal{A}$.
		\end{itemize}
	\end{itemize}
	Then, the cube $Q$ is OK.
\end{lemma}
\begin{proof} 
	If $\#(E_0 \cap 5Q) \leq 1$ we are done. Otherwise, we just need to compare \eqref{cz3} with \eqref{cz3-orig}. Suppose \eqref{cz3-orig}.
	
	Then, we know there exist $\hat{P}^y$ and $P_{\alpha}$, $\alpha \in \A$ satisfying:

	\begin{itemize}
		\item[\eqref{pb1}] $\hat{P}^{y}\in
		\Gamma \left( x_{0},CM_{0}, C\tau_0\right) $.
		
		\item[\eqref{pb2}] $\hat{P}^{y}+%
		\frac{M_{0}(\epsilon ^{-1}\delta _{Q}) ^{m-\left\vert \alpha \right\vert }}{C}P_{\alpha }$, $%
		\hat{P}^{y}-\frac{M_{0}(\epsilon ^{-1}\delta _{Q}) ^{m-\left\vert \alpha \right\vert }}{C}P_{\alpha
		}\in \Gamma \left( x_{0},CM_{0},C\tau_0\right) $ for all $\alpha \in \mathcal{A}$%
		.
		\item[\eqref{pb3}] $\partial
		^{\beta }P_{\alpha }\left( x_{0}\right) =\delta _{\alpha \beta }$ (Kronecker
		delta) for $\beta ,\alpha \in \mathcal{A}$.
		
		\item[\eqref{pb4}] $\left\vert
		\partial ^{\beta }P_{\alpha}\left( x_{0}\right) \right\vert \leq
		C(\epsilon ^{-1}\delta _{Q}) ^{\left\vert \alpha \right\vert -\left\vert \beta \right\vert }$
		for all $\alpha \in \mathcal{A}$, $\beta \geq \alpha$.
	\end{itemize}
	Also, applying Algorithm \ref{alg:find-crit-delta} as in \eqref{cz3} returns a $\tilde{\delta}$ such that:
	\begin{enumerate}
		\item[(I)] There exist $P_w\in\Gamma_{in}$ and $\tilde{P}_\alpha\in\P$ ($\alpha\in\A$) such that:
		\begin{enumerate}
			\item[(A)] $\partial^\beta \tilde{P}_\alpha (x_0) =\delta_{\beta\alpha}$ for $\beta,\alpha\in\A$.
			\item[(B)] $|\partial^\beta \tilde{P}_\alpha(x_0)| \leq CA\tilde{\delta}^{|\alpha|-|\beta|}$ for $\alpha\in\A$,$\beta\in\M$, $\beta\geq\alpha$.
			\item[(C)] $P_w \pm \frac{M_0\tilde{\delta}^{m-|\alpha|}\tilde{P}_\alpha}{CA}\in(1+A\tau_0)\blacklozenge\Gamma$
		\end{enumerate}
		\item[(II)] Suppose $0<\delta<\infty$ and $P_w \in \Gamma_{in}$, $P_\alpha\in\P$ ($\alpha\in \A$) satisfy:
		\begin{enumerate}
			\item[(A)] $\partial^\beta P_\alpha (x_0) =\delta_{\beta\alpha}$ for $\beta,\alpha\in\A$.
			\item[(B)] $|\partial^\beta P_\alpha(x_0)| \leq cA\delta^{|\alpha|-|\beta|}$ for $\alpha\in\A$,$\beta\in\M$, $\beta\geq\alpha$.
			\item[(C)] $P_w \pm \frac{M_0\delta^{m-|\alpha|}P_\alpha}{cA}\in(1+A\tau_0)\blacklozenge\Gamma$
		\end{enumerate}
		Then $0<\delta\leq\tilde{\delta}$.
	\end{enumerate}
	
	Thanks to the large $A$ assumption, we know that $A$ is greater than $\max\{C,\frac{C}{c}\}$ (so that $\hat{P}^y \in \Gamma_{in}$). Then it is clear we are in case (II), therefore $\epsilon^{-1}\delta_Q \leq \tilde{\delta}$.
\end{proof}

\begin{lemma}
	\label{lem:ok-basis} We place ourselves in the setting of Section \ref{setup-for-the-induction-step}. Let $Q$ be an OK dyadic cube. Then:
		\begin{itemize}
		\item[\refstepcounter{equation}\text{(\theequation)}\label{cz1-wb}] $%
		5Q\subseteq 5Q_{0}$.
		
		\item[\refstepcounter{equation}\text{(\theequation)}\label{cz2-wb}] Either $%
		\#(E_0\cap 5Q)\leq 1$ or there exists $\hat{A}<\mathcal{A}$ (strict
		inequality) for which the following holds:
		
		\item[\refstepcounter{equation}\text{(\theequation)}\label{cz3-wb}] For each $%
		y\in E_0\cap 5Q$ there exists $\hat{P}^{y}\in \mathcal{P}$ satisfying
		
		\begin{itemize}
			\item[(\ref{cz3-wb}a)] $\tg_{l\left( \mathcal{A}\right) -3}$ has a weak $%
			\left( \hat{\mathcal{A}},\epsilon ^{-1}\delta _{Q},CA\right) $-basis at $%
			\left( y,M_{0},\tau_{0},\hat{P}^{y}\right) $.
			
			\item[(\ref{cz3-wb}b)] $\left\vert \partial ^{\beta }\left( \hat{P}^{y}-P^{0}\right)
			\left( x_{0}\right) \right\vert \leq AM_{0}\left( \epsilon ^{-1}\delta
			_{Q_{0}}\right) ^{m-\left\vert \beta \right\vert }$ for all $\beta \in 
			\mathcal{M}$.
			
			\item[(\ref{cz3-wb}c)] $\partial ^{\beta }\left( \hat{P}^{y}-P^{0}\right) \equiv 0$ for 
			$\beta \in \mathcal{A}$.
		\end{itemize}
	\end{itemize}
\end{lemma}
\begin{proof}
	If $\#(E_0 \cap 5Q) \leq 1$ we are done. Suppose $\#(E_0 \cap 5Q) \geq 2$. It is clear from the definition of an OK cube that Algorithm \ref{alg:find-crit-delta} will return a $\tilde{\delta} \geq \epsilon^{-1}\delta_Q$ such that:
		\begin{enumerate}
		\item[(I)] There exist $P_w\in\Gamma_{in}$ and $P_\alpha\in\P$ ($\alpha\in\A$) such that:
		\begin{enumerate}
			\item[(A)] $\partial^\beta P_\alpha (x_0) =\delta_{\beta\alpha}$ for $\beta,\alpha\in\A$.
			\item[(B)] $|\partial^\beta P_\alpha(x_0)| \leq CA\tilde{\delta}^{|\alpha|-|\beta|}$ for $\alpha\in\A$,$\beta\in\M$, $\beta\geq\alpha$.
			\item[(C)] $P_w \pm \frac{M_0\tilde{\delta}^{m-|\alpha|}P_\alpha}{CA}\in(1+A\tau_0)\blacklozenge\Gamma$
		\end{enumerate}
		\item[(II)] Suppose $0<\delta<\infty$ and $P_w \in \Gamma_{in}$, $P_\alpha\in\P$ ($\alpha\in \A$) satisfy:
		\begin{enumerate}
			\item[(A)] $\partial^\beta P_\alpha (x_0) =\delta_{\beta\alpha}$ for $\beta,\alpha\in\A$.
			\item[(B)] $|\partial^\beta P_\alpha(x_0)| \leq cA\delta^{|\alpha|-|\beta|}$ for $\alpha\in\A$,$\beta\in\M$, $\beta\geq\alpha$.
			\item[(C)] $P_w \pm \frac{M_0\delta^{m-|\alpha|}P_\alpha}{cA}\in(1+A\tau_0)\blacklozenge\Gamma$
		\end{enumerate}
		Then $0<\delta\leq\tilde{\delta}$.
	\end{enumerate}
In particular, because $\tg$ is a blob field, $P_\alpha$ forms a weak $(\A, \tilde{\delta}, C_\Gamma CA)$-basis for $\tg$ at $(x_0, M_0, \tau_0, P_w)$. Therefore, it also forms a weak $(\A, \epsilon^{-1}\delta_Q, C_\Gamma CA)$-basis.
\end{proof}

A dyadic cube $Q$ will be called a \underline{Calder\'on-Zygmund cube} (or a 
\underline{CZ} cube) if it is OK, but no dyadic cube strictly containing $Q$
is OK.

Recall that given any two distinct dyadic cubes $Q$, $Q^{\prime }$, either $Q
$ is strictly contained in $Q^{\prime }$, or $Q^{\prime }$ is strictly
contained in $Q$, or $Q \cap Q^{\prime }=\emptyset$. The first two
alternatives here are ruled out if $Q$, $Q^{\prime }$ are CZ cubes. Hence,
the Calder\'on-Zygmund cubes are pairwise disjoint.

Any CZ cube $Q$ satisfies \eqref{cz1} and is therefore contained in the
interior of $5Q_0$. On the other hand, let $x$ be an interior point of $5Q_0$%
. Then any sufficiently small dyadic cube $Q$ containing $x$ satisfies $5Q
\subset 5Q_0$ and $\#(E_0 \cap 5Q) \leq 1$; hence, $Q$ is OK. However, any
sufficiently large dyadic cube $Q$ containing $x$ will fail to satisfy $5Q
\subseteq 5Q_0$; hence $Q$ is not OK. It follows that $x$ is contained in a
maximal OK dyadic cube. Thus, we have proven

\begin{lemma}
\label{lemma-cz1} The CZ cubes form a partition of the interior of $5Q_0$.
\end{lemma}

Next, we establish

\begin{lemma}
\label{lemma-cz2} Let $Q$, $Q^{\prime }$ be CZ cubes. If $\frac{65}{64}Q
\cap \frac{65}{64}Q^{\prime }\not= \emptyset$, then $\frac{1}{2}\delta_Q
\leq \delta_{Q^{\prime }}\leq 2\delta_Q$.
\end{lemma}

\begin{proof}
Suppose not. Without loss of generality, we may suppose that $\delta_Q \leq 
\frac{1}{4}\delta_{Q^{\prime }}$. Then $\delta_{Q^+} \leq \frac{1}{2}%
\delta_{Q^{\prime }}$, and $\frac{65}{64}Q^+ \cap \frac{65}{64}Q^{\prime
}\not= \emptyset$; hence, $5Q^+ \subset 5Q^{\prime }$. The cube $Q^{\prime }$
is OK. Therefore,

\begin{itemize}
\item[\refstepcounter{equation}\text{(\theequation)}\label{cz4}] $%
5Q^{+}\subset 5Q^{\prime }\subseteq 5Q_{0}$.
\end{itemize}

If $\#\left( E_0 \cap 5Q^{\prime }\right) \leq 1$, then also $\#\left( E_0 \cap
5Q^{+}\right) \leq 1$. Otherwise, since $Q'$ is OK, there exists $\hat{\mathcal{A}}<\mathcal{A}
$ such that for each $y\in E\cap 5Q^{\prime }$, Algorithm \ref{alg:find-crit-delta} with the corresponding data will produce a $\tilde{\delta}$ such that $\tilde{\delta} \geq \epsilon^{-1}\delta_{Q'} \geq \epsilon^{-1}\delta_{Q^+}$.

Therefore, for each $y \in E_0 \cap 5Q^+\subseteq E_0 \cap 5Q^{\prime }$, Algorithm \ref{alg:find-crit-delta} produces a $\tilde{\delta}$ such that $\tilde{\delta} \geq \epsilon^{-1}\delta_{Q^+}$.

This tells us that $Q^{+}$ is OK. However, $Q^{+}$ strictly contains the CZ cube $Q$; therefore, $Q^{+}$ cannot be OK. This contradiction completes the proof of Lemma \ref{lemma-cz2}.
\end{proof}

Note that the proof of Lemma \ref{lemma-cz2} made use of our decision to
involve $x_0$, $\delta_{Q_0}$ rather than $y$, $\delta_Q$ in \eqref{cz3}, as well as Algorithm \ref{alg:find-crit-delta} producing a weak basis instead of a strong basis.

\begin{lemma}
\label{lemma-cz3} Only finitely many CZ cubes $Q$ satisfy the condition

\begin{itemize}
\item[\refstepcounter{equation}\text{(\theequation)}\label{cz9}] $\frac{65}{%
64}Q \cap \frac{65}{64} Q_0 \not= \emptyset$.
\end{itemize}
\end{lemma}

\begin{proof}
There exists some small positive number $\delta _{\ast }$ such that any
dyadic cube $Q$ satisfying \eqref{cz9} and $\delta _{Q}\leq \delta _{\ast }$
must satisfy also $5Q\subset 5Q_{0}$ and $\#(E_0\cap 5Q)\leq 1$. (Here we use
the finiteness of $E$.)

Consequently, any CZ cube $Q$ satisfying \eqref{cz9} must have sidelength $\delta
_{Q}\geq \delta _{\ast }$ (and also $\delta _{Q}\leq \delta
_{Q_{0}}$ since $5Q\subset 5Q_{0}$ because $Q$ is OK). There are only
finitely many dyadic cubes $Q$ satisfying both \eqref{cz9} and $\delta
_{\ast }\leq \delta _{Q}\leq \delta _{Q_{0}}$.

The proof of Lemma \ref{lemma-cz3} is complete.
\end{proof}

\section{Auxiliary Polynomials}

\label{auxiliary-polynomials}

We again place ourselves in the setting of Section \ref{setup-for-the-induction-step} and we make use of the Calder\'on-Zygmund
decomposition defined in Section \ref{cz-decomposition}.

Recall that $x_0  \in E_0 = E_0 \cap 5Q_0^+$, and that $\tg_{l(\mathcal{A})}$
has an $(\mathcal{A}, \epsilon^{-1}\delta_{Q_0}, C_B)$-basis at $(x_0,M_0,\tau_0,P^0) $; moreover, $\mathcal{A} \subseteq \mathcal{M}$ is monotonic,
and $\epsilon$ is less than a small enough constant determined by $C_B$, $C_w
$, $m$, $n$.

Let $y_0\in E_0\cap 5Q_{0}$. Then $|x_{0}-y|\leq C\delta _{Q_{0}}=(C\epsilon
)(\epsilon ^{-1}\delta _{Q_{0}})$. Hence, by Corollary \ref{cor-to-transport}
in Section \ref{transport-lemma}, there exists $P^{y}\in \mathcal{P}$ with
the following properties.

\begin{itemize}
\item[\refstepcounter{equation}\text{(\theequation)}\label{ap1}] $\tg_{l\left( \mathcal{A}\right) -1}$ has an $\left( \mathcal{A},\epsilon
^{-1}\delta _{Q_{0}},C\right) $-basis $\left( P_{\alpha }^{y}\right)
_{\alpha \in \mathcal{A}}$ at $\left( y,M_{0},\tau_0,P^{y}\right) $,
\end{itemize}

\begin{itemize}
\item[\refstepcounter{equation}\text{(\theequation)}\label{ap2}] $\partial
^{\beta }\left( P^{y}-P^{0}\right) \equiv 0$ for $\beta \in \mathcal{A}$,
\end{itemize}

\begin{itemize}
\item[\refstepcounter{equation}\text{(\theequation)}\label{ap3}] $\left\vert
\partial ^{\beta }\left( P^{y}-P^{0}\right) \left( x_{0}\right) \right\vert
\leq CM_{0}\left( \epsilon ^{-1}\delta _{Q_{0}}\right) ^{m-\left\vert \beta
\right\vert }$ for $\beta \in \mathcal{M}$.
\end{itemize}

We fix $P^{y},P_{\alpha }^{y}$ $\left( \alpha \in \mathcal{A}\right) $ as
above for each $y\in E_0\cap 5Q_{0}$. We study the relationship between the
polynomials $P^{y},P_{\alpha }^{y}$ $(\alpha \in \mathcal{A})$ and the Calder%
\'{o}n-Zygmund decomposition.

\begin{lemma}[``Controlled Auxiliary Polynomials'']
\label{lemma-ap1} Let $Q \in$ CZ, and suppose that

\begin{itemize}
\item[\refstepcounter{equation}\text{(\theequation)}\label{ap4}] $\frac{65}{%
64}Q\cap \frac{65}{64}Q_{0}\not=\emptyset $.
\end{itemize}

Let

\begin{itemize}
\item[\refstepcounter{equation}\text{(\theequation)}\label{ap5}] $y\in E_0 \cap
5Q_{0}\cap 5Q^{+}$.
\end{itemize}

Then

\begin{itemize}
\item[\refstepcounter{equation}\text{(\theequation)}\label{ap6}] $\left\vert
\partial ^{\beta }P_{\alpha }^{y}\left( y\right) \right\vert \leq C\cdot
\left( \epsilon ^{-1}\delta _{Q}\right) ^{\left\vert \alpha \right\vert
-\left\vert \beta \right\vert }$ for $\alpha \in \mathcal{A}$, $\beta \in 
\mathcal{M}$.
\end{itemize}
\end{lemma}

\begin{proof}
Let $K\geq 1$ be a large enough constant to be picked below and assume that

\begin{itemize}
\item[\refstepcounter{equation}\text{(\theequation)}\label{ap7}] $%
\max_{\alpha \in \mathcal{A}\text{,}\beta \in \mathcal{M}}\left( \epsilon
^{-1}\delta _{Q}\right) ^{\left\vert \beta \right\vert -\left\vert \alpha
\right\vert }\left\vert \partial ^{\beta }P_{\alpha }^{y}\left( y\right)
\right\vert >K\text{.} $
\end{itemize}

We will derive a contradiction.

Thanks to \eqref{ap1}, we have

\begin{itemize}
\item[\refstepcounter{equation}\text{(\theequation)}\label{ap8}] $%
P^{y},P^{y}\pm cM_{0}\cdot \left( \epsilon ^{-1}\delta _{Q_{0}}\right)
^{m-\left\vert \alpha \right\vert }P_{\alpha }^{y}\in \Gamma _{l\left( 
\mathcal{A}\right) -1}\left( y,CM_{0},C\tau_0\right) $ for $\alpha \in \mathcal{A}$,
\end{itemize}

\begin{itemize}
\item[\refstepcounter{equation}\text{(\theequation)}\label{ap9}] $\partial
^{\beta }P_{\alpha }^{y}\left( y\right) =\delta _{\beta \alpha }$ for $\beta
,\alpha \in \mathcal{A}$,
\end{itemize}

and

\begin{itemize}
\item[\refstepcounter{equation}\text{(\theequation)}\label{ap10}] $%
\left\vert \partial ^{\beta }P_{\alpha }^{y}\left( y\right) \right\vert \leq
C\left( \epsilon ^{-1}\delta _{Q_{0}}\right) ^{\left\vert \alpha \right\vert
-\left\vert \beta \right\vert }$ for $\alpha \in \mathcal{A}$, $\beta \in 
\mathcal{M}$.
\end{itemize}

Also,

\begin{itemize}
\item[\refstepcounter{equation}\text{(\theequation)}\label{ap11}] $5Q\subset
5Q_{0}$ since $Q$ is OK.
\end{itemize}

If $\delta _{Q}\geq 2^{-12}\delta _{Q_{0}}$, then from \eqref{ap10}, %
\eqref{ap11}, we would have

\begin{itemize}
\item[\refstepcounter{equation}\text{(\theequation)}\label{ap14}] $%
\max_{\alpha \in \mathcal{A}\text{,}\beta \in \mathcal{M}}\left( \epsilon
^{-1}\delta _{Q}\right) ^{\left\vert \beta \right\vert -\left\vert \alpha
\right\vert }\left\vert \partial ^{\beta }P_{\alpha }^{y}\left( y\right)
\right\vert \leq C^{\prime }$.
\end{itemize}

We will pick

\begin{itemize}
\item[\refstepcounter{equation}\text{(\theequation)}\label{ap15}] $%
K>C^{\prime }$, with $C^{\prime }$ as in \eqref{ap14}.
\end{itemize}

Then \eqref{ap14} contradicts our assumption \eqref{ap7}.

Thus, we must have

\begin{itemize}
\item[\refstepcounter{equation}\text{(\theequation)}\label{ap16}] $\delta
_{Q}<2^{-12}\delta _{Q_{0}}$.
\end{itemize}

Let

\begin{itemize}
\item[\refstepcounter{equation}\text{(\theequation)}\label{ap17}] $Q=\hat{Q}%
_{0}\subset \hat{Q}_{1}\subset \cdots \subset \hat{Q}_{\nu _{\max }}$ be all
the dyadic cubes containing $Q$ and having sidelength at most $2^{-10}\delta
_{Q_{0}}$.
\end{itemize}

Then

\begin{itemize}
\item[\refstepcounter{equation}\text{(\theequation)}\label{ap18}] $\hat{Q}%
_{0}=Q$, $\delta _{\hat{Q}_{\nu _{\max }}}=2^{-10}\delta _{Q_{0}}$, $\hat{Q}%
_{\nu +1}=\left( \hat{Q}_{\nu }\right) ^{+}$ for $0\leq \nu \leq \nu _{\max
}-1$, and $\nu _{\max }\geq 2$.
\end{itemize}

For $0\leq \nu \leq \nu _{\max }$, we define

\begin{itemize}
\item[\refstepcounter{equation}\text{(\theequation)}\label{ap19}] $X_{\nu
}=\max_{\alpha \in \mathcal{A}\text{,}\beta \in \mathcal{M}}\left( \epsilon
^{-1}\delta _{\hat{Q}_{\nu }}\right) ^{\left\vert \beta \right\vert
-\left\vert \alpha \right\vert }\left\vert \partial ^{\beta }P_{\alpha
}^{y}\left( y\right) \right\vert $.
\end{itemize}

From \eqref{ap7} and \eqref{ap10}, we have

\begin{itemize}
\item[\refstepcounter{equation}\text{(\theequation)}\label{ap20}] $X_{0}>K$, 
$X_{\nu _{\max }}\leq C^{\prime }$,
\end{itemize}

and from \eqref{ap18}, \eqref{ap19}, we have

\begin{itemize}
\item[\refstepcounter{equation}\text{(\theequation)}\label{ap21}] $%
2^{-m}X_{\nu }\leq X_{\nu +1}\leq 2^{m}X_{\nu }$, for $0\leq \nu \leq \nu
_{\max }$.
\end{itemize}

We will pick

\begin{itemize}
\item[\refstepcounter{equation}\text{(\theequation)}\label{ap22}] $%
K>C^{\prime }$ with $C^{\prime }$ as in \eqref{ap20}.
\end{itemize}

Then $\tilde{\nu}:=\min \left\{ \nu :X_{\nu }\leq K\right\} $ and $\tilde{Q}=%
\hat{Q}_{\tilde{\nu}}$ satisfy the following, thanks to \eqref{ap20}, %
\eqref{ap21}, \eqref{ap22}: $\tilde{\nu}\not=0$, hence

\begin{itemize}
\item[\refstepcounter{equation}\text{(\theequation)}\label{ap23}] $\tilde{Q}$
is a dyadic cube strictly containing $Q$; also $2^{-m}K\leq X_{\tilde{\nu}%
}\leq K$,
\end{itemize}

hence

\begin{itemize}
\item[\refstepcounter{equation}\text{(\theequation)}\label{ap24}] $%
2^{-m}K\leq \max_{\alpha \in \mathcal{A}\text{,}\beta \in \mathcal{M}}\left(
\epsilon ^{-1}\delta _{\tilde{Q}}\right) ^{\left\vert \beta \right\vert
-\left\vert \alpha \right\vert }\left\vert \partial ^{\beta }P_{\alpha
}^{y}\left( y\right) \right\vert \leq K$.
\end{itemize}

Also, since $Q\subset \tilde{Q}$, we have $\frac{65}{64}\tilde{Q}\cap \frac{%
65}{64}Q_{0}\not=\emptyset $ by \eqref{ap4}; and since $\delta _{\tilde{Q}%
}\leq 2^{-10}\delta _{Q_{0}}$, we conclude that

\begin{itemize}
\item[\refstepcounter{equation}\text{(\theequation)}\label{ap25}] $5\tilde{Q}%
\subset 5Q_{0}$. 
\end{itemize}

From \eqref{ap8}, \eqref{ap10}, and \eqref{ap25}, we have

\begin{itemize}
\item[\refstepcounter{equation}\text{(\theequation)}\label{ap26}] $%
P^{y},P^{y}\pm cM_{0}\left( \epsilon ^{-1}\delta _{\tilde{Q}}\right)
^{m-\left\vert \alpha \right\vert }P_{\alpha }^{y}\in \Gamma _{l\left( 
\mathcal{A}\right) -1}\left( y,CM_{0},C\tau_0\right) \subset \Gamma _{l\left( 
\mathcal{A}\right) -2}\left( y,C'M_{0},C'\tau_0\right) $ for $\alpha \in \mathcal{A}$;
\end{itemize}

and

\begin{itemize}
\item[\refstepcounter{equation}\text{(\theequation)}\label{ap27}] $%
\left\vert \partial ^{\beta }P_{\alpha }^{y}\left( y\right) \right\vert \leq
C'\left( \epsilon ^{-1}\delta _{\tilde{Q}}\right) ^{\left\vert \alpha
\right\vert -\left\vert \beta \right\vert }$ for $\alpha \in \mathcal{A}$, $%
\beta \in \mathcal{M}$, $\beta \geq \alpha $.
\end{itemize}

Our results \eqref{ap9}, \eqref{ap26}, \eqref{ap27} tell us that

\begin{itemize}
\item[\refstepcounter{equation}\text{(\theequation)}\label{ap28}] $\left(
P_{\alpha }^{y}\right) _{\alpha \in \mathcal{A}}$ is a weak $\left( \mathcal{%
A},\epsilon ^{-1}\delta _{\tilde{Q}},C\right) $-basis for $\tg%
_{l\left( \mathcal{A}\right) -2}$ at $\left( y,M_{0},\tau_0,P^{y}\right) $.
\end{itemize}

Note also that

\begin{itemize}
\item[\refstepcounter{equation}\text{(\theequation)}\label{ap29}] $\epsilon
^{-1}\delta _{\tilde{Q}}\leq \epsilon ^{-1}\delta _{Q_{0}}\leq \delta _{\max
}$, by \eqref{ap25} and hypothesis (A2) of the Main Lemma for $\mathcal{A}$.
\end{itemize}

Moreover,

\begin{itemize}
\item[\refstepcounter{equation}\text{(\theequation)}\label{ap30}] $\tg_{l\left( \mathcal{A}\right) -2}$ is $\left( C,\delta _{\max }\right) 
$-convex.
\end{itemize}

If we take

\begin{itemize}
\item[\refstepcounter{equation}\text{(\theequation)}\label{ap31}] $K\geq
C^{\ast }$ for a large enough $C^{\ast }$,
\end{itemize}

then \eqref{ap24}, \eqref{ap28}$\cdots $\eqref{ap31} and the Relabeling
Lemma (Lemma \ref{lemma-pb2}) produce a monotonic set $\hat{\mathcal{A}}%
\subset \mathcal{M}$, such that

\begin{itemize}
\item[\refstepcounter{equation}\text{(\theequation)}\label{ap32}] $\hat{%
\mathcal{A}}<\mathcal{A}$ (strict inequality)
\end{itemize}

and

\begin{itemize}
\item[\refstepcounter{equation}\text{(\theequation)}\label{ap33}] $\tg_{l\left( \mathcal{A}\right) -2}$ has an $\left( \hat{\mathcal{A}}%
,\epsilon ^{-1}\delta _{\tilde{Q}},C'\right) $-basis at $\left(
y,M_{0},\tau_0,P^{y}\right) $.
\end{itemize}

Also, from \eqref{ap9}, \eqref{ap24}, \eqref{ap26}, we see that

\begin{itemize}
\item[\refstepcounter{equation}\text{(\theequation)}\label{ap34}] $\left(
P_{\alpha }^{y}\right) _{\alpha \in \mathcal{A}}$ is an $\left( \mathcal{A}%
,\epsilon ^{-1}\delta _{\tilde{Q}},CK\right) $-basis for $\tg%
_{l\left( \mathcal{A}\right) -2}$ at $\left( y,M_{0},\tau_0,P^{y}\right) $.
\end{itemize}

We now pick

\begin{itemize}
\item[\refstepcounter{equation}\text{(\theequation)}\label{ap35}] $K=\hat{C}$
(a constant determined by $C_{B}$, $C_{w}$, $m$, $n$), with $\hat{C}\geq 1 $
large enough to satisfy \eqref{ap15}, \eqref{ap22}, \eqref{ap31}.
\end{itemize}

Then \eqref{ap33} and \eqref{ap34} tell us that

\begin{itemize}
\item[\refstepcounter{equation}\text{(\theequation)}\label{ap36}] $\tg_{l\left( \mathcal{A}\right) -2}$ has both an $\left( \hat{\mathcal{A}%
},\epsilon ^{-1}\delta _{\tilde{Q}},C\right) $-basis and an $\left( \mathcal{%
A},\epsilon ^{-1}\delta _{\tilde{Q}},C\right) $-basis at $\left(
y,M_{0},\tau_0,P^{y}\right) $.
\end{itemize}

Let $z\in E_0\cap 5\tilde{Q}$. Then $z,y\in 5\tilde{Q}^{+}$, hence

\begin{itemize}
\item[\refstepcounter{equation}\text{(\theequation)}\label{ap37}] $%
\left\vert z-y\right\vert \leq C\delta _{\tilde{Q}}=C\epsilon \cdot \left(
\epsilon ^{-1}\delta _{\tilde{Q}}\right) $.
\end{itemize}

From \eqref{ap36}, \eqref{ap37}, the Small $\epsilon $ Assumption and Lemma %
\ref{lemma-transport} (and our hypothesis that $\mathcal{A}$ is monotonic;
see Section \ref{setup-for-the-induction-step}), we obtain a polynomial $%
\check{P}^{z}\in \mathcal{P}$, such that

\begin{itemize}
\item[\refstepcounter{equation}\text{(\theequation)}\label{ap38}] $\tg
_{l\left( \mathcal{A}\right) -3}$ has an $\left( \hat{\mathcal{A}},\epsilon
^{-1}\delta _{\tilde{Q}},C\right) $-basis at $\left( z,M_{0},\tau_0,\check{P}%
^{z}\right) $,
\end{itemize}

\begin{itemize}
\item[\refstepcounter{equation}\text{(\theequation)}\label{ap39}] $\partial
^{\beta }\left( \check{P}^{z}-P^{y}\right) \equiv 0$ for $\beta \in \mathcal{%
A}$,
\end{itemize}

and

\begin{itemize}
\item[\refstepcounter{equation}\text{(\theequation)}\label{ap40}] $%
\left\vert \partial ^{\beta }\left( \check{P}^{z}-P^{y}\right) \left(
y\right) \right\vert \leq CM_{0}\left( \epsilon ^{-1}\delta _{\tilde{Q}%
}\right) ^{m-\left\vert \beta \right\vert }$ for $\beta \in \mathcal{M}$.
\end{itemize}

From \eqref{ap25} and \eqref{ap40}, we have

\begin{itemize}
\item[\refstepcounter{equation}\text{(\theequation)}\label{ap41}] $%
\left\vert \partial ^{\beta }\left( \check{P}^{z}-P^{y}\right) \left(
y\right) \right\vert \leq CM_{0}\left( \epsilon ^{-1}\delta _{Q_{0}}\right)
^{m-\left\vert \beta \right\vert }$ for $\beta \in \mathcal{M}$.
\end{itemize}

Since $y\in \frac{65}{64}Q_{0}$ by hypothesis of Lemma \ref{lemma-ap1}, while $x_{0}\in
\frac{65}{64}Q_{0}$ by hypothesis of the Main Lemma for $\mathcal{A}$, we have $%
\left\vert x_{0}-y\right\vert \leq C\delta _{Q_{0}}$, and therefore %
\eqref{ap41} implies that

\begin{itemize}
\item[\refstepcounter{equation}\text{(\theequation)}\label{ap42}] $%
\left\vert \partial ^{\beta }\left( \check{P}^{z}-P^{y}\right) \left(
x_{0}\right) \right\vert \leq CM_{0}\left( \epsilon ^{-1}\delta
_{Q_{0}}\right) ^{m-\left\vert \beta \right\vert }$ for $\beta \in \mathcal{M%
}$.
\end{itemize}

From \eqref{ap2}, \eqref{ap3}, \eqref{ap39}, \eqref{ap42}, we now have

\begin{itemize}
\item[\refstepcounter{equation}\text{(\theequation)}\label{ap43}] $\partial
^{\beta }\left( \check{P}^{z}-P^{0}\right) \equiv 0$ for $\beta \in \mathcal{%
A}$
\end{itemize}

and

\begin{itemize}
\item[\refstepcounter{equation}\text{(\theequation)}\label{ap44}] $%
\left\vert \partial^{\beta }\left( \check{P}^{z}-P^{0}\right) \left(
x_{0}\right) \right\vert \leq CM_{0}\left( \epsilon ^{-1}\delta
_{Q_{0}}\right) ^{m-\left\vert \beta \right\vert }$ for $\beta \in \mathcal{M%
}$.
\end{itemize}

Our results \eqref{ap38}, \eqref{ap43}, \eqref{ap44} hold for every $z\in
E_0\cap 5\tilde{Q}$. Therefore, for each $z\in E_0\cap 5\tilde{Q}$ there exists $\hat{P}^{z}\in \mathcal{P}$ satisfying

\begin{itemize}
	\item[(\ref{cz3-orig}a)] $\tg_{l\left( \mathcal{A}\right) -3}$ has a weak $%
	\left( \hat{\mathcal{A}},\epsilon ^{-1}\delta _{Q},C\right) $-basis at $%
	\left( z,M_{0},\tau_{0},\hat{P}^{z}\right) $.
	
	\item[(\ref{cz3-orig}b)] $\left\vert \partial ^{\beta }\left( \hat{P}^{z}-P^{0}\right)
	\left( x_{0}\right) \right\vert \leq CM_{0}\left( \epsilon ^{-1}\delta
	_{Q_{0}}\right) ^{m-\left\vert \beta \right\vert }$ for all $\beta \in 
	\mathcal{M}$.
	
	\item[(\ref{cz3-orig}c)] $\partial ^{\beta }\left( \hat{P}^{z}-P^{0}\right) \equiv 0$ for 
	$\beta \in \mathcal{A}$.
\end{itemize}

We can apply now Lemma \ref{lem:cz-equiv}. Therefore we conclude that $\tilde{Q}$ is OK. 

However, since $\tilde{Q}$ properly contains the CZ cube $Q$, (see %
\eqref{ap23}), $\tilde{Q}$ cannot be OK.

This contradiction proves that our assumption \eqref{ap7} must be false.

Thus, $\left\vert \partial ^{\beta }P_{\alpha }^{y}\left( y\right)
\right\vert \leq K\left( \epsilon ^{-1}\delta _{Q}\right) ^{\left\vert
\alpha \right\vert -\left\vert \beta \right\vert }$ for $\alpha \in \mathcal{%
A}$, $\beta \in \mathcal{M}$.

Since we picked $K=\hat{C}$ in \eqref{ap35}, this implies the estimate %
\eqref{ap6}, completing the proof of Lemma \ref{lemma-ap1}.
\end{proof}

\begin{corollary}
\label{cor-to-lemma-ap1}

Let $Q\in $ CZ, and suppose $\frac{65}{64}Q\cap \frac{65}{64}%
Q_{0}\not=\emptyset $. Let $y\in E_0\cap 5Q_{0}\cap 5Q^{+}$. Then $\left(
P_{\alpha }^{y}\right) _{\alpha \in \mathcal{A}}$ is an $\left( \mathcal{A}%
,\epsilon ^{-1}\delta _{Q},C\right) $-basis for $\tg_{l\left( 
\mathcal{A}\right) -1}$ at $\left( y,M_{0},\tau_0,P^{y}\right) $.
\end{corollary}

\begin{proof}
From \eqref{ap1} we have

\begin{itemize}
\item[\refstepcounter{equation}\text{(\theequation)}\label{ap45}] $%
P^{y},P^{y}\pm cM_{0}\left( \epsilon ^{-1}\delta _{Q_{0}}\right)
^{m-\left\vert \alpha \right\vert }P_{\alpha }\in \Gamma _{l\left( \mathcal{A%
}\right) -1}\left( y,CM_{0},C\tau_0\right) $ for $\alpha \in \mathcal{A}$;
\end{itemize}

and

\begin{itemize}
\item[\refstepcounter{equation}\text{(\theequation)}\label{ap46}] $\partial
^{\beta }P_{\alpha }^{y}\left( y\right) =\delta _{\beta \alpha }$ for $\beta
,\alpha \in \mathcal{A}$.
\end{itemize}

Since $5Q \subseteq 5Q_0$ (because $Q$ is OK), we have $\delta_Q \leq
\delta_{Q_0}$, and \eqref{ap45} implies

\begin{itemize}
\item[\refstepcounter{equation}\text{(\theequation)}\label{ap47}] $%
P^{y},P^{y}\pm cM_{0}\left( \epsilon ^{-1}\delta _{Q}\right) ^{m-\left\vert
\alpha \right\vert }P_{\alpha }\in \Gamma _{l\left( \mathcal{A}\right)
-1}\left( y,CM_{0},C\tau_0\right) $ for $\alpha \in \mathcal{A}$.
\end{itemize}

Lemma \ref{lemma-ap1} tells us that

\begin{itemize}
\item[\refstepcounter{equation}\text{(\theequation)}\label{ap48}] $%
\left\vert \partial ^{\beta }P_{\alpha }^{y}\left( y\right) \right\vert \leq
C\left( \epsilon ^{-1}\delta _{Q}\right) ^{\left\vert \alpha \right\vert
-\left\vert \beta \right\vert }$ for $\alpha \in \mathcal{A}$, $\beta \in 
\mathcal{M}$.
\end{itemize}

From \eqref{ap46}, \eqref{ap47}, \eqref{ap48}, we conclude that $%
(P^y_\alpha)_{\alpha \in \mathcal{A}}$ is an $(\mathcal{A}%
,\epsilon^{-1}\delta_Q,C)$-basis for $\tg_{l(\mathcal{A})-1}$ at $%
(y,M_0,P^y)$, completing the proof of Corollary \ref{cor-to-lemma-ap1}.
\end{proof}

\begin{lemma}[\textquotedblleft Consistency of Auxiliary
Polynomials\textquotedblright ]
\label{lemma-ap2} Let $Q,Q^{\prime }\in $ CZ, with

\begin{itemize}
\item[\refstepcounter{equation}\text{(\theequation)}\label{ap49}] $\frac{65}{%
64}Q\cap \frac{65}{64}Q_{0}\not=\emptyset $, $\frac{65}{64}Q^{\prime }\cap 
\frac{65}{64}Q_{0}\not=\emptyset $
\end{itemize}

and

\begin{itemize}
\item[\refstepcounter{equation}\text{(\theequation)}\label{ap50}] $\frac{65}{%
64}Q\cap \frac{65}{64}Q^{\prime }\not=\emptyset $.
\end{itemize}

Let

\begin{itemize}
\item[\refstepcounter{equation}\text{(\theequation)}\label{ap51}] $y\in
E_0\cap 5Q_{0}\cap 5Q^{+}$, $y^{\prime }\in E_0\cap 5Q_{0}\cap 5\left( Q^{\prime
}\right) ^{+}$.
\end{itemize}

Then

\begin{itemize}
\item[\refstepcounter{equation}\text{(\theequation)}\label{ap52}] $%
\left\vert \partial ^{\beta }\left( P^{y}-P^{y^{\prime }}\right) \left(
y^{\prime }\right) \right\vert \leq CM_{0}\left( \epsilon ^{-1}\delta
_{Q}\right) ^{m-\left\vert \beta \right\vert }$ for $\beta \in \mathcal{M}$.
\end{itemize}
\end{lemma}

\begin{proof}
Suppose first that $\delta _{Q}\geq 2^{-20}\delta _{Q_{0}}$. Then \eqref{ap3}
(applied to $y$ and to $y^{\prime }$) tells us that 
\begin{equation*}
\left\vert \partial ^{\beta }\left( P^{y}-P^{y^{\prime }}\right) \left(
x_{0}\right) \right\vert \leq CM_{0}\left( \epsilon ^{-1}\delta
_{Q_{0}}\right) ^{m-\left\vert \beta \right\vert }\text{ for }\beta \in 
\mathcal{M}\text{.}
\end{equation*}%
Hence, $\left\vert \partial ^{\beta }\left( P^{y}-P^{y^{\prime }}\right)
\left( y^{\prime }\right) \right\vert \leq C^{\prime }M_{0}\left( \epsilon
^{-1}\delta _{Q_{0}}\right) ^{m-\left\vert \beta \right\vert }\leq C^{\prime
\prime }M_{0}\left( \epsilon ^{-1}\delta _{Q}\right) ^{m-\left\vert \beta
\right\vert }$ for $\beta \in \mathcal{M}$, since $x_{0}$, $y^{\prime }\in
\frac{65}{64}Q_0$. Thus, \eqref{ap52} holds if $\delta _{Q}\geq 2^{-20}\delta
_{Q_{0}}$. Suppose

\begin{itemize}
\item[\refstepcounter{equation}\text{(\theequation)}\label{ap53}] $\delta
_{Q}<2^{-20}\delta _{Q_{0}}$.
\end{itemize}

By \eqref{ap50} and Lemma \ref{lemma-cz2}, we have

\begin{itemize}
\item[\refstepcounter{equation}\text{(\theequation)}\label{ap54}] $\delta
_{Q},\delta _{Q^{\prime }}\leq 2^{-20}\delta _{Q_{0}}$ and $\frac{1}{2}%
\delta _{Q}\leq \delta _{Q^{\prime }}\leq 2\delta _{Q}$.
\end{itemize}

Together with \eqref{ap49}, this implies that

\begin{itemize}
\item[\refstepcounter{equation}\text{(\theequation)}\label{ap55}] $%
5Q^{+},5\left( Q^{\prime }\right) ^{+}\subseteq 5Q_{0}$. 
\end{itemize}

From Corollary \ref{cor-to-lemma-ap1}, we have

\begin{itemize}
\item[\refstepcounter{equation}\text{(\theequation)}\label{ap56}] $\tg_{l\left( \mathcal{A}\right) -1}$ has an $\left( \mathcal{A},\epsilon
^{-1}\delta _{Q^{\prime }},C\right) $-basis at $\left( y^{\prime
},M_{0},\tau_0,P^{y^{\prime }}\right) $.
\end{itemize}

From \eqref{ap50}, \eqref{ap51}, \eqref{ap54}, we have

\begin{itemize}
\item[\refstepcounter{equation}\text{(\theequation)}\label{ap57}] $%
\left\vert y-y^{\prime }\right\vert \leq C\delta _{Q^{\prime }}=C\epsilon
\left( \epsilon ^{-1}\delta _{Q^{\prime }}\right) $.
\end{itemize}

We recall from \eqref{ap54} and the hypotheses of the Main Lemma for $%
\mathcal{A}$ that

\begin{itemize}
\item[\refstepcounter{equation}\text{(\theequation)}\label{ap58}] $\epsilon
^{-1}\delta _{Q^{\prime }}\leq \epsilon ^{-1}\delta _{Q_{0}}\leq \delta
_{\max }$,
\end{itemize}

and we recall from Section \ref{setup-for-the-induction-step} that

\begin{itemize}
\item[\refstepcounter{equation}\text{(\theequation)}\label{ap59}] $\mathcal{A%
}$ is monotonic.
\end{itemize}

Thanks to \eqref{ap56}$\cdots $\eqref{ap59}, Corollary \ref{cor-to-transport}
in Section \ref{transport-lemma} produces a polynomial $P^{\prime }\in 
\mathcal{P}$ such that

\begin{itemize}
\item[\refstepcounter{equation}\text{(\theequation)}\label{ap60}] $\tg_{l\left( \mathcal{A}\right) -2}$ has an $\left( \mathcal{A},\epsilon
^{-1}\delta _{Q^{\prime }},C\right) $-basis at $\left( y,M_{0},\tau_0,P^{\prime
}\right) $;
\end{itemize}

\begin{itemize}
\item[\refstepcounter{equation}\text{(\theequation)}\label{ap61}] $\partial
^{\beta }\left( P^{\prime }-P^{y^{\prime }}\right) \equiv 0$ for $\beta \in 
\mathcal{A}$;
\end{itemize}

and

\begin{itemize}
\item[\refstepcounter{equation}\text{(\theequation)}\label{ap62}] $%
\left\vert \partial ^{\beta }\left( P^{\prime }-P^{y^{\prime }}\right)
\left( y^{\prime }\right) \right\vert \leq CM_{0}\left( \epsilon ^{-1}\delta
_{Q^{\prime }}\right) ^{m-\left\vert \beta \right\vert }$ for $\beta \in 
\mathcal{M}$.
\end{itemize}

From \eqref{ap60} we have in particular that

\begin{itemize}
\item[\refstepcounter{equation}\text{(\theequation)}\label{ap63}] $P^{\prime
}\in \Gamma _{l\left( \mathcal{A}\right) -2}\left( y,CM_{0},\tau_0\right) $,
\end{itemize}

and from \eqref{ap62} and \eqref{ap54} we obtain

\begin{itemize}
\item[\refstepcounter{equation}\text{(\theequation)}\label{ap64}] $%
\left\vert \partial ^{\beta }\left( P^{y^{\prime }}-P^{\prime }\right)
\left( y^{\prime }\right) \right\vert \leq CM_{0}\left( \epsilon ^{-1}\delta
_{Q}\right) ^{m-\left\vert \beta \right\vert }$ for $\beta \in \mathcal{M}$.
\end{itemize}

If we knew that

\begin{itemize}
\item[\refstepcounter{equation}\text{(\theequation)}\label{ap65}] $%
\left\vert \partial ^{\beta }\left( P^{y}-P^{\prime }\right) \left( y\right)
\right\vert \leq M_{0}\left( \epsilon ^{-1}\delta _{Q}\right) ^{m-\left\vert
\beta \right\vert }$ for $\beta \in \mathcal{M}$,
\end{itemize}

then also $\left\vert \partial ^{\beta }\left( P^{y}-P^{\prime }\right)
\left( y^{\prime }\right) \right\vert \leq C^{\prime }M_{0}\left( \epsilon
^{-1}\delta _{Q}\right) ^{m-\left\vert \beta \right\vert }$ for $\beta \in 
\mathcal{M}$ since $\left\vert y-y^{\prime }\right\vert \leq C\delta _{Q}$
thanks to \eqref{ap50}, \eqref{ap51}, \eqref{ap54}. Consequently, by %
\eqref{ap64}, we would have $\left\vert \partial ^{\beta }\left(
P^{y^{\prime }}-P^{y}\right) \left( y^{\prime }\right) \right\vert \leq
CM_{0}\left( \epsilon ^{-1}\delta _{Q}\right) ^{m-\left\vert \beta
\right\vert }$ for $\beta \in \mathcal{M}$, which is our desired inequality %
\eqref{ap52}. Thus, Lemma \ref{lemma-ap2} will follow if we can prove %
\eqref{ap65}.

Suppose \eqref{ap65} fails. We will deduce a contradiction.

Corollary \ref{cor-to-lemma-ap1} shows that $\tg_{l\left( \mathcal{A%
}\right) -1}$ has an $\left( \mathcal{A},\epsilon ^{-1}\delta _{Q},C\right) $%
-basis at $\left( y,M_{0},\tau_0,P^{y}\right) $. Since $\Gamma _{l\left( \mathcal{A}%
\right) -1}\left( x,M,\tau\right) \subset \Gamma _{l\left( \mathcal{A}\right)
-2}\left( x,CM,C\tau\right) $ for all $x\in E_0$, $M>0$, it follows that
\begin{itemize}
\item[\refstepcounter{equation}\text{(\theequation)}\label{ap66}] $\vec{%
\Gamma}_{l\left( \mathcal{A}\right) -2}$ has an $\left( \mathcal{A},\epsilon
^{-1}\delta _{Q},C\right) $-basis at $\left( y,M_{0},\tau_{0},P^{y}\right) $.
\end{itemize}
\begin{remark}
  This small difference $\Gamma_{l(\A)-1}(x,M,\tau)\subset\G_{l(\A)-2}(x,CM,C\tau)$ instead of $\Gamma_{l(\A)-1}(x,M,\tau)\subset\G_{l(\A)-2}(x,M,\tau)$ (which would be the direct analogy from \cite{feffermanFinitenessPrinciplesSmooth2016}) doesn't affect the result, it just modifies $C$ in \eqref{ap66}.
\end{remark}

From \eqref{ap61} and \eqref{ap2} (applied to $y$ and $y^{\prime }$), we see
that

\begin{itemize}
\item[\refstepcounter{equation}\text{(\theequation)}\label{ap67}] $\partial
^{\beta }\left( P^{y}-P^{\prime }\right) \equiv 0$ for $\beta \in \mathcal{A}
$.
\end{itemize}

Since we are assuming that \eqref{ap65} fails, we have

\begin{itemize}
\item[\refstepcounter{equation}\text{(\theequation)}\label{ap68}] $%
\max_{\beta \in \mathcal{M}}\left( \epsilon ^{-1}\delta _{Q}\right)
^{\left\vert \beta \right\vert }\left\vert \partial ^{\beta }\left(
P^{y}-P^{\prime }\right) \left( y\right) \right\vert \geq M_{0}\left(
\epsilon ^{-1}\delta _{Q}\right) ^{m}$.
\end{itemize}

Also, from \eqref{ap54} and the hypotheses of the Main Lemma for $\mathcal{A}
$, we have

\begin{itemize}
\item[\refstepcounter{equation}\text{(\theequation)}\label{ap69}] $\epsilon
^{-1}\delta _{Q}<\epsilon ^{-1}\delta _{Q_{0}}\leq \delta _{\max }$.
\end{itemize}

But we know that

\begin{itemize}
\item[\refstepcounter{equation}\text{(\theequation)}\label{ap70}] $\tg_{l\left( \mathcal{A}\right) -2}$ is $\left( C,\delta _{\max }\right) 
$-convex.
\end{itemize}

Our results \eqref{ap63}, \eqref{ap66}$\cdots$\eqref{ap70} and Lemma \ref%
{lemma-pb3} produce a set $\hat{\mathcal{A}}\subseteq \mathcal{M}$ and a
polynomial $\hat{P}\in \mathcal{P}$, with the following properties:

\begin{itemize}
\item[\refstepcounter{equation}\text{(\theequation)}\label{ap71}] $\hat{%
\mathcal{A}}$ is monotonic;
\end{itemize}

\begin{itemize}
\item[\refstepcounter{equation}\text{(\theequation)}\label{ap72}] $\hat{%
\mathcal{A}}<\mathcal{A}$ (strict inequality);
\end{itemize}

\begin{itemize}
\item[\refstepcounter{equation}\text{(\theequation)}\label{ap73}] $\tg_{l\left( \mathcal{A}\right) -2}$ has an $\left( \hat{\mathcal{A}}%
,\epsilon ^{-1}\delta _{Q},C\right) $-basis at $\left( y,M_{0},\tau_{0},\hat{P}%
\right) $;
\end{itemize}

\begin{itemize}
\item[\refstepcounter{equation}\text{(\theequation)}\label{ap74}] $\partial
^{\beta }\left( \hat{P}-P^{y}\right) \equiv 0$ for $\beta \in \mathcal{A}$
(recall, $\mathcal{A}$ is monotonic);
\end{itemize}

and

\begin{itemize}
\item[\refstepcounter{equation}\text{(\theequation)}\label{ap75}] $%
\left\vert \partial ^{\beta }\left( \hat{P}-P^{y}\right) \left( y\right)
\right\vert \leq CM\left( \epsilon ^{-1}\delta _{Q}\right) ^{m-\left\vert
\beta \right\vert }$ for $\beta \in \mathcal{M}$.
\end{itemize}

Now let $z\in E_0\cap 5Q^{+}$. We recall that $\mathcal{A}$ is monotonic, and
that \eqref{ap66}, \eqref{ap67}, \eqref{ap73}, \eqref{ap74}, \eqref{ap75} hold. Moreover,
since $y,z\in 5Q^{+}$, we have $\left\vert y-z\right\vert \leq C\delta
_{Q}=C\epsilon \left( \epsilon ^{-1}\delta _{Q}\right) $. Thanks to the
above remarks and the Small $\epsilon $ Assumption, we may apply Lemma \ref%
{lemma-transport} to produce $\check{P}^{z}\in \mathcal{P}$ satisfying the
following conditions.

\begin{itemize}
\item[\refstepcounter{equation}\text{(\theequation)}\label{ap76}] $\tg_{l\left( \mathcal{A}\right) -3}$ has an $\left( \hat{\mathcal{A}}%
,\epsilon ^{-1}\delta _{Q},C\right) $-basis at $\left( z,M_{0},\tau_0,\check{P}%
^{z}\right) $.
\end{itemize}

\begin{itemize}
\item[\refstepcounter{equation}\text{(\theequation)}\label{ap77}] $\partial
^{\beta }\left( \check{P}^{z}-P^{y}\right) \equiv 0$ for $\beta \in \mathcal{%
A}$.
\end{itemize}

\begin{itemize}
\item[\refstepcounter{equation}\text{(\theequation)}\label{ap78}] $%
\left\vert \partial ^{\beta }\left( \check{P}^{z}-P^{y}\right) \left(
y\right) \right\vert \leq CM_{0}\left( \epsilon ^{-1}\delta _{Q}\right)
^{m-\left\vert \beta \right\vert }$ for $\beta \in \mathcal{M}$.
\end{itemize}

By \eqref{ap2} and \eqref{ap77}, we have

\begin{itemize}
\item[\refstepcounter{equation}\text{(\theequation)}\label{ap80}] $\partial
^{\beta }\left( \check{P}^{z}-P^{0}\right) \equiv 0$ for $\beta \in \mathcal{%
A}$.
\end{itemize}

By \eqref{ap54} and \eqref{ap78}, we have $\left\vert \partial ^{\beta
}\left( \check{P}^{z}-P^{y}\right) \left( y\right) \right\vert \leq
CM_{0}\left( \epsilon ^{-1}\delta _{Q_{0}}\right) ^{m-\left\vert \beta
\right\vert }$ for $\beta \in \mathcal{M}$, hence $\left\vert \partial
^{\beta }\left( \check{P}^{z}-P^{y}\right) \left( x_{0}\right) \right\vert
\leq CM_{0}\left( \epsilon ^{-1}\delta _{Q_{0}}\right) ^{m-\left\vert \beta
\right\vert }$ for $\beta \in \mathcal{M}$, since $x,y\in 5Q_{0}^{+}$.
Together with \eqref{ap3}, this yields the
estimate

\begin{itemize}
\item[\refstepcounter{equation}\text{(\theequation)}\label{ap81}] $%
\left\vert \partial ^{\beta }\left( \check{P}^{z}-P^{0}\right) \left(
x_{0}\right) \right\vert \leq CM_{0}\left( \epsilon ^{-1}\delta
_{Q_{0}}\right) ^{m-\left\vert \beta \right\vert }$ for $\beta \in \mathcal{M%
}$.
\end{itemize}

We have proven \eqref{ap76}, \eqref{ap80}, \eqref{ap81} for each $z\in E_0\cap
5Q^{+}$. Thus, $5Q^{+}\subset 5Q_{0}$ (see \eqref{ap55}), $\hat{\mathcal{A}}<%
\mathcal{A}$ (strict inequality; see \eqref{ap72}), and for each $z\in E_0\cap
5Q^{+}$ there exists $\check{P}^{z}\in \mathcal{P}$ such that

\begin{itemize}
\item $\tg_{l\left( \mathcal{A}\right) -3}$ has an $\left( \hat{%
\mathcal{A}},\epsilon ^{-1}\delta _{Q^{+}},C\right) $-basis at $\left(
z,M_{0},\tau_0,\check{P}^{z}\right) $;

\item $\partial ^{\beta }\left( \check{P}^{z}-P^{0}\right) \equiv 0$ for $%
\beta \in \mathcal{A}$; and

\item $\left\vert \partial ^{\beta }\left( \check{P}^{z}-P^{0}\right) \left(
x_{0}\right) \right\vert \leq CM_{0}\left( \epsilon ^{-1}\delta
_{Q_{0}}\right) ^{m-\left\vert \beta \right\vert }$ for $\beta \in \mathcal{M%
}$. (See \eqref{ap74}, \eqref{ap80}, \eqref{ap81}.)
\end{itemize}

We can apply now Lemma \ref{lem:cz-equiv}, and we see that $%
Q^{+}$ is OK. On the other hand $Q^{+}$ cannot be OK, since it properly
contains the CZ cube $Q$. Assuming that \eqref{ap65} fails, we have derived
a contradiction. Thus, \eqref{ap65} holds, completing the proof of Lemma \ref%
{lemma-ap2}.
\end{proof}

\section{Good News About CZ Cubes}

\label{gn}

In this section we again place ourselves in the setting of Section \ref%
{setup-for-the-induction-step}, and we make use of the auxiliary polynomials 
$P^{y}$ and the CZ cubes $Q$ defined above.

\begin{lemma}
\label{lemma-gn1} Let $Q \in$ CZ, with

\begin{itemize}
\item[\refstepcounter{equation}\text{(\theequation)}\label{gn1}] $\frac{65}{%
64}Q\cap \frac{65}{64}Q_{0}\not=\emptyset $
\end{itemize}

and

\begin{itemize}
\item[\refstepcounter{equation}\text{(\theequation)}\label{gn2}] $\#\left(
E_0\cap 5Q\right) \geq 2$.
\end{itemize}

Let

\begin{itemize}
\item[\refstepcounter{equation}\text{(\theequation)}\label{gn3}] $y\in E_0\cap
5Q$.
\end{itemize}

Then there exist a set $\mathcal{A}^\# \subseteq \mathcal{M}$ and a
polynomial $P^\# \in \mathcal{P}$ with the following properties.

\begin{itemize}
\item[\refstepcounter{equation}\text{(\theequation)}\label{gn4}] $\mathcal{A}%
^\#$ is monotonic.
\end{itemize}

\begin{itemize}
\item[\refstepcounter{equation}\text{(\theequation)}\label{gn5}] $\mathcal{A}%
^\# < \mathcal{A}$ (strict inequality).
\end{itemize}

\begin{itemize}
\item[\refstepcounter{equation}\text{(\theequation)}\label{gn6}] $\tg_{l\left( \mathcal{A}\right) -3}$ has an $\left( \mathcal{A}^{\#},\epsilon
^{-1}\delta _{Q},C\left( A\right) \right) $-basis at $\left(
y,M_{0},\tau_0,P^{\#}\right) $.
\end{itemize}

\begin{itemize}
\item[\refstepcounter{equation}\text{(\theequation)}\label{gn7}] $\left\vert
\partial ^{\beta }\left( P^{\#}-P^{y}\right) \left( y\right) \right\vert
\leq C\left( A\right) M_{0}\left( \epsilon ^{-1}\delta _{Q}\right)
^{m-\left\vert \beta \right\vert }$ for $\beta \in \mathcal{M}$.
\end{itemize}
\end{lemma}

\begin{proof}
Recall that

\begin{itemize}
\item[\refstepcounter{equation}\text{(\theequation)}\label{gn8}] $\partial
^{\beta }\left( P^{y}-P^{0}\right) \equiv 0$ for $\beta \in \mathcal{A}$
(see \eqref{ap2} in Section \ref{auxiliary-polynomials})
\end{itemize}

and that

\begin{itemize}
\item[\refstepcounter{equation}\text{(\theequation)}\label{gn9}] $5Q
\subseteq 5Q_0$, since $Q$ is OK.
\end{itemize}

Thanks to \eqref{gn3} and \eqref{gn9}, Corollary \ref{cor-to-lemma-ap1} in
Section \ref{auxiliary-polynomials} applies, and it tells us that

\begin{itemize}
\item[\refstepcounter{equation}\text{(\theequation)}\label{gn10}] $\tg_{l\left( \mathcal{A}\right) -1}$ has an $\left( \mathcal{A},\epsilon
^{-1}\delta _{Q},C\right) $-basis at $\left( y,M_{0},\tau_0,P^{y}\right) $.
\end{itemize}

On the other hand, $Q$ is OK and $\#(E\cap 5Q) \geq 2$; hence by Lemma \ref{lem:ok-basis}, there exist $%
\hat{\mathcal{A}} \subseteq \mathcal{M}$ and $\hat{P} \in \mathcal{P}$ with
the following properties

\begin{itemize}
\item[\refstepcounter{equation}\text{(\theequation)}\label{gn11}] $\tg_{l\left( \mathcal{A}\right) -3}$ has a weak $\left( \hat{\mathcal{A}}%
,\epsilon ^{-1}\delta _{Q},CA\right) $-basis at $\left( y,M_{0},\tau_0,\hat{P}%
\right) $.
\end{itemize}

\begin{itemize}
\item[\refstepcounter{equation}\text{(\theequation)}\label{gn12}] $%
\left\vert \partial ^{\beta }\left( \hat{P}-P^{0}\right) \left( x_{0}\right)
\right\vert \leq AM_{0}\left( \epsilon ^{-1}\delta _{Q_{0}}\right)
^{m-\left\vert \beta \right\vert }$ for $\beta \in \mathcal{M}$.
\end{itemize}

\begin{itemize}
\item[\refstepcounter{equation}\text{(\theequation)}\label{gn13}] $\partial
^{\beta }\left( \hat{P}-P^{0}\right) \equiv 0$ for $\beta \in \mathcal{A}$.
\end{itemize}

\begin{itemize}
\item[\refstepcounter{equation}\text{(\theequation)}\label{gn14}] $\hat{%
\mathcal{A}} < \mathcal{A}$ (strict inequality).
\end{itemize}

We consider separately two cases.

\underline{Case 1:} Suppose that

\begin{itemize}
\item[\refstepcounter{equation}\text{(\theequation)}\label{gn15}] $%
\left\vert \partial ^{\beta }\left( \hat{P}-P^{y}\right) \left( y\right)
\right\vert \leq M_{0}\left( \epsilon ^{-1}\delta _{Q}\right) ^{m-\left\vert
\beta \right\vert }$ for $\beta \in \mathcal{M}$.
\end{itemize}

The properties of approximate refinements guarantee that

\begin{itemize}
\item[\refstepcounter{equation}\text{(\theequation)}\label{gn16}] $\tg_{l\left( \mathcal{A}\right) -3}$ is $\left( C,\delta _{\max }\right) 
$-convex.
\end{itemize}

Also, \eqref{gn9} and hypothesis (A2) of the Main Lemma for $\mathcal{A}$
give

\begin{itemize}
\item[\refstepcounter{equation}\text{(\theequation)}\label{gn17}] $\epsilon
^{-1}\delta _{Q}\leq \epsilon ^{-1}\delta _{Q_{0}}\leq \delta _{\max }$.
\end{itemize}

Applying \eqref{gn11}, \eqref{gn16}, \eqref{17}, and Lemma \ref{lemma-pb2},
we obtain a set $\mathcal{A}^\# \subseteq \mathcal{M}$ such that

\begin{itemize}
\item[\refstepcounter{equation}\text{(\theequation)}\label{gn18}] $\mathcal{A%
}^{\#}\leq \hat{\mathcal{A}}$,
\end{itemize}

\begin{itemize}
\item[\refstepcounter{equation}\text{(\theequation)}\label{gn19}] $
\mathcal{A}^\#$ is monotonic,
\end{itemize}

and

\begin{itemize}
\item[\refstepcounter{equation}\text{(\theequation)}\label{gn20}] $\tg_{l\left( \mathcal{A}\right) -3}$ has an $\left( \mathcal{A}%
^{\#},\epsilon ^{-1}\delta _{Q},C\left( A\right) \right) $-basis at $\left(
y,M_{0},\tau_0,\hat{P}\right) $.
\end{itemize}

Setting $P^\# =\hat{P}$, we obtain the desired conclusions \eqref{gn4}$\cdots
$\eqref{gn7} at once from \eqref{gn14}, \eqref{gn15}, \eqref{gn18}, %
\eqref{gn19}, and \eqref{gn20}.

Thus, Lemma \ref{lemma-gn1} holds in Case 1.

\underline{Case 2:} Suppose that $\left\vert \partial ^{\beta }\left( \hat{P}%
-P^{y}\right) \left( y\right) \right\vert >M_{0}\left( \epsilon ^{-1}\delta
_{Q}\right) ^{m-\left\vert \beta \right\vert }$ for some $\beta \in \mathcal{%
M}$, i.e.,

\begin{itemize}
\item[\refstepcounter{equation}\text{(\theequation)}\label{gn21}] $%
\max_{\beta \in \mathcal{M}}\left( \epsilon ^{-1}\delta _{Q}\right)
^{\left\vert \beta \right\vert }\left\vert \partial ^{\beta }\left( \hat{P}%
-P^{y}\right) \left( y\right) \right\vert >M_{0}\left( \epsilon ^{-1}\delta
_{Q}\right) ^{m}$.
\end{itemize}

From \eqref{gn11} we have

\begin{itemize}
\item[\refstepcounter{equation}\text{(\theequation)}\label{gn22}] $\hat{P}%
\in \Gamma _{l\left( \mathcal{A}\right) -3}\left( y,CAM_{0},CA\tau_0\right) $
\end{itemize}

Since $\Gamma _{l(\mathcal{A})-1}(x,M,\tau)\subseteq \Gamma _{l\left( \mathcal{A}%
\right) -3}\left( x,CM,C\tau\right) $ for all $x\in E,M>0$, \eqref{gn10} implies
that

\begin{itemize}
\item[\refstepcounter{equation}\text{(\theequation)}\label{gn23}] $\tg_{l\left( \mathcal{A}\right) -3}$ has an $\left( \mathcal{A},\epsilon
^{-1}\delta _{Q},C\right) $-basis at $\left( y,M_{0},\tau_0,P^{y}\right) $.
\end{itemize}

As in Case 1,

\begin{itemize}
\item[\refstepcounter{equation}\text{(\theequation)}\label{gn24}] $\vec{%
\Gamma}_{l\left( \mathcal{A}\right) -3}$ is $\left( C,\delta _{\max }\right) 
$-convex,
\end{itemize}

and

\begin{itemize}
\item[\refstepcounter{equation}\text{(\theequation)}\label{gn25}] $\epsilon
^{-1}\delta _{Q}\leq \epsilon ^{-1}\delta _{Q_{0}}\leq \delta _{\max }$.
\end{itemize}

From \eqref{gn8} and \eqref{gn13} we have

\begin{itemize}
\item[\refstepcounter{equation}\text{(\theequation)}\label{gn26}] $\partial
^{\beta }\left( \hat{P}-P^{y}\right) \equiv 0$ for $\beta \in \mathcal{A}$.
\end{itemize}

Thanks to \eqref{gn21}$\cdots$\eqref{gn26} and Lemma \ref{lemma-pb3} there exist 
$\mathcal{A}^\# \subseteq \mathcal{M}$ and $P^\# \in \mathcal{P}$ with the
following properties: $\mathcal{A}^\#$ is monotonic; $\mathcal{A}^\# <%
\mathcal{A}$ (strict inequality); $\tg_{l(\mathcal{A})-3}$ has an $(%
\mathcal{A}^\#,\epsilon^{-1}\delta_{Q},C(A))$-basis at $(y,M_0,\tau_0,P^\#)$; $%
\partial^\beta(P^\#-P^y)\equiv 0$ for $\beta \in \mathcal{A}$; $%
|\partial^\beta (P^\# -P^y)(y)|\leq M_0 (\epsilon^{-1}\delta_Q)^{m-|\beta|}$
for $\beta \in \mathcal{M}$.

Thus, $\mathcal{A}^\#$ and $P^\#$ satisfy \eqref{gn4}$\cdots$\eqref{gn7},
proving Lemma \ref{lemma-gn1} in Case 2.

We have seen that Lemma \ref{lemma-gn1} holds in all cases.
\end{proof}

\begin{remarks}

\begin{itemize} 
\item 	We will need to find the polynomial $P^{\#}$ from Lemma \ref{lemma-gn1} in the main algorithm. We can do so by solving a linear programming problem with dimension and number of constraints bounded by a constant depending on $n,m$; and we know a solution exists.
	
\item Once again, the fact that the approximate refinements don't satisfy $\Gamma_{l+1}(x, M, \tau) \subset \Gamma_{l}(x, M, \tau)$ but instead $\Gamma_{l+1}(x, M, \tau) \subset \Gamma_l(x, CM, C\tau)$ doesn't affect the fact that previous refinements also have a basis, it only affects the constant $C$ of such a basis.

\item The proof of Lemma \ref{lemma-gn1} gives a $\hat{P}$ that satisfies
also $\partial^\beta(\hat{P}-P^0)\equiv 0$ for $\beta \in \mathcal{A}$, but
we make no use of that.

\item Note that $x_0$ and $\delta_{Q_0}$ appear in \eqref{gn12}, rather than
the desired $y, \delta_Q$. Consequently, \eqref{gn12} is of no help in the
proof of Lemma \ref{lemma-gn1}.
\end{itemize}
\end{remarks}

In the proof of our next result, we use our Induction Hypothesis that the
Main Lemma for $\mathcal{A}^{\prime }$ holds whenever $\mathcal{A}^{\prime }<%
\mathcal{A}$ and $\mathcal{A}^{\prime }$ is monotonic. (See Section \ref%
{setup-for-the-induction-step}.)

\begin{lemma}
\label{lemma-gn2}

Let $Q\in $ CZ. Suppose that

\begin{itemize}
\item[\refstepcounter{equation}\text{(\theequation)}\label{gn27}] $\frac{65}{%
64}Q\cap \frac{65}{64}Q_{0}\not=\emptyset $
\end{itemize}

and

\begin{itemize}
\item[\refstepcounter{equation}\text{(\theequation)}\label{gn28}] $\#\left(
E_0\cap 5Q\right) \geq 2$.
\end{itemize}

Let

\begin{itemize}
\item[\refstepcounter{equation}\text{(\theequation)}\label{gn29}] $y\in
E_0\cap 5Q$. If $\#(E_0\cap \frac{65}{64}Q) > 0$, assume $y \in E_0 \cap \frac{65}{64}Q$.
\end{itemize}

Then there exists $F^{y,Q} \in C^m(\frac{65}{64}Q)$ such that

\begin{itemize}
\item[(*1)] $\left\vert \partial ^{\beta }\left( F^{y,Q}-P^{y}\right)
\right\vert \leq C\left( \epsilon \right) M_{0}\delta _{Q}^{m-\left\vert
\beta \right\vert }$ on $\frac{65}{64}Q$, for $\left\vert \beta \right\vert
\leq m$; and

\item[(*2)] $J_{z}\left( F^{y,Q}\right) \in \Gamma _{0}\left( z,C\left(
\epsilon \right) M_{0}, \tau_0\right) $ for all $z\in E\cap \frac{65}{64}Q$.
\end{itemize}
\end{lemma}

\begin{proof}
Our hypotheses \eqref{gn27}, \eqref{gn28}, \eqref{gn29} imply the
hypotheses of Lemma \ref{lemma-gn1} (\eqref{gn29} is stronger than the corresponding hypothesis of Lemma \ref{lemma-gn1}). Let $\mathcal{A}^{\#}$, $P^{\#}$
satisfy the conclusions \eqref{gn4}$\cdots $\eqref{gn7} of that Lemma.

Thanks to conclusion \eqref{gn7} of Lemma \ref{lemma-gn1} (together with %
\eqref{gn29}), we have

\begin{itemize}
	\item[\refstepcounter{equation}\text{(\theequation)}\label{gn34}] $%
	\left\vert \partial ^{\beta }\left( P^{\#}-P^{y}\right) \right\vert \leq
	C\left( \epsilon \right) M_{0}\delta _{Q}^{m-\left\vert \beta \right\vert }$
	on $\frac{65}{64}Q$ for $\left\vert \beta \right\vert \leq m$.
\end{itemize}

(Recall that $P^\#-P^y$ is a polynomial of degree at most $m-1$.)

We distinguish two cases:

\underline{Case 1}. Suppose $\#(E_0 \cap \frac{65}{64}Q_0)>0$.

Recall the definition of $l(\mathcal{A})$; see \eqref{gn1}, \eqref{gn2} in
Section \ref{statement-of-the-main-lemma}. We have $l(\mathcal{A}^{\#})\leq
l(\mathcal{A})-3$ since $\mathcal{A}^{\#}<\mathcal{A}$; hence \eqref{gn6}
implies that

\begin{itemize}
\item[\refstepcounter{equation}\text{(\theequation)}\label{gn30}] $\tg_{l\left( \mathcal{A}^{\#}\right) }$ has an $\left( \mathcal{A}%
^{\#},\epsilon ^{-1}\delta _{Q},C\left( A\right) \right) $-basis at $\left(
y,M_{0},\tau_0, P^{\#}\right) $.
\end{itemize}

Also, since $Q$ is OK, we have $5Q \subseteq 5Q_0$, hence $\delta_Q \leq
\delta_{Q_0}$. Hence, hypothesis (A2) of the Main Lemma for $\mathcal{A}$
implies that

\begin{itemize}
\item[\refstepcounter{equation}\text{(\theequation)}\label{gn31}] $\epsilon
^{-1}\delta _{Q}\leq \delta _{\max }$.
\end{itemize}

By \eqref{gn4}, \eqref{gn5}, and our Inductive Hypothesis, the Main Lemma
holds for $\mathcal{A}^{\#}$. Thanks to \eqref{gn29}, \eqref{gn30}, %
\eqref{gn31} and the Small $\epsilon $ Assumption in Section \ref%
{setup-for-the-induction-step}, the Main Lemma for $\mathcal{A}^{\#}$ now
yields a function $F\in C^{m}\left( \frac{65}{64}Q\right) $, such that

\begin{itemize}
\item[\refstepcounter{equation}\text{(\theequation)}\label{gn32}] $%
\left\vert \partial ^{\beta }\left( F-P^{\#}\right) \right\vert \leq C\left(
\epsilon \right) M_{0}\delta_Q ^{m-\left\vert \beta \right\vert }$ on $\frac{65%
}{64}Q$, for $\left\vert \beta \right\vert \leq m$; and
\end{itemize}

\begin{itemize}
\item[\refstepcounter{equation}\text{(\theequation)}\label{gn33}] $%
J_{z}\left( F\right) \in \Gamma _{0}\left( z,C\left( \epsilon \right)
M_{0},C(\epsilon)\tau_0\right) $ for all $z\in E\cap \frac{65}{64}Q$.
\end{itemize}

 Taking $%
F^{y,Q}=F$, we may read off the desired conclusions (*1) and (*2) from %
\eqref{gn32}, \eqref{gn33}, \eqref{gn34}.

\underline{Case 2}. Suppose $\#(E_0 \cap \frac{65}{64}Q_0) = 0$. Take $F^{y,Q} = P^\#$. Then \eqref{gn34} implies the conclusion (*1), and conclusion (*2) holds vacuously.

The proof of Lemma \ref{lemma-gn2} is complete.
\end{proof}

\section{Local Interpolants}
\label{sec:locint}
In this section, we again place ourselves in the setting of Section \ref%
{setup-for-the-induction-step}. We make use of the Calder\'{o}n-Zygmund
cubes $Q$ and the auxiliary polynomials $P^{y}$ defined above. Let

\begin{itemize}
\item[\refstepcounter{equation}\text{(\theequation)}\label{li0}] $\mathcal{Q}%
=\left\{ Q\in CZ:\frac{65}{64}Q\cap \frac{65}{64}Q_{0}\not=\emptyset
\right\} $.
\end{itemize}

For each $Q\in \mathcal{Q}$, we define a function $F^{Q}\in C^{m}\left( 
\frac{65}{64}Q\right) $ and a polynomial $P^{Q}\in \mathcal{P}$. We proceed
by cases. We say that $Q \in \mathcal{Q}$ is

\begin{description}
\item[Type 1] if $\# (E_0 \cap 5Q) \geq 2$,

\item[Type 2] if $\# (E_0 \cap 5Q) = 1$,

\item[Type 3] if $\#(E_0\cap 5Q)=0$ and $\delta _{Q}\leq \frac{1}{1024}\delta
_{Q_{0}}$, and

\item[Type 4] if $\#(E_0\cap 5Q)=0$ and $\delta _{Q}>\frac{1}{1024}\delta
_{Q_{0}}$.
\end{description}

\underline{If $Q$ is of Type 1}, then:
\begin{itemize}
	\item If $\# (E_0 \cap \frac{65}{64}Q) \geq 1$, we pick a point  $y_{Q}\in E_0\cap \frac{65}{64}Q$,
	and set $P^{Q}=P^{y_{Q}}$. Applying Lemma \ref{lemma-gn2}, we obtain a
	function $F^{Q}\in C^{m}\left( \frac{65}{64}Q\right) $ such that
	
	\begin{itemize}
		\item[\refstepcounter{equation}\text{(\theequation)}\label{li1}] $\left\vert
		\partial ^{\beta }\left( F^{Q}-P^{Q}\right) \right\vert \leq C\left(
		\epsilon \right) M_{0}\delta _{Q}^{m-\left\vert \beta \right\vert }$ on $%
		\frac{65}{64}Q$, for $\left\vert \beta \right\vert \leq m$; and
	\end{itemize}
	
	\begin{itemize}
		\item[\refstepcounter{equation}\text{(\theequation)}\label{li2}] $%
		J_{z}\left( F^{Q}\right) \in \Gamma _{0}\left( z,C\left( \epsilon \right)
		M_{0},C(\epsilon)\tau_0\right) $ for all $z\in E_0\cap \frac{65}{64}Q$.
	\end{itemize}

	\item Otherwise, we pick a point $y_Q \in E_0 \cap 5Q$ and set $F^Q = P^Q = P^{y_Q}$. Then \eqref{li1} holds trivially and \eqref{li2} holds vacuously.
	
\end{itemize}

\underline{If $Q$ is of Type 2}, then we let $y_{Q}$ be the one and only
point of $E_0\cap 5Q$, and define $F^{Q}=P^{Q}=P^{y_{Q}}$. Then \eqref{li1}
holds trivially. If $y_{Q}\not\in \frac{65}{64}Q$ then \eqref{li2} holds
vacuously.

If $y_{Q}\in \frac{65}{64}Q$, then \eqref{li2} asserts that $P^{y_{Q}}\in
\Gamma _{0}\left( y_{Q},C\left( \epsilon \right) M_{0},\tau_0\right) $. Thanks to %
\eqref{li1} in Section \ref{auxiliary-polynomials}, we know that $%
P^{y_{Q}}\in \Gamma _{l\left( \mathcal{A}\right) -1}\left(
y_{Q},CM_{0},C\tau_0\right) \subset \Gamma _{0}\left( y_{Q},C\left( \epsilon \right)
M_{0},C(\epsilon)\tau_0\right) $. Thus, \eqref{li1} and \eqref{li2} hold also when $Q$ is of
Type 2.

\underline{If $Q$ is of Type 3}, then $5Q^{+}\subset 5Q_{0}$, since $\frac{65%
}{64}Q\cap \frac{65}{64}Q_{0}\not=\emptyset $ and $\delta _{Q}\leq \frac{1}{%
1024}\delta _{Q_{0}}$. However, $Q^{+}$ cannot be OK, since $Q$ is a CZ
cube. Therefore $\#\left( E_0\cap 5Q^{+}\right) \geq 2$. We pick $y_{Q}\in
E\cap 5Q^{+}$, and set $F^{Q}=P^{Q}=P^{y_{Q}}$. Then \eqref{li1} holds
trivially, and \eqref{li2} holds vacuously.

\underline{If $Q$ is of Type 4}, then we set $F^{Q}=P^{Q}=P^{0}$, and again %
\eqref{li1} holds trivially, and \eqref{li2} holds vacuously.

Note that if $Q$ is of Type 1, 2, or 3, then we have defined a point $y_{Q}$%
, and we have $P^Q=P^{y_Q}$ and 

\begin{itemize}
\item[\refstepcounter{equation}\text{(\theequation)}\label{li3}] $y_{Q}\in
E_0\cap 5Q^{+}\cap 5Q_{0}$.
\end{itemize}

(If $Q$ is of Type 1 or 2, then $y_{Q}\in E_0\cap 5Q$ and $5Q\subseteq 5Q_{0}$
since $Q$ is OK. If $Q$ is of Type 3, then $y_{Q}\in E_0\cap 5Q^{+}$ and $%
5Q^{+}\subset 5Q_{0}$). We have picked $F^{Q}$ and $P^{Q}$ for all $Q\in 
\mathcal{Q}$, and \eqref{li1}, \eqref{li2} hold in all cases.

\begin{lemma}[\textquotedblleft Consistency of the $P^{Q}$\textquotedblright 
]
\label{lemma-li1} Let $Q,Q^{\prime }\in \mathcal{Q}$, and suppose $\frac{65}{%
64}Q\cap \frac{65}{64}Q^{\prime }\not=\emptyset $. Then

\begin{itemize}
\item[\refstepcounter{equation}\text{(\theequation)}\label{li4}] $\left\vert
\partial ^{\beta }\left( P^{Q}-P^{Q^{\prime }}\right) \right\vert \leq
C\left( \epsilon \right) M_{0}\delta _{Q}^{m-\left\vert \beta \right\vert }$
on $\frac{65}{64}Q\cap \frac{65}{64}Q^{\prime }$, for $\left\vert \beta
\right\vert \leq m$.
\end{itemize}
\end{lemma}

\begin{proof}
Suppose first that neither $Q$ nor $Q^{\prime }$ is Type 4. Then $%
P^{Q}=P^{y_{Q}}$ and $P^{Q^{\prime }}=P^{y_{Q^{\prime }}}$ with $y_{Q}\in
E_0\cap 5Q^{+}\cap 5Q_{0}$, $y_{Q^{\prime }}\in E_0\cap 5\left( Q^{\prime
}\right) ^{+}\cap 5Q_{0}$. Thanks to Lemma \ref{lemma-ap2}, we have%
\begin{equation*}
\left\vert \partial ^{\beta }\left( P^{Q}-P^{Q^{\prime }}\right) \left(
y_{Q}\right) \right\vert \leq C\left( \epsilon \right) M_{0}\delta
_{Q}^{m-\left\vert \beta \right\vert }\text{ for }\beta \in \mathcal{M}\text{%
,}
\end{equation*}
which implies \eqref{li4}, since $y_{Q}\in 5Q^{+}$ and $P^{Q}-P^{Q^{\prime
}} $ is an $(m-1)^{rst}$ degree polynomial.

Next, suppose that $Q$ and $Q^{\prime }$ are both Type 4.

Then by definition $P^{Q}=P^{Q^{\prime }}=P^{0}$, and consequently %
\eqref{li4} holds trivially.

Finally, suppose that exactly one of $Q$, $Q^{\prime }$ is of Type 4.

Since $\delta _{Q}$ and $\delta _{Q^{\prime }}$, differ by at most a factor
of $2$, the cubes $Q$ and $Q^{\prime }$ may be interchanged without loss of
generality. Hence, we may assume that $Q^{\prime }$ is of Type 4 and $Q$ is
not. By definition of Type 4,

\begin{itemize}
\item[\refstepcounter{equation}\text{(\theequation)}\label{li5}] $\delta
_{Q^{\prime }}>\frac{1}{1024}\delta _{Q_{0}}$; hence also $\delta _{Q}\geq 
\frac{1}{1024}\delta _{Q_{0}}$,
\end{itemize}

since $\delta _{Q}$, $\delta _{Q^{\prime }}$, are powers of $2$ that differ
by at most a factor of $2$.

Since $Q^{\prime }$ is of Type 4 and $Q$ is not, we have $P^Q= P^{y_Q}$ and $%
P^{Q^{\prime }} = P^0$, with

\begin{itemize}
\item[\refstepcounter{equation}\text{(\theequation)}\label{li6}] $y_{Q}\in
E_0\cap 5Q^{+}\cap 5Q_{0}$.
\end{itemize}

Thus, in this case, \eqref{li4} asserts that

\begin{itemize}
\item[\refstepcounter{equation}\text{(\theequation)}\label{li7}] $\left\vert
\partial ^{\beta }\left( P^{y_{Q}}-P^{0}\right) \right\vert \leq C\left(
\epsilon \right) M_{0}\delta _{Q}^{m-\left\vert \beta \right\vert }$ on $%
\frac{65}{64}Q\cap \frac{65}{64}Q^{\prime }$, for $\left\vert \beta
\right\vert \leq m$.
\end{itemize}

However, by \eqref{li6} above, property \eqref{ap3} in Section \ref%
{auxiliary-polynomials} gives the estimate

\begin{itemize}
\item[\refstepcounter{equation}\text{(\theequation)}\label{li8}] $\left\vert
\partial ^{\beta }\left( P^{y_{Q}}-P^{0}\right) \left( x_{0}\right)
\right\vert \leq C\left( \epsilon \right) M_{0}\delta _{Q_{0}}^{m-\left\vert
\beta \right\vert }$ for $\left\vert \beta \right\vert \leq m-1$.
\end{itemize}

Recall from the hypotheses of the Main Lemma for $\mathcal{A}$ that $%
x_{0}\in \frac{65}{64}Q_0$. Since $P^{y_{Q}}-P^{0}$ is an $%
(m-1)^{rst}$ degree polynomial, we conclude from \eqref{li8} that

\begin{itemize}
\item[\refstepcounter{equation}\text{(\theequation)}\label{li9}] $\left\vert
\partial ^{\beta }\left( P^{y_{Q}}-P^{0}\right) \right\vert \leq C\left(
\epsilon \right) M_{0}\delta _{Q_{0}}^{m-\left\vert \beta \right\vert }$ on $%
5Q$, for $\left\vert \beta \right\vert \leq m$.
\end{itemize}

The desired inequality \eqref{li7} now follows from \eqref{li5} and %
\eqref{li9}. Thus, \eqref{li4} holds in all cases.

The proof of Lemma \ref{lemma-li1} is complete.
\end{proof}

From estimate \eqref{li1}, Lemma \ref{lemma-li1}, and Lemma \ref{lemma-cz2},
we immediately obtain the following.

\begin{corollary}
\label{cor-to-lemma-li1} Let $Q,Q^{\prime }\in \mathcal{Q}$ and suppose that 
$\frac{65}{64}Q\cap \frac{65}{64}Q^{\prime }\not=\emptyset $. Then

\begin{itemize}
\item[\refstepcounter{equation}\text{(\theequation)}\label{li10}] $%
\left\vert \partial ^{\beta }\left( F^{Q}-F^{Q^{\prime }}\right) \right\vert
\leq C\left( \epsilon \right) M_{0}\delta _{Q}^{m-\left\vert \beta
\right\vert }$ on $\frac{65}{64}Q\cap \frac{65}{64}Q^{\prime }$, for $%
\left\vert \beta \right\vert \leq m$.
\end{itemize}
\end{corollary}

Regarding the polynomials $P^{Q}$, we make the following simple observation.

\begin{lemma}
\label{lemma-li2} We have

\begin{itemize}
\item[\refstepcounter{equation}\text{(\theequation)}\label{li11}] $%
\left\vert \partial ^{\beta }\left( P^{Q}-P^{0}\right) \right\vert \leq
C\left( \epsilon \right) M_{0}\delta _{Q_{0}}^{m-\left\vert \beta
\right\vert }$ on $\frac{65}{64}Q$, for $\left\vert \beta \right\vert \leq m$
and $Q\in \mathcal{Q}$.
\end{itemize}
\end{lemma}

\begin{proof}
Recall that if $Q$ is of Type 1, 2, or 3, then $P^{Q}=P^{y_{Q}}$ for some $%
y_{Q}\in 5Q_{0}$. From estimate \eqref{ap3} in Section \ref%
{auxiliary-polynomials}, we know that

\begin{itemize}
\item[\refstepcounter{equation}\text{(\theequation)}\label{li12}] $%
\left\vert \partial ^{\beta }\left( P^{Q}-P^{0}\right) \left( x_{0}\right)
\right\vert \leq C\left( \epsilon \right) M_{0}\delta _{Q_{0}}^{m-\left\vert
\beta \right\vert }$ for $\left\vert \beta \right\vert \leq m-1$.
\end{itemize}

Since $x_{0}\in \frac{65}{64}Q_0$ (see the hypotheses of the Main Lemma for $%
\mathcal{A}$) and $P^{Q}-P^{0}$ is a polynomial of degree at most $m-1$, and
since $\frac{65}{64}Q\subset 5Q\subset 5Q_{0}$ (because $Q$ is OK), estimate %
\eqref{li12} implies the desired estimate \eqref{li11}.

If instead, $Q$ is of Type 4, then by definition $P^{Q}=P^{0}$, hence
estimate \eqref{li11} holds trivially.

Thus, \eqref{li11} holds in all cases.
\end{proof}

\begin{corollary}
\label{cor-to-lemma-li2} For $Q\in \mathcal{Q}$ and $|\beta |\leq m$, we
have $\left\vert \partial ^{\beta }\left( F^{Q}-P^{0}\right) \right\vert
\leq C\left( \epsilon \right) M_{0}\delta _{Q_{0}}^{m-\left\vert \beta
\right\vert }$ on $\frac{65}{64}Q$.
\end{corollary}

\begin{proof}
Recall that, since $Q$ is OK, we have $5Q\subset 5Q_{0}$. The desired
estimate therefore follows from estimates \eqref{li1} and \eqref{li11}.
\end{proof}

\section{Completing the Induction}

\label{completing-the-induction}

We again place ourselves in the setting of Section \ref%
{setup-for-the-induction-step}. We use the CZ cubes $Q$ and the functions $%
F^Q$ defined above. We recall several basic results from earlier sections.

\begin{itemize}
\item[\refstepcounter{equation}\text{(\theequation)}\label{ci1}] $\tg_{0}$ is a $\left( C,\delta _{\max }\right) $-convex blob field with blob constant $C_{\G}$.
\end{itemize}

\begin{itemize}
\item[\refstepcounter{equation}\text{(\theequation)}\label{ci2}] $\epsilon
^{-1}\delta _{Q_{0}}\leq \delta _{\max }$, hence $\epsilon ^{-1}\delta
_{Q}\leq \delta _{\max }$ for $Q\in $ CZ.
\end{itemize}

\begin{itemize}
\item[\refstepcounter{equation}\text{(\theequation)}\label{ci3}] The cubes $%
Q\in $ CZ partition the interior of $5Q_{0}$.
\end{itemize}

\begin{itemize}
\item[\refstepcounter{equation}\text{(\theequation)}\label{ci4}] For $%
Q,Q^{\prime }\in $ CZ, if $\frac{65}{64}Q\cap \frac{65}{64}Q^{\prime
}\not=\emptyset $, then $\frac{1}{2}\delta _{Q}\leq \delta _{Q^{\prime
}}\leq 2\delta _{Q}$.
\end{itemize}

Recall that

\begin{itemize}
\item[\refstepcounter{equation}\text{(\theequation)}\label{ci5}] $\mathcal{Q}%
=\left\{ Q\in CZ:\frac{65}{64}Q\cap \frac{65}{64}Q_{0}\not=\emptyset
\right\} $.
\end{itemize}

Then

\begin{itemize}
\item[\refstepcounter{equation}\text{(\theequation)}\label{ci6}] $\mathcal{Q}
$ is finite.
\end{itemize}

For each $Q \in \mathcal{Q}$, we have

\begin{itemize}
\item[\refstepcounter{equation}\text{(\theequation)}\label{ci7}] $F^{Q}\in
C^{m}\left( \frac{65}{64}Q\right) $,
\end{itemize}

\begin{itemize}
\item[\refstepcounter{equation}\text{(\theequation)}\label{ci8}] $%
J_{z}\left( F^{Q}\right) \in \Gamma _{0}\left( z,C\left( \epsilon \right)
M_{0}, C(\epsilon)\tau_0\right) $ for $z\in E_0\cap \frac{65}{64}Q$, and
\end{itemize}

\begin{itemize}
\item[\refstepcounter{equation}\text{(\theequation)}\label{ci9}] $\left\vert
\partial ^{\beta }\left( F^{Q}-P^{0}\right) \right\vert \leq C\left(
\epsilon \right) M_{0}\delta _{Q_{0}}^{m-\left\vert \beta \right\vert }$ on $%
\frac{65}{64}Q$, for $\left\vert \beta \right\vert \leq m$.
\end{itemize}

\begin{itemize}
\item[\refstepcounter{equation}\text{(\theequation)}\label{ci10}] For each $%
Q,Q^{\prime }\in \mathcal{Q}$, if $\frac{65}{64}Q\cap \frac{65}{64}Q^{\prime
}\not=\emptyset $, then $\left\vert \partial ^{\beta }\left(
F^{Q}-F^{Q^{\prime }}\right) \right\vert \leq C\left( \epsilon \right)
M_{0}\delta _{Q}^{m-\left\vert \beta \right\vert }$ on $\frac{65}{64}Q\cap 
\frac{65}{64}Q^{\prime }$, for $\left\vert \beta \right\vert \leq m$.
\end{itemize}

We introduce a Whitney partition of unity adapted to the cubes $Q\in $ CZ.
For each $Q\in $ CZ, let $\tilde{\theta}_{Q}\in C^{m}\left( \mathbb{R}%
^{n}\right) $ satisfy 
\begin{equation*}
\tilde{\theta}_{Q}=1\text{ on }Q\text{, support }\left( \tilde{\theta}%
_{Q}\right) \subset \frac{65}{64}Q\text{, }\left\vert \partial ^{\beta }%
\tilde{\theta}_{Q}\right\vert \leq C\delta _{Q}^{-\left\vert \beta
\right\vert }\text{ for }\left\vert \beta \right\vert \leq m\text{.}
\end{equation*}%
Setting $\theta _{Q}=\tilde{\theta}_{Q}\cdot \left( \sum_{Q^{\prime }\in
CZ}\left( \tilde{\theta}_{Q^{\prime }}\right) ^{2}\right) ^{-1/2}$, we see
that

\begin{itemize}
\item[\refstepcounter{equation}\text{(\theequation)}\label{ci11}] $\theta
_{Q}\in C^{m}\left( \mathbb{R}^{n}\right) $ for $Q\in $ CZ;
\end{itemize}

\begin{itemize}
\item[\refstepcounter{equation}\text{(\theequation)}\label{ci12}] support $%
\left( \theta _{Q}\right) \subset \frac{65}{64}Q$ for $Q\in $ CZ.
\end{itemize}

\begin{itemize}
\item[\refstepcounter{equation}\text{(\theequation)}\label{ci13}] $%
\left\vert \partial ^{\beta }\theta _{Q}\right\vert \leq C\delta
_{Q}^{-\left\vert \beta \right\vert }$ for $\left\vert \beta \right\vert
\leq m,Q\in $ CZ;
\end{itemize}

and $\sum_{Q\in CZ}\theta _{Q}^{2}=1$ on the interior of $5Q_{0}$, hence

\begin{itemize}
\item[\refstepcounter{equation}\text{(\theequation)}\label{ci14}] $%
\sum_{Q\in \mathcal{Q}}\theta _{Q}^{2}=1$ on $\frac{65}{64}Q_{0}$.
\end{itemize}

We define

\begin{itemize}
\item[\refstepcounter{equation}\text{(\theequation)}\label{ci15}] $%
F=\sum_{Q\in \mathcal{Q}}\theta _{Q}^{2}F^{Q}$.
\end{itemize}

For each $Q\in \mathcal{Q}$, \eqref{ci7}, \eqref{ci11}, \eqref{ci12} show
that $\theta _{Q}^{2}F^{Q}\in C^{m}\left( \mathbb{R}^{n}\right) $. Since
also $\mathcal{Q}$ is finite (see \eqref{ci6}), it follows that

\begin{itemize}
\item[\refstepcounter{equation}\text{(\theequation)}\label{ci16}] $F\in
C^{m}\left( \mathbb{R}^{n}\right) $.
\end{itemize}

Moreover, for any $x\in \frac{65}{64}Q_{0}$ and any $\beta $ of order $%
|\beta |\leq m$, we have

\begin{itemize}
\item[\refstepcounter{equation}\text{(\theequation)}\label{ci17}] $\partial
^{\beta }F\left( x\right) =\sum_{Q\in \mathcal{Q}\left( x\right) }\partial
^{\beta }\left\{ \theta _{Q}^{2}F^{Q}\right\} $, where
\end{itemize}

\begin{itemize}
\item[\refstepcounter{equation}\text{(\theequation)}\label{ci18}] $\mathcal{Q%
}\left( x\right) =\left\{ Q\in \mathcal{Q}:x\in \frac{65}{64}Q\right\} $.
\end{itemize}

Note that

\begin{itemize}
\item[\refstepcounter{equation}\text{(\theequation)}\label{ci19}] $\#\left( 
\mathcal{Q}\left( x\right) \right) \leq C$, by \eqref{ci4}.
\end{itemize}

Let $\hat{Q}$ be the CZ cube containing $x$. (There is one and only one such
cube, thanks to \eqref{ci3}; recall that we suppose that $x\in \frac{65}{64}%
Q_{0}$.) Then $\hat{Q}\in \mathcal{Q}(x)$, and \eqref{ci17} may be written
in the form

\begin{itemize}
\item[\refstepcounter{equation}\text{(\theequation)}\label{ci20}] $\partial
^{\beta }\left( F-P^{0}\right) \left( x\right) =\partial ^{\beta }\left( F^{%
\hat{Q}}-P^{0}\right) \left( x\right) +\sum_{Q\in \mathcal{Q}\left( x\right)
}\partial ^{\beta }\left\{ \theta _{Q}^{2}\cdot \left( F^{Q}-F^{\hat{Q}%
}\right) \right\} \left( x\right) $.
\end{itemize}

(Here we use \eqref{ci14}.) The first term on the right in \eqref{ci20} has
absolute value at most $C\left( \epsilon \right) M_{0}\delta
_{Q_{0}}^{m-\left\vert \beta \right\vert }$; see \eqref{ci9}. At most $C$
distinct cubes $Q$ enter into the second term on the right in \eqref{ci20};
see \eqref{ci19}. For each $Q\in \mathcal{Q}(x)$, we have 
\begin{equation*}
\left\vert \partial ^{\beta }\left\{ \theta _{Q}^{2}\cdot \left( F^{Q}-F^{%
\hat{Q}}\right) \right\} \left( x\right) \right\vert \leq C\left( \epsilon
\right) M_{0}\delta _{Q}^{m-\left\vert \beta \right\vert }\text{,}
\end{equation*}%
by \eqref{ci10} and \eqref{ci13}. Hence, for each $Q\in \mathcal{Q}(x)$, we
have 
\begin{equation*}
\left\vert \partial ^{\beta }\left\{ \theta _{Q}^{2}\cdot \left( F^{Q}-F^{%
\hat{Q}}\right) \right\} \left( x\right) \right\vert \leq C\left( \epsilon
\right) M_{0}\delta _{Q_{0}}^{m-\left\vert \beta \right\vert };
\end{equation*}%
see \eqref{ci3}.

The above remarks and \eqref{ci19}, \eqref{ci20} together yield the estimate

\begin{itemize}
\item[\refstepcounter{equation}\text{(\theequation)}\label{ci21}] $%
\left\vert \partial ^{\beta }\left( F-P^{0}\right) \right\vert \leq C\left(
\epsilon \right) M_{0}\delta _{Q_{0}}^{m-\left\vert \beta \right\vert }$ on $%
\frac{65}{64}Q_{0}$, for $\left\vert \beta \right\vert \leq m$.
\end{itemize}

Moreover, let $z\in E_0$. Then 
\begin{equation*}
J_{z}\left( F\right) =\sum_{Q\in \mathcal{Q}\left( z\right) }J_{z}\left(
\theta _{Q}\right) \odot _{z}J_{z}\left( \theta _{Q}\right) \odot
_{z}J_{z}\left( F^{Q}\right) \text{ (see \eqref{ci17});}
\end{equation*}%
\begin{equation*}
\left\vert \partial ^{\beta }\left[ J_{z}\left( \theta _{Q}\right) \right]
\left( z\right) \right\vert \leq C\delta _{Q}^{-\left\vert \beta \right\vert
}\text{ for }\left\vert \beta \right\vert \leq m-1\text{, }Q\in \mathcal{Q}%
\left( z\right) \text{ (see \eqref{ci13});}
\end{equation*}%
\begin{equation*}
\sum_{Q\in \mathcal{Q}\left( z\right) }\left[ J_{z}\left( \theta _{Q}\right) %
\right] \odot _{z}\left[ J_{z}\left( \theta _{Q}\right) \right] =1
\end{equation*}%
(see \eqref{ci14} and note that $J_{z}(\theta _{Q})=0$ for $Q\not\in 
\mathcal{Q}(z)$ by \eqref{ci12} and \eqref{ci18}); 
\begin{equation*}
J_{z}\left( F^{Q}\right) \in \Gamma _{0}\left( z,C\left( \epsilon \right)
M_{0}, C(\epsilon)\tau_0\right) \text{ for }Q\in \mathcal{Q}\left( z\right) \text{ (see %
\eqref{ci8});}
\end{equation*}%
\begin{equation*}
\left\vert \partial ^{\beta }\left\{ J_{z}\left( F^{Q}\right) -J_{z}\left(
F^{Q^{\prime }}\right) \right\} \left( z\right) \right\vert \leq C\left(
\epsilon \right) M_{0}\delta _{Q}^{m-\left\vert \beta \right\vert }
\end{equation*}%
for $\left\vert \beta \right\vert \leq m-1$, $Q,Q^{\prime }\in \mathcal{Q}%
\left( z\right) $ (see\eqref{ci10}); 
\begin{equation*}
\#\left( \mathcal{Q}\left( z\right) \right) \leq C\text{ (see \eqref{ci19});}
\end{equation*}%
\begin{equation*}
\delta _{Q}\leq \delta _{\max }\text{ (see \eqref{ci2});}
\end{equation*}%
$\tg_{0}$ is a $\left( C,\delta _{\max }\right) $-convex shape
field (see \eqref{ci1}), and recall that the $\delta_Q$ for $Q \in CZ$ differ by at most a factor of 2 for contiguous cubes. Recall that $E_0 = E \cap \frac{65}{64}Q_0$ (see Section  \ref{statement-of-the-main-lemma}). The above results, together with Lemma \ref%
{lemma-wsf2}, tell us that

\begin{itemize}
\item[\refstepcounter{equation}\text{(\theequation)}\label{ci22}] $%
J_{z}\left( F\right) \in \Gamma _{0}\left( z,C\left( \epsilon \right)
M_{0},C(\epsilon)\tau_0\right) $ for all $z\in E\cap \frac{65}{64}Q_{0}$.
\end{itemize}

From \eqref{ci16}, \eqref{ci21}, \eqref{ci22}, we see at once that the
restriction of $F$ to $\frac{65}{64}Q_{0}$ belongs to $C^{m}\left( \frac{65}{%
64}Q_{0}\right) $ and satisfies conditions (C*1) and (C*2) in Section \ref%
{setup-for-the-induction-step}. As we explained in that section, once we
have found a function in $C^{m}\left( \frac{65}{64}Q_{0}\right) $ satisfying
(C*1) and (C*2), our induction on $\mathcal{A}$ is complete. Thus, we have
proven the Main Lemma for all monotonic $\mathcal{A}\subseteq \mathcal{M}$.

\section{Restatement of the Main Lemma}

\label{rml}

In this section, we reformulate the Main Lemma for $\mathcal{A}$ in the case
in which $\mathcal{A}$ is the empty set $\emptyset$. Let us examine
hypotheses (A1), (A2), (A3) for the Main Lemma for $\mathcal{A}$, taking $%
\mathcal{A} = \emptyset$.

Hypothesis (A1) says that $\tg_{l\left( \emptyset \right) }$ has an 
$\left( \emptyset ,\epsilon ^{-1}\delta _{Q_{0}},C_{B}\right) $-basis at $%
\left( x_{0},M_{0},\tau_{0},P^{0}\right) $. This means simply that $P^{0}\in \Gamma
_{l\left( \emptyset \right) }\left( x_{0},C_{B}M_{0},C_{B}\tau_0\right) $.

Hypothesis (A2) says that $\delta _{Q_{0}}\leq \epsilon \delta _{\max }$,
and hypothesis (A3) says that $\epsilon $ is less than a small enough
constant determined by $C_{B}$, $C_{w}$, $m$, $n$, $C_{\G}$.

We take $\epsilon $ to be a small enough constant (determined by $C_{B}$, $%
C_{w}$, $m$, $n$, $C_{\G}$) such that (A3) is satisfied. We take $C_{B}=1$. Thus, we
arrive at the following equivalent version of the Main Lemma for $\emptyset $%
.

\theoremstyle{plain} 
\newtheorem*{thm Restated Main Lemma}{Restated Main
Lemma}%
\begin{thm Restated Main Lemma}Let $\tg_{0}=\left( \Gamma _{0}\left( x,M,\tau \right) \right) _{x\in
E,M>0, \tau\in(0,\tau_{\max}]}$ be a $\left( C_{w},\delta _{\max }\right) $-convex blob
field. For $l\geq 1$, let $\tg_{l}=\left( \Gamma _{l}\left(
x,M,\tau\right) \right) _{x\in E, M>0, \tau\in(0,\tau_{\max}]}$ be the approximate $l^{th}$-refinement of $\tg_{0}$. Fix a dyadic cube $Q_0$ of sidelength $\delta _{Q_{0}}\leq \epsilon
\delta _{\max }$, where $\epsilon >0$ is a small enough constant determined
by $m$, $n$, $C_{W}$, $C_{\G}$. Let $x_{0}\in E\cap \frac{65}{64}Q_0$, and let $P_{0}\in \Gamma
_{l\left( \emptyset \right) }\left( x_{0},M_{0},\tau_0\right) $.

Then there exists a function $F\in C^{m}\left( \frac{65}{64}Q_{0}\right) $,
satisfying 

\begin{itemize}
\item $\left\vert \partial ^{\beta }\left( F-P_{0}\right) \left( x\right)
\right\vert \leq C_{\ast }M_{0}\delta _{Q_{0}}^{m-\left\vert \beta
\right\vert }$ for $x\in \frac{65}{64}Q_{0}$, $\left\vert \beta \right\vert
\leq m$; and 
\item $J_{z}\left( F\right) \in \Gamma _{0}\left( z,C_{\ast }M_{0},C_{\ast}\tau_0\right) $
for all $z\in E\cap \frac{65}{64}Q_{0}$;
\end{itemize}
where $C_{\ast }$ is determined by $C_{w}$, $m$, $n$, $C_{\G}$.\end{thm Restated Main Lemma}

\subsection{What the Main Lemma gives us}
\label{sec:tree}

The statement and proof of the Main Lemma essentially describe a tree that we create top to bottom and then traverse to generate an appropriate function $F$. We fix the constant $\epsilon > 0$ to be a small enough constant determined by $m, n, C_w, C_\Gamma$. We also fix $\hat{M}_0, \hat{\tau}_0$ (inputs of the problem).

We define a node of the tree:
\begin{definition}
	A node $T$ is a tuple of the form $(\A_T, x_T, P_T, Q_T, E_T, C_T)$, where the following properties hold:
	\begin{itemize}
		\item[\refstepcounter{equation}\text{(\theequation)}\label{node1}] $C_T$ belongs to a list $I(\A_T)$ of constants determined by the label $\A_T$ and the constants $C_\G, C_w$, to be specified below.
		\item[\refstepcounter{equation}\text{(\theequation)}\label{node2}]$\A_T$ is monotonic, $\delta_{Q_T} < \epsilon\delta_{\max}$; $E_T = E \cap \frac{65}{64}Q_T$; $x_T \in E_T$ or, if $E_T = \emptyset$, $x_T \in E \cap 5Q_T^+$; $P_T \in \tg_{l(\A)}(x_T, C_T\hat{M}_0, C_T\hat{\tau}_0)$. 
		\item[\refstepcounter{equation}\text{(\theequation)}\label{node3}] $\tg_{l(\A)}(x_T, C_T\hat{M}_0, C_T\hat{\tau}_0)$ has an $(\A, \epsilon^{-1}\delta_{Q_T}, C_{\text{node}})$-basis at $(x_T, \hat{M}_0, \hat{\tau}_0, P_T)$. 
	\end{itemize}
\end{definition}

The root node is $(\emptyset, x_0, P_0, Q_0, E \cap \frac{65}{64}Q_0, 1)$, where $P_0 \in \tg_{l(\emptyset)}(x_0, \hat{M}_0, \hat{\tau}_0)$, $\delta_{Q_0} \leq \epsilon\delta_{\max}$, $x_0 \in E\cap \frac{65}{64}Q_0$.

Corresponding to a node $T$ there is an instance of the \textbf{Main Lemma} in which $\A = \A_T$, $x_0 = x_T$, $P_0 = P_T$, $Q_0 = Q_T$, $E_0 = E_T$, $M_0 = C_{T}\hat{M}_0$ and $\tau_0 = C_{T}\hat{\tau}_0$.

The induction step in our proof of the \textbf{Main Lemma} reduces the construction of an interpolant for a node $T$ (with $\A_T \neq \M$) to the construction of interpolants for nodes $T' = (\A_{T'}, y_{T'}, P_{T'}, Q_{T'}, E_{T'}, C^{\#}_{T'})$ with $\A_{T'} < \A_T$, $Q_{T'}$ a $CZ$ cube, and $C^{\#}_{T'}$ a constant depending only on $C_{T}$, $\A_T$ and $\A_{T'}$.

We take the children of a node $T$ to be all the nodes $T'$ arising in this way. Note that the constants $C_{T}$ associated to nodes containing the label $\A_T$ belong to a finite list $I(\A_T)$ determined by $\A_T, C_\G, C_w$ (see Section \ref{sec:const}).

Nodes of the form $(\A_T, x_T, P_T, Q_T, E_T, C_{T})$ with $\A = \M$ have no children, and the interpolant is $P_T$. Nodes of the form $(\A_T, x_T, P_T, Q_T, E_T, C_{T})$ for which $E_T$ contains at most a single point also have no children, and the interpolant is also $P_T$. All other nodes of our tree have children. This completes our description of the tree. For an algorithmic explanation, see Section \ref{sec:comp-main}.

\section{Tidying Up}

\label{tu}

In this section, we remove from the Restated Main Lemma the small constant $%
\epsilon$ and the assumption that $Q_0$ is dyadic.

\begin{theorem}
\label{theorem-tu1} Let $\tg_{0}=\left( \Gamma _{0}\left(
x,M,\tau\right) \right) _{x\in E, M>0, \tau\in(0,\tau_{\max}]}$ be a $\left( C_{w},\delta _{\max }\right) $%
-convex blob field with blob constant $C_{\G}$. For $l\geq 1$, let $\tg_{l}=\left( \Gamma
_{l}\left( x,M,\tau\right) \right) _{x\in E, M>0, \tau\in(0,\tau_{\max}]}$ be the approximate $l^{th}$-refinement of $%
\tg_{0}$. Fix a cube $Q_0$ of sidelength $\delta _{Q_{0}}\leq
\delta _{\max }$, a point $x_0 \in E \cap \frac{65}{64}Q_0$, and a real number $M_0>0$.
Let $P_{0}\in \Gamma _{l\left( \emptyset \right) +1}\left(
x_{0},M_{0},\tau_0\right) $.

Then there exists a function $F\in C^{m}\left(Q_{0}\right) $ satisfying the
following, with $C_{\ast }$ determined by $C_{w}$, $m$, $n$, $C_{\G}$.

\begin{itemize}
\item $\left\vert \partial ^{\beta }\left( F-P_{0}\right) \left( x\right)
\right\vert \leq C_{\ast }M_{0}\delta _{Q_{0}}^{m-\left\vert \beta
\right\vert }$ for $x\in Q_{0}$, $\left\vert \beta \right\vert \leq m$; and

\item $J_{z}\left( F\right) \in \Gamma _{0}\left( z,C_{\ast }M_{0}, C_{\ast}\tau_0\right) $
for all $z\in E\cap Q_{0}$.
\end{itemize}
\end{theorem}

\begin{proof}
Let $\epsilon >0$ be the small constant in the statement of the Restated
Main Lemma in Section \ref{rml}. In particular, $\epsilon $ is determined by 
$C_{w}$, $m$, $n$, $C_{\G}$. We write $c$, $C$, $C^{\prime }$, etc., to denote
constants determined by $C_{w}$, $m$, $n$, $C_{\G}$. These symbols may denote
different constants in different occurrences.

We cover $CQ_{0}$ by a grid of dyadic cubes $\{Q_{\nu }\}$, all having same
sidelength $\delta _{Q_{\nu }}$, with $\frac{\epsilon }{20}\delta
_{Q_{0}}\leq \delta _{Q_{\nu }}\leq \epsilon \delta _{Q_{0}}$, and all
contained in $C^{\prime }Q_{0}$. (We use at most $C$ distinct $Q_{\nu }$ to
do so.)

For each $Q_{\nu }$ with $E\cap \frac{65}{64}Q_{\nu }\not=\emptyset $, we
pick a point $x_{\nu }\in E\cap \frac{65}{64}Q_{\nu }$; we know (by virtue of the results on refinements) there exists $P_{\nu }\in \Gamma _{l(\emptyset
)}(x_{\nu },CM_{0}, C\tau_0)$ such that $\left\vert \partial ^{\beta }\left( P_{\nu
}-P_{0}\right) \left( x_{0}\right) \right\vert \leq CM_{0}\delta
_{Q_{0}}^{m-\left\vert \beta \right\vert }$ for $\beta \in \mathcal{M}$, and
therefore

\begin{itemize}
\item[\refstepcounter{equation}\text{(\theequation)}\label{tu1}] $\left\vert
\partial ^{\beta }\left( P_{\nu }-P_{0}\right) \left( x\right) \right\vert
\leq C^{\prime }M_{0}\delta _{Q_{0}}^{m-\left\vert \beta \right\vert }$ for $%
x\in \frac{65}{64}Q_{0}$ and $\left\vert \beta \right\vert \leq m$.
\end{itemize}

Since $x_{\nu }\in E\cap \frac{65}{64}Q_{\nu }$, $P_{\nu }\in \Gamma
_{l\left( \emptyset \right) }(x_{\nu },CM_{0}, C\tau_0)$, and $\delta _{Q_{\nu }}\leq \epsilon \delta
_{Q_{0}}\leq \epsilon \delta _{\max }$, the Restated Main Lemma applies to $%
x_{\nu },P_{\nu },Q_{\nu }$ to produce $F_{\nu }\in C^{m}\left( \frac{65}{64}%
Q_{\nu }\right) $ satisfying

\begin{itemize}
\item[\refstepcounter{equation}\text{(\theequation)}\label{tu2}] $\left\vert
\partial ^{\beta }\left( F_{\nu }-P_{\nu }\right) \left( x\right)
\right\vert \leq CM_{0}\delta _{Q_{\nu }}^{m-\left\vert \beta \right\vert
}\leq CM_{0}\delta _{Q_{0}}^{m-\left\vert \beta \right\vert }$ for $x\in 
\frac{65}{64}Q_{\nu }$, $\left\vert \beta \right\vert \leq m$;
\end{itemize}

and

\begin{itemize}
\item[\refstepcounter{equation}\text{(\theequation)}\label{tu3}] $%
J_{z}\left( F_{\nu }\right) \in \Gamma _{0}\left( z,CM_{0}, C\tau_0\right) $ for all $%
z\in E\cap \frac{65}{64}Q_{\nu }$.
\end{itemize}

From \eqref{tu1} and \eqref{tu2}, we have

\begin{itemize}
\item[\refstepcounter{equation}\text{(\theequation)}\label{tu4}] $\left\vert
\partial ^{\beta }\left( F_{\nu }-P_{0}\right) \left( x\right) \right\vert
\leq CM_{0}\delta _{Q_{0}}^{m-\left\vert \beta \right\vert }$ for $x\in 
\frac{65}{64}Q_{\nu }$, $\left\vert \beta \right\vert \leq m$.
\end{itemize}

We have produced such $F_{\nu }$ for those $\nu $ satisfying $E\cap \frac{65%
}{64}Q_{\nu }\not=\emptyset $. If instead $E\cap \frac{65}{64}Q_{\nu
}=\emptyset $, then we set $F_{\nu }=P_{0}$. Then \eqref{tu3} holds
vacuously and \eqref{tu4} holds trivially. Thus, our $F_{\nu }$ satisfy %
\eqref{tu3}, \eqref{tu4} for all $\nu $. From \eqref{tu4} we obtain

\begin{itemize}
\item[\refstepcounter{equation}\text{(\theequation)}\label{tu5}] $\left\vert
\partial ^{\beta }\left( F_{\nu }-F_{\nu ^{\prime }}\right) \left( x\right)
\right\vert \leq CM_{0}\delta _{Q_{0}}^{m-\left\vert \beta \right\vert }$
for $x\in \frac{65}{64}Q_{\nu }\cap \frac{65}{64}Q_{\nu ^{\prime }}$, $%
\left\vert \beta \right\vert \leq m$.
\end{itemize}

Next, we introduce a partition of unity. We fix cutoff functions $\theta
_{\nu }\in C^{m}\left( \mathbb{R}^{n}\right) $ satisfying

\begin{itemize}
\item[\refstepcounter{equation}\text{(\theequation)}\label{tu6}] support $%
\theta _{\nu }\subset \frac{65}{64}Q_{\nu }$, $\left\vert \partial ^{\beta
}\theta _{\nu }\right\vert \leq C\delta _{Q_{0}}^{-\left\vert \beta
\right\vert }$ for $\left\vert \beta \right\vert \leq m$, $\sum_{\nu }\theta
_{\nu }^{2}=1$ on $Q_{0}$.
\end{itemize}

We then define

\begin{itemize}
\item[\refstepcounter{equation}\text{(\theequation)}\label{tu7}] $%
F=\sum_{\nu }\theta _{\nu }^{2}F_{\nu }$ on $Q_{0}$.
\end{itemize}

We have then

\begin{itemize}
\item[\refstepcounter{equation}\text{(\theequation)}\label{tu8}] $%
F-P_{0}=\sum_{\nu }\theta _{\nu }^{2}\left( F_{\nu }-P_{0}\right) $ on $%
Q_{0} $.
\end{itemize}

Thanks to \eqref{tu4} and \eqref{tu6}, we have $\theta _{\nu }^{2}\left(
F_{\nu }-P_{0}\right) \in C^{m}\left( Q_{0}\right) $ and $\left\vert
\partial ^{\beta }\left( \theta _{\nu }^{2}\cdot \left( F_{\nu
}-P_{0}\right) \right) \left( x\right) \right\vert \leq CM_{0}\delta
_{Q_{0}}^{m-\left\vert \beta \right\vert }$ for $x\in Q_{0},|\beta |\leq m$,
all $\nu $. Moreover, there are at most $C$ distinct $\nu $ appearing in %
\eqref{tu8}. Hence,

\begin{itemize}
\item[\refstepcounter{equation}\text{(\theequation)}\label{tu9}] $F\in
C^{m}\left( Q_{0}\right) $
\end{itemize}

and

\begin{itemize}
\item[\refstepcounter{equation}\text{(\theequation)}\label{tu10}] $%
\left\vert \partial ^{\beta }\left( F-P_{0}\right) \left( x\right)
\right\vert \leq CM_{0}\delta _{Q_{0}}^{m-\left\vert \beta \right\vert }$
for $x\in Q_{0}$, $\left\vert \beta \right\vert \leq m$.
\end{itemize}

Next, let $z\in E\cap Q_{0}$, and let $Y$ be the set of all $\nu $ such that 
$z\in \frac{65}{64}Q_{\nu }$. Then \eqref{tu6}, \eqref{tu7} give $%
J_{z}(F)=\sum_{\nu \in Y}J_{z}\left( \theta _{\nu }\right) \odot
_{z}J_{z}\left( \theta _{\nu }\right) \odot _{z}J_{z}\left( F_{\nu }\right) $
and we know that $J_{z}(F_{\nu })\in \Gamma _{0}\left( z,CM_{0}, C\tau_0\right) $ for 
$\nu \in Y$ (by \eqref{tu3}); $|\partial ^{\beta }\left[ J_{z}\left( F_{\nu
}\right) -J_{z}\left( F_{\nu ^{\prime }}\right) \right] \left( z\right)
|\leq CM_{0}\delta _{Q_{0}}^{m-\left\vert \beta \right\vert }$ for $%
\left\vert \beta \right\vert \leq m-1$, $\nu ,\nu ^{\prime }\in Y$ (by %
\eqref{tu5}); $\left\vert \partial ^{\beta }\left[ J_{z}\left( \theta _{\nu
}\right) \right] \left( z\right) \right\vert \leq C\delta
_{Q_{0}}^{-\left\vert \beta \right\vert }$ for $\left\vert \beta \right\vert
\leq m-1$, $\nu \in Y$ (by \eqref{tu6}); $\sum_{\nu \in Y}J_{z}\left( \theta
_{\nu }\right) \odot _{z}J_{z}\left( \theta _{\nu }\right) =1$ (again thanks
to \eqref{tu6}); $\#(Y)\leq C$ (since there are at most $C$ distinct $%
Q_{\nu }$ in our grid); and $\delta _{Q_{0}}\leq \delta _{\max }$ (by
hypothesis of Theorem \ref{theorem-tu1}). Since $\tg_{0}$ is $%
(C,\delta _{\max })$-convex, the above remarks and Lemma \ref{lemma-wsf2}
tell us that $J_{z}(F)\in \Gamma _{0}(z,CM_{0},C\tau_0)$. Thus,

\begin{itemize}
\item[\refstepcounter{equation}\text{(\theequation)}\label{tu11}] $%
J_{z}\left( F\right) \in \Gamma _{0}\left( z,CM_{0}, C\tau_0 \right) $ for all $z\in
E\cap Q_{0}$.
\end{itemize}

Our results \eqref{tu9}, \eqref{tu10}, \eqref{tu11} are the conclusions of
Theorem \ref{theorem-tu1}.

The proof of that Theorem is complete.
\end{proof}
\section{From Shape Field to Blob Field}
\label{sec:smoothsel}

In this section we show an application of this result to the "Smooth Selection Problem". In Section III.2 of \cite{feffermanFinitenessPrinciplesSmooth2016} we can see a result on finiteness principles for this problem.

We define the Smooth Selection Problem as follows: Let $E \subset \R^n$ be finite, let $M>0$. For each $x \in E$ let $K(x) \subset \R^D$ be convex. We want to know if there exists a function $F \in C^m(\R^n, \R^D)$ such that $\|F\|_{C^m(\R^n, \R^D)} \leq M$ and $F(z) \in K(z)$ for all $z \in E$. If it exists, we want to give its jet $J_x(F)$ at each point $x\in E$.

Let us first set up notation. We write $c$, $C$, $C^{\prime }$, etc., to
denote constants determined by $m$, $n$, $D$; these symbols may denote
different constants in different occurrences. We will work with $C^{m}$
vector and scalar-valued functions on $\mathbb{R}^{n}$, and also with $%
C^{m+1}$ scalar-valued functions on $\mathbb{R}^{n+D}$. We use Roman letters
($x$, $y$, $z$$,\cdots $) to denote points of $\mathbb{R}^{n}$, and Greek
letters $(\xi ,\eta ,\zeta ,\cdots )$ to denote points of $\mathbb{R}^{D}$.
We denote points of the $\mathbb{R}^{n+D}$ by $(x,\xi )$, $(y,\eta )$, etc.
As usual, $\mathcal{P}$ denotes the vector space of real-valued polynomials of degree at
most $m-1$ on $\mathbb{R}^{n}$. We write $\mathcal{P}^{D}$ to denote the
direct sum of $D$ copies of $\mathcal{P}$. If $F\in C^{m-1}(\mathbb{R}^{n},%
\mathbb{R}^{D})$ with $F(x)=\left( F_{1}\left( x\right) ,\cdots ,F_{D}\left(
x\right) \right) $ for $x\in \mathbb{R}^{n}$, then $J_{x}(F):=(J_{x}\left(
F_{1}\right) ,\cdots ,J_{x}\left( F_{D}\right) )\in \mathcal{P}^{D}$.

We write $\mathcal{P}^{+}$ to denote the vector space of real-valued polynomials of
degree at most $m$ on $\mathbb{R}^{n+D}$. If $F\in C^{m+1}\left( \mathbb{R}%
^{n+D}\right) $, then we write $J_{\left( x,\xi \right) }^{+}F\in \mathcal{P}%
^{+}$ to denote the $m^{th}$-degree Taylor polynomial of $F$ at the point $%
\left( x,\xi \right) \in \mathbb{R}^{n+D}$. 

When we work with $\mathcal{P}^{+}$, we write $\odot _{\left( x,\xi \right)
} $ to denote the multiplication 
\begin{equation*}
P\odot _{\left( x,\xi \right) }Q:=J_{\left( x,\xi \right) }^{+}\left(
PQ\right) \in \mathcal{P}^{+}\text{ for }P,Q\in \mathcal{P}^{+}\text{.}
\end{equation*}

We now introduce the relevant blob field.

Let $E^{+}=\left\{ \left( x,0\right) \in \mathbb{R}^{n+D}:x\in E\right\} $.
For $x_0 \in E$ let $K(x_0)$ be a compact convex sets in $\R^D$.  For $\left( x_{0},0\right) \in E^{+}$, $M>0$ and $\tau\in(0,\tau_{\max}]$ (where $\tau_{\max}$ is a constant depending only on $m,n,D$), let

\begin{itemize}
\item[\refstepcounter{equation}\text{(\theequation)}\label{fpii3}] $\Gamma
\left( \left( x_0,0\right) ,M,\tau\right) =\left\{ 
\begin{array}{c}
P\in \mathcal{P}^{+}:P\left( x_{0},0\right) =0,\nabla _{\xi }P\left(
x_{0},0\right) \in (1+\tau)\blacklozenge K\left( x_{0}\right) , \\ 
\left\vert \partial _{x}^{\alpha }\partial _{\xi }^{\beta }P\left(
x_{0},0\right) \right\vert \leq M\text{ for }\left\vert \alpha \right\vert
+\left\vert \beta \right\vert \leq m%
\end{array}%
\right\} \mathcal{\subset \mathcal{P}}^{+}$.
\end{itemize}

Let $\tg=\left( \Gamma (\left( x_{0},0\right) ,M,\tau)\right) $ $_{\left(
x_0,0\right) \in E^{+},M>0,\tau\in(0,\tau_{\max}]}$.

\begin{lemma}
\label{lemma-bfp1} $\tg$ is a $(C,1)$-convex blob field of blob constant $(2+\tau_{\max})$.
\end{lemma}

\begin{proof}[Proof of Lemma \protect\ref{lemma-bfp1}]
Clearly, each $\Gamma ((x_0,0),M,\tau)$ is a (possibly empty) convex subset of $%
\mathcal{P}^{+}$.

Let's look at $P \in (1+\tau)\blacklozenge\G((x_0,0),M,\tau)$. That is, there exist $P', P_+, P_- \in \G((x_0,0),M,\tau)$ such that $P = P' +\frac{\tau}{2}P_+-\frac{\tau}{2}P_-$. Clearly, $P(x_0,0) = 0$. Furthermore, $|\partial_x^{\alpha}\partial_{\xi}^{\beta}P(x_0,0)|\leq (1+\tau)M\leq (2+\tau_{\max})M$. Finally, $\nabla_{\xi}P(x_0,0) \in (1+(2+\tau)\tau)\blacklozenge K(x_0) \subset (1+(2+\tau_{\max})\tau)\blacklozenge K(x_0)$ (see Lemma \ref{lem:twicetau}). Thus, $\tg$ is a blob field (with $m+1$, $n+D$ playing the roles of $m$, $n$, respectively) with blob constant $(2+\tau_{\max})$.

To prove $(C,1)$%
-convexity, let $x_{0}\in E,0<\delta \leq 1$, let

\begin{itemize}
\item[\refstepcounter{equation}\text{(\theequation)}\label{fpii4}] $%
P_{1},P_{2}\in \Gamma \left( \left( x_{0},0\right) ,M,\tau\right) $ with
\end{itemize}

\begin{itemize}
\item[\refstepcounter{equation}\text{(\theequation)}\label{fpii5}] $%
\left\vert \partial _{x}^{\alpha }\partial _{\xi }^{\beta }\left(
P_{1}-P_{2}\right) \left( x_{0},0\right) \right\vert \leq M\delta ^{\left(
m+1\right) -\left\vert \alpha \right\vert -\left\vert \beta \right\vert }$
for $\left\vert \alpha \right\vert +\left\vert \beta \right\vert \leq m$;
and let
\end{itemize}

\begin{itemize}
\item[\refstepcounter{equation}\text{(\theequation)}\label{fpii6}] $%
Q_{1},Q_{2}\in \mathcal{P}^{+}$, with
\end{itemize}

\begin{itemize}
\item[\refstepcounter{equation}\text{(\theequation)}\label{fpii7}] $%
\left\vert \partial _{x}^{\alpha }\partial _{\xi }^{\beta }Q_{i}\left(
x_{0},0\right) \right\vert \leq \delta ^{-\left\vert \alpha \right\vert
-\left\vert \beta \right\vert }$ for $i=1,2$, $\left\vert \alpha \right\vert
+\left\vert \beta \right\vert \leq m$, and with
\end{itemize}

\begin{itemize}
\item[\refstepcounter{equation}\text{(\theequation)}\label{fpii8}] $%
Q_{1}\odot _{\left( x_{0},0\right) }Q_{1}+Q_{2}\odot _{\left( x_{0},0\right)
}Q_{2}=1$.
\end{itemize}

We must show that the polynomial

\begin{itemize}
\item[\refstepcounter{equation}\text{(\theequation)}\label{fpii9}] $%
P:=Q_{1}\odot _{\left( x_{0},0\right) }Q_{1}\odot _{\left( x_{0},0\right)
}P_{1}+Q_{2}\odot _{\left( x_{0},0\right) }Q_{2}\odot _{\left(
x_{0},0\right) }P_{2}$ \end{itemize}
belongs to $\Gamma \left( \left( x_{0},0\right)
,CM,C\tau \right) $.

From \eqref{fpii3}, \eqref{fpii4}, we have

\begin{itemize}
\item[\refstepcounter{equation}\text{(\theequation)}\label{fpii10}] $%
\left\vert \partial _{x}^{\alpha }\partial _{\xi }^{\beta }P_{1}\left(
x_{0},0\right) \right\vert \leq M$ for $\left\vert \alpha \right\vert
+\left\vert \beta \right\vert \leq m$,
\end{itemize}

\begin{itemize}
\item[\refstepcounter{equation}\text{(\theequation)}\label{fpii11}] $%
P_{1}\left( x_{0},0\right) =P_{2}\left( x_{0},0\right) =0$, and
\end{itemize}

\begin{itemize}
\item[\refstepcounter{equation}\text{(\theequation)}\label{fpii12}] $\nabla
_{\xi }P_{1}\left( x_{0},0\right) $, $\nabla _{\xi }P_{2}\left(
x_{0},0\right) \in (1+\tau)\blacklozenge K\left( x_{0}\right) $.
\end{itemize}

Then \eqref{fpii9}, \eqref{fpii11} give 
\begin{equation*}
P\left( x_{0},0\right) =0
\end{equation*}%
and 
\begin{equation*}
\nabla _{\xi }P\left( x_{0},0\right) =\left( Q_{1}\left( x_{0},0\right)
\right) ^{2}\nabla _{\xi }P_{1}\left( x_{0},0\right) +\left( Q_{2}\left(
x_{0},0\right) \right) ^{2}\nabla _{\xi }P_{2}\left( x_{0},0\right)
\end{equation*}%
while \eqref{fpii8} yields%
\begin{equation*}
\left( Q_{1}\left( x_{0},0\right) \right) ^{2}+\left( Q_{2}\left(
x_{0},0\right) \right) ^{2}=1\text{.}
\end{equation*}%
Together with \eqref{fpii12} and convexity of $(1+\tau)\blacklozenge K(x_{0})$, the above remarks
imply that

\begin{itemize}
\item[\refstepcounter{equation}\text{(\theequation)}\label{fpii13}] $P\left(
x_{0},0\right) =0$ and $\nabla _{\xi }P\left( x_{0},0\right) \in (1+\tau)\blacklozenge K\left(
x_{0}\right) $.
\end{itemize}

Also, \eqref{fpii8}, \eqref{fpii9} imply the formula

\begin{itemize}
\item[\refstepcounter{equation}\text{(\theequation)}\label{fpii14}] $%
P=P_{1}+Q_{2}\odot _{\left( x_{0},0\right) }Q_{2}\odot _{\left(
x_{0},0\right) }\left( P_{2}-P_{1}\right) $.
\end{itemize}

From \eqref{fpii5}, \eqref{fpii7}, and $\delta \leq 1$, we have 
\begin{eqnarray*}
\left\vert \partial _{x}^{\alpha }\partial _{\xi }^{\beta }\left[ Q_{2}\odot
_{\left( x_{0},0\right) }Q_{2}\odot _{\left( x_{0},0\right) }\left(
P_{2}-P_{1}\right) \right] \left( x_{0},0\right) \right\vert &\leq &CM\delta
^{\left( m+1\right) -\left\vert \alpha \right\vert -\left\vert \beta
\right\vert } \\
&\leq &CM
\end{eqnarray*}%
for $\left\vert \alpha \right\vert +\left\vert \beta \right\vert \leq m$.
Together with \eqref{fpii10} and \eqref{fpii14}, this tells us that

\begin{itemize}
\item[\refstepcounter{equation}\text{(\theequation)}\label{fpii15}] $%
\left\vert \partial _{x}^{\alpha }\partial _{\xi }^{\beta }P\left(
x_{0},0\right) \right\vert \leq CM$ for $\left\vert \alpha \right\vert
+\left\vert \beta \right\vert \leq m$.
\end{itemize}

From \eqref{fpii13}, \eqref{fpii15} and the definition \eqref{fpii3}, we see
that $P\in \Gamma \left( \left( x_{0},0\right) ,CM, C\tau\right) $, completing the
proof of Lemma \ref{lemma-bfp1}.
\end{proof}

Note that the $K(x_0)$ need not be polytopes. For Lemma \ref{lem:bloboracle}, which involves the computation of the initial blob Oracle that will allow us the computation of the interpolant, we need as input the descriptors of some polytopes. Therefore we will work with slightly different blob fields.

We assume that we have an Oracle that given $x \in E$ and $\tau_0 > 0$ charges us $C(\tau_0)$ work to produce $\Delta(x_0, \tau_0)$, a descriptor of length $|\Delta(x_0, \tau_0)| \leq C(\tau_0)$ such that $K(x_0) \subset K(\Delta(x_0, \tau_0)) \subset (1+\tau_0)\blacklozenge K(x_0)$ and $K(\Delta(x_0, \tau_0)) \subset K(\Delta(x_0, \tau_0'))$ for $\tau_0' \geq \tau_0$. The blob field we'll be working with is now given by $K(\Delta(x_0, \tau_0))$.

\begin{remark}
	To obtain this Oracle, supposing $K(x_0)$ are polytopes, we could simply use Algorithm \ref{alg:approx-tau} on the initial polytopes for each $\tau$. In order to have the second property, however, if we use this algorithm, we need a way to guarantee that the $\tau$-nets of the unit ball are "nested", i.e., if $\Lambda_{\tau}$ is a $\tau$-net of the sphere, then we need $\Lambda_{\tau} \subset \Lambda_{\tau'}$ if $\tau' \leq \tau$. For computational purposes, we instead produce a lazy (that is, computed as needed) list of $\tau$s, starting from some $\tau_0$ and computing (as needed) $\Lambda_{2^{-jD}\tau_0}$ for $j \geq 0$. These nets are nested, and therefore so are the convex sets that we find from the descriptor. 
\end{remark}

\begin{lemma}
	\label{lem:finite}
	The blob fields
	\begin{itemize}
		
		\item[\eqref{fpii3}] $\Gamma
		\left( \left( x_0,0\right) ,M,\tau\right) =\left\{ 
		\begin{array}{c}
		P\in \mathcal{P}^{+}:P\left( x_{0},0\right) =0,\nabla _{\xi }P\left(
		x_{0},0\right) \in (1+\tau)\blacklozenge K\left( x_{0}\right) , \\ 
		\left\vert \partial _{x}^{\alpha }\partial _{\xi }^{\beta }P\left(
		x_{0},0\right) \right\vert \leq M\text{ for }\left\vert \alpha \right\vert
		+\left\vert \beta \right\vert \leq m%
		\end{array}%
		\right\} \mathcal{\subset \mathcal{P}}^{+}$.
		
		\item[\refstepcounter{equation}\text{(\theequation)}\label{approx-smooth}] $\Gamma'
		\left( \left( x_0,0\right) ,M,\tau\right) =\left\{ 
		\begin{array}{c}
		P\in \mathcal{P}^{+}:P\left( x_{0},0\right) =0,\nabla _{\xi }P\left(
		x_{0},0\right) \in (1+\tau)\blacklozenge K\left(\Delta(x_0, \tau)\right) , \\ 
		\left\vert \partial _{x}^{\alpha }\partial _{\xi }^{\beta }P\left(
		x_{0},0\right) \right\vert \leq M\text{ for }\left\vert \alpha \right\vert
		+\left\vert \beta \right\vert \leq m%
		\end{array}%
		\right\} \mathcal{\subset \mathcal{P}}^{+}$.
	\end{itemize}
	are $C-$equivalent with $C$ depending only on $m, n, D$. In particular, $\Gamma'$ is a blob field.
\end{lemma}
\begin{proof}
	Clearly, for every $x \in E$ we have $\Gamma(x, M, \tau) \subset \Gamma'(x, M, \tau)$, and $\Gamma'(x, M, \tau)$ is a convex set.
	
	Let $P \in \Gamma'(x, M, \tau)$. Then by definition $P(x_0, 0) = 0$, $\left\vert \partial _{x}^{\alpha }\partial _{\xi }^{\beta }P\left(
	x_{0},0\right) \right\vert \leq M \leq (2+\tau_{\max})M\text{ for }\left\vert \alpha \right\vert
	+\left\vert \beta \right\vert \leq m$ and $\nabla_{\xi }P(x_0, 0) \in (1+\tau)\blacklozenge K(\Delta(x_0, \tau)) \subset (1+\tau)\blacklozenge \left[(1+\tau)\blacklozenge K(x_0)\right] \subset (1+(2+\tau_{\max})\tau)\blacklozenge K(x_0)$ by Lemma \ref{lem:twicetau}.
	
	Therefore $\Gamma'(x, M, \tau) \subset \Gamma(x, (2+\tau_{\max})M, (2+\tau_{\max})\tau)$. By Lemma \ref{lem:c-eq-bl}, $\Gamma'$ is a blob field and it is $C-$equivalent to $\Gamma$, where $C$ depends only on $\tau_{\max}$, which depends only on $n, m, D$.
\end{proof}
\begin{remark}
	This proof gives us worse constants than the optimal ones.
\end{remark}
\begin{lemma}
	\label{lem:bloboracle}
	Suppose we are given $E, E^+$ as above, and an Oracle that for each $x_0 \in E, \tau$ returns a descriptor $\Delta(x_0, \tau)$ ($|\Delta(x_0)| \leq C(\tau)$ charging us $C(\tau)$ work) as defined above. We can produce a blob field Oracle that for $M_0, \tau_0$ will return a list of $\G'(x, M_0, \tau_0)_{x\in E}$ as defined in Lemma \ref{lem:finite}, in time at most $C(\tau_0)N \log N$. 
\end{lemma}
\begin{proof}
	Fix $M_0, \tau_0$. For each $\Delta(x_0, \tau_0)$ (obtained in $C(\tau_0)$ operations by calling the Oracle), we produce $\Delta_{\P}(x_0, \tau_0)$ such that $K(\Delta_{\P}(x_0, \tau_0)) = \{P \in \P^+ : \nabla_{\xi}P(x_0, 0) \in (1+\tau_0)\blacklozenge K(\Delta(x_0, \tau_0))\}$. Finding this descriptor is simple and requires only a matrix-matrix multiplication (of dimension $C(\tau_0) \times D$ for $\xi$ and $D\times\text{dim}(\P^+)$ for $\nabla_{\xi}P$), which takes $C(\tau_0)$ operations. The rows of this matrix-matrix multiplication will be the coefficients of the descriptor. The size of the descriptor is $C(\tau_0) \times \text{dim}(\P^+)$. 
	
	Similarly we compute a descriptor $\mathring{\Delta}_{\P}(x_0, M_0)$ such that $K(\mathring{\Delta}_{\P}(x_0, M_0)) = \{P \in \P^+ : |\partial_x^\alpha \partial_\xi^\beta P(x_0, 0)| \leq M_0 \text{ for }|\alpha|+|\beta|\leq m\}$. The number of constraints here is bounded by a constant depending only on $n, m$ and computing the coefficients takes $C$ operations. 
	
	Finally, we create another descriptor representing the constraint $P(x_0, 0) = 0$. 
	
	We return the complete descriptor (combination of these three descriptors, corresponding to the intersection of the convex sets as explained in Remark \ref{rem:intersect}) $\tilde{\Delta}_{\P}(x_0, M_0, \tau_0)$ in $C(\tau_0)$ time. The number of constraints is bounded by $C(\tau_0)$, not depending on $\# E$. We don't need to compute an approximation via Algorithm \ref{alg:AI} because as soon as we start building the refinements those approximations will be computed.
	
	In total, with at most $C(\tau_0)N$ operations, we have returned a list (indexed by $x_0$) of descriptors $\tilde{\Delta}_\P(x_0, M_0, \tau_0)$, which is even less than required by our definition of an Oracle.
\end{proof}
\part{Algorithms}

In this part we will present the two algorithms to use for finding the norm of the interpolant and for computing the interpolant in the smooth selection problem. We also discuss the complexity of both algorithms.

\section{Finding the norm of the interpolant}

Here, we will provide an algorithm that finds (up to an order of magnitude) the norm of the function guaranteed to exist by Theorem \ref{theorem-tu1} of section \ref{tu}.

\subsection{Decision problem}
Given $\tg_0 = (\Gamma_0(x,M,\tau))_{x\in E, M>0, \tau\in(0,\tau_{\max}]}$ a $(C_w,\delta_{\max})$-convex blob field with blob constant $C_{\Gamma}$, and fixing $M_0 > 0$, $\tau_0 \in (0,\frac{\tau_{\max}}{C}]$ this algorithm returns $0$ if no function $F\in \mathcal{C}^m(\R^n)$ exists such that $J_x(F) \in \Gamma_0(x,CM_0,C\tau_0)$ for all $x\in E$ and such that $|\partial^{\beta}F|\leq CM$ in $\R^n$, and $1$ if there exists a function $F\in\mathcal{C}^m(\R^n)$ such that $J_x(F) \in \Gamma_0(x,c_* M_0, c_* \tau_0)$ for all $x \in E$ and $|\partial^{\beta}F|\leq c_*M$ in $\R^n$ with $c, C_*$ determined by $C_w, m, n, C_{\G}$.

Note that $\tg_0$ will come in the form of an Oracle $\Omega$ (as in Definition \ref{def:oracle}) that responds to a query $(M, \tau)$ with a list of the descriptors of $(\Gamma_0(x, M, \tau))_{x\in E}$ and charges work $O(N\log N)$ and storage $O(N)$.

\begin{algorithm}
  \label{alg:dec}
  \KwData{WSPD of $E$, $\tg_0$, $Q_0$ with $\delta_{Q_0}\leq \delta_{\max}$, $x_0 \in E \cap 5Q_0$, $M_0$, $\tau_0$}
  \KwResult{0 if there does not exist a function $F\in \mathcal{C}^m(Q_0)$ such that $J_x(F) \in \Gamma_0(x,CM_0,C\tau_0)$ for all $x\in E\cap \frac{65}{64}Q_0$ and such that $|\partial^{\beta}F|\leq CM$ in $\R^n$, 1 if there exists a function $F\in \mathcal{C}^m(\R^n)$ such that $J_x(F) \in \Gamma_0(x,c_* M_0, c_* \tau_0)$ for all $x \in E \cap \frac{65}{64}Q_0$ and $|\partial^{\beta}F|\leq c_*M$}
  Compute approximate $l(\emptyset)+1$th refinement of $\tg_0$, $\tg_{l(\emptyset)+1}$\;
  \uIf{$\tg_{l(\emptyset)+1} == \emptyset$}{
    \KwRet{0}\;
  }\uElse{
    \KwRet{1}\;
  }
  \caption{Decision}
\end{algorithm}

To compute the refinements, we use the results from Section \ref{sec:computing}. Note that a single call to the Oracle is needed for a given pair $(M, \tau)$. Recall that each refinement takes $CN\log N$ operations, with $C$ depending on $\tau_0$, $n$ and $m$. Computing $l(\emptyset) + 1$ refinements will take then $CN\log N$ operations again. The storage required is $CN$. Therefore, each refinement can be called just like the Oracle.

\begin{remark}
	Recall that, by Megiddo's algorithm \cite{megiddoLinearProgrammingLinear1984}, we can decide whether a given $\Gamma_l (x, M, \tau)$ is empty.
\end{remark}

\section{Constants}
\label{sec:const}
To save ourselves from trouble in the next sections, we will compute and store all the necessary refinements with the appropriate constants as follows. Assume we are given $E$, $\tg_0$ (with relevant constants $C_w, C_\G$), $M_0$, $\tau_0$, $\epsilon$. We are preparing to implement the inductive proof of the Main Lemma by an algorithm.

For $\A = \emptyset$ we only need to compute $\Gamma_{l(\emptyset)}(x, M_0, \tau_0)_{x\in E}$, $\Gamma_{l(\emptyset)-1}(x, C_B'M_0, C_B'\tau_0)_{x \in E}$ and $\Gamma_{l(\emptyset)-3}(x, C_B''M_0, C_B''\tau_0)_{x\in E}$ for certain fixed $C_B'$, $C_B''$ depending only on $n, m, C_{\G}, C_{w}$.

Given $\A, \Gamma_{l(A)}(x, CM_0, C\tau_0), y_Q$, there is a step in the algorithm  corresponding to the Main Lemma for $\A$ where we will find $P^\#$ and $\A^\#$ such that $P^\# \in \Gamma_{l(\A^{\#})}(y_Q, \hat{C}(\A,\A^\#)CM_0, \hat{C}(\A,\A^\#)C\tau_0)$ for $\hat{C}(\A, \A^{\#})$ depending only on $n, m, C_\G, C_w, \A, \A^\#$. We store these $\hat{C}(\A, \A^{\#})$ (for example, in a hashable map) and use them to define the following lists. 

We initialize all $I[\A]$ to be empty. We set $I[\emptyset] = (1)$. Then, for each $\A^\#$, we iterate over all $\A > \A^\#$ and we add to $I[\A^\#]$ all the constants of the form $\hat{C}(\A, \A^\#)C$ with $\hat{C}(\A, \A^\#)$ as above and $C \in I[\A]$. Note that the list of constants $I[\A]$ only depends on $\A, m, n, C_w, C_\G$.

For each monotonic $\A$, for each $C_j^\A \in I[\A]$ we compute and store $\Gamma_{l(\A)}(x, C_j^\A M_0, C_j^\A \tau_0)$, $\Gamma_{l(\A)-1}(x, C_j^\A C'_BM_0, C_j^\A C'_B\tau_0)$ and $\Gamma_{l(\A)-3}(x, C_j^\A C''_BM_0, C_j^\A C''_B\tau_0)$ for all $x\in E$. Since the number of constants depends only on $m, n$, the total time required to compute this collection of refinements is at most $C(\tau_0)N\log N$ and the total space required to store them is at most $C(\tau_0)N$.

\begin{remark}
	Note: the constant $C''_B$ is related to the big $A$ constant and so it will be large.
\end{remark}

\section{Computing CZ decompositions}

\subsection{CZ decomposition}
\label{sec:comp-cz-decomp}

As part of the one-time work, we will need to compute a CZ decomposition for different $\mathcal{A}, CM_0, C\tau_0, x_0, Q_0, P_0$ in the same way as seen in \cite{feffermanFittingSmoothFunction2009a}. Recall that $M_0$ and $\tau_0$ are fixed.

This computation is done for each given node $T = (A_T, x_T, P_T, Q_T, E_T, C_T)$ as well as its corresponding $\tg_{l(\A)}(x, C_T M_0, C_T \tau_0)_{x\in E}$, $\tg_{l(\A)-1}(x, C_B' C_T M_0, C'_B C_T \tau_0)_{x\in E}$ and $\tg_{l(\A)-3}(x, C''_B C_T M_0, C''_B C_T \tau_0)$. 

As we proceed with the one-time work, we calculate the lengthscales $\delta(x, \hat{A})$ using the algorithm ``Finding critical $\delta$, general case'' (Algorithm \ref{alg:find-crit-delta}), with data $x, \hat{\A}, C''_B C_T, M_0, \tau_0$, with $\Gamma_{in}, \Gamma$ as in \eqref{cz3} and with $\hat{\A} < \A$. We calculate these for every $x \in E \cap \frac{65}{64}Q_T = E_T$ and every $\hat{\A} < \A$, and we pass them, along with $\epsilon^{-1}$, to the algorithms defined in \cite{feffermanFittingSmoothFunction2009a} to generate a CZ decomposition of $Q_T$ (see Section 24 in \cite{feffermanFittingSmoothFunction2009a}).

The CZ decomposition corresponding to this particular tuple takes at most $C(\tau_0)N_T\log N_T$ time, and at most $C(\tau_0)N_T$ storage. After performing this work, it answers the following queries in at most $C(\tau_0)\log N_T$ time (see Sections 25, 26, 27 in \cite{feffermanFittingSmoothFunction2009a}):

\begin{itemize}
	\item Given a dyadic cube $Q$ with $5Q \subset 5Q_0$ we can decide whether $Q \in CZ$.
	\item Given $x \in \frac{65}{64}Q_0$ we give a list of all cubes $Q \in CZ$ such that $x \in \frac{65}{64}Q$.
	\item Given a dyadic cube $Q$ we decide whether $E_0 \cap 5Q$ is empty, and if it is not empty we return a representative $y_Q$. If $E_0 \cap \frac{65}{64}Q$ is not empty, then $y_Q \in E_0 \cap \frac{65}{54}Q$. This function is called \texttt{FindRepresentative} in the algorithms.

\end{itemize}
\begin{remark}
	The function in \cite{feffermanFittingSmoothFunction2009a} decides whether $Q^+ \cap E_0$ is empty and if it is not, returns a representative in $Q^{++}\cap E_0$. We can use the same process to find a representative in a general dyadic cube $Q \cap E_0$ in the same time. 
\end{remark}

\begin{remark}
	We can find whether $\frac{65}{64}Q$ is not empty by checking whether $128^n$ smaller dyadic cubes contain a point, similarly for $5Q$.
\end{remark}

\begin{remark}
	The total work for all the one-time work is at most $C(\tau_0) \#(E\cap \frac{65}{64}Q_0)$. See Lemma \ref{lem:constwork} \ref{sec:comp-main} for a discussion.
\end{remark}

\subsection{Partitions of Unity}
Once we have a CZ decomposition for a node $T$, we can compute a partition of unity adapted to the CZ decomposition in at most $C(\tau_0)\log N_T$ work and storage. See Section 28 in \cite{feffermanFittingSmoothFunction2009a} for more details.

\section{Finding a Neighbor}
In this section we describe the algorithm Find-Neighbor that returns $P^{\#}$ as in Section \ref{sec:locint}, case I, or $P^y$ as in Section \ref{sec:locint}, cases II and III. These algorithms will be called always within a node $T$ (see Section \ref{sec:tree}) so all the data needed for the algorithms will be contained in the node. 

Since we know a basis exists for the appropriate $\delta$, we don't need to apply Algorithm \ref{alg:find-crit-delta}. Finding a basis is a linear programming problem with bounded dimension because at most we will be optimizing over vectors consisting of a large (but controlled) number of degree $m-1$ polynomials. Furthermore the number of constraints is also bounded by $C(\tau_0)$ because we are using approximate polytopes. Finally, in the case of \texttt{FindNeighbor} we perform at most $C(\tau_0) \log (N_T)$ query work to find a representative, and then we solve a bounded number of such linear programs. Therefore the total work for \texttt{TransportPoly} is at most $C(\tau_0)$ and for \texttt{FindNeighbor} it is at most $C(\tau_0) \log( N_T)$.

Although these algorithms are simply applications of Megiddo's Algorithm \cite{megiddoLinearProgrammingLinear1984} to different Linear Programming problems, we will write them down here because the input data and constraints are slightly different from each other. For example, \texttt{TransportPoly} uses $\Gamma_{l(\A)-1}$ for a given $\A$, while \texttt{FindNeighbor} uses $\Gamma_{l(\A)-3}$ and has to solve the problem many times (going over all monotonic $\A' < \A$). Note that the constants depend on the large constant $A$ and small constant $\epsilon$ as well as the other intrinsic constants of the problem. Since $A$ and $\epsilon$ depend on $n,m$ and other constants, we don't go into the exact identification of these constants and leave that as a detail to work out in an implementation for a fixed dimension and smoothness problem.

The algorithms use the list and the hashable map defined in Section \ref{sec:const}. 

\begin{algorithm}
	\label{alg:transport}
	\SetKwFunction{Transport}{TransportPoly}
	\SetKwProg{Fn}{Function}{:}{}
	\Fn{\Transport{($\A_T, x_T, P_T, Q_T, E_T, C_T$), $l_0$, $\tg_{l_0}$, $\delta$, $M_0$, $\tau_0$, $y$}}{
		\tcc{$\tg_{l_0}$ has an $(\A, \delta, C_B)$-basis at $(x_T, C_TM_0, C_T\tau_0, P_T)$ }
		\tcc{$\A$ is monotonic}
		\tcc{$|x_0 - y_0|\leq \epsilon_0 \delta$}
		\tcc{$C_T$ is from the list described in Section \ref{sec:const}}
		\tcc{$\tg_{l_0}$ corresponds to $C_T$ in the list described in Section \ref{sec:const}}
		\tcc{$C'$ (in the last restriction) is fixed and depends on $m,n,A,\epsilon$}
		\KwResult{$\hat{P}^{\#}$}
		\tcc{$\tg_{l_0-1}$ has an $(\A,\delta, C'_B)$-basis at $(y_0, C_TM_0, C_T\tau_0, \hat{P}^{\#})$ formed by $P_\alpha$ (byproduct of the computation that we don't need)}

		Use Megiddo's Algorithm to solve 
			\begin{equation*}
			\mathop{\text{maximize}}\limits_{\substack{\hat{P}^{\#} \in \tg_{l_0-1}(y, C_T C'_BM_0, C_T C'_B\tau_0)\\ P_\alpha \in \P}}  1
			\end{equation*}
			\begin{equation*}
			\begin{array}{ll@{}ll}			\text{subject to}& \partial^\beta P_\alpha(y) &=\delta_{\beta\alpha} & \alpha,\beta\in\A\\
			&|\partial^\beta P_\alpha(y)| &\leq C'_B\delta^{|\alpha|-|\beta|}&\alpha \in \A, \beta\in\M\\
			&\hat{P}^{\#} \pm \frac{C_T M_0\delta^{m-|\alpha|}P_{\alpha}}{ C'_B} &\in \tg_{l_0-1}(y, C_T C'_BM_0, C_T C'_B \tau_0)& \alpha\in\A\\
			&\partial^{\beta}(\hat{P}^{\#} - P_T) &\equiv 0 &\beta\in\A\\
			&|\partial^{\beta}(\hat{P}^{\#} - P_T)(x_T)| &\leq C'C_TM_0\delta^{m-|\beta|} &\beta \in \M
			\end{array}
			\end{equation*}
		\KwRet{$\hat{P}^{\#}$}\;
	}
	\caption{Algorithm for Transport Lemma}
\end{algorithm}

\begin{algorithm}
	\label{alg:neighbor}
	\SetKwFunction{NNeighbor}{FindNeighbor}
	\SetKwFunction{Repr}{FindRepresentative}
	\SetKwFunction{Transport}{TransportPoly}
	\SetKwProg{Fn}{Function}{:}{}
	\Fn{\NNeighbor{($\A_T, x_T, P_T, Q_T, E_T, C_T$), $Q \in CZ(\A_T, P_T, x_T)$,  list of $(\tg)$, $M_0$, $\tau_0$}}{
		\tcc{Conditions from \ref{statement-of-the-main-lemma} apply}
		\tcc{$\#(E_0\cap 5Q) \geq 2$}
		\tcc{The list of $(\tg)$ corresponds to the constants $C_T$ as seen in Section \ref{sec:const}}
		\tcc{$C'$ is fixed}
		\KwResult{$y$, $\A^{\#}$, $P^{\#}$, $P_\alpha$}
		\tcc{$\A^{\#} < \A_T$ is monotonic}
		\tcc{$P_\alpha$ form an $(\A^{\#},\epsilon^{-1}\delta_{Q}, C''_B)$-basis for $\tg_{l(\A)-3}$ at $(y, C_T M_0, C_T \tau_0, P^{\#})$}
		
		$y$ = \Repr($Q$)\;
		$P^y = $ \Transport(($\A_T, x_T, P_T, Q_T, E_T, C_T$), $l(\A)$, $\tg_{l(\A)}$,  $\epsilon^{-1}\delta_{Q_T}$, $M_0$, $\tau_0$, $y$)\;
		\For{$\A' = \M$; $\A' < \A$, $\A'$ monotonic}{
			Use Megiddo's Algorithm to solve 	
			\begin{equation*}
			\mathop{\text{maximize}}\limits_{\substack{P' \in \tg_{l(\A)-3}(y,C''_B C_T M_0, C''_B C_T\tau_0)\\ P_\alpha \in \P}} 1
			\end{equation*}
			\begin{equation*}
			\begin{array}{ll@{}ll}
			\text{subject to}& \partial^\beta P_\alpha(y) &=\delta_{\beta\alpha} & \alpha,\beta\in\A'\\
			&|\partial^\beta P_\alpha(y)| &\leq C''_B(\epsilon^{-1}\delta_{Q})^{|\alpha|-|\beta|}&\alpha \in \A', \beta\in\M\\
			&P' \pm \frac{C_T M_0(\epsilon^{-1}\delta_{Q})^{m-|\alpha|}P_{\alpha}}{C''_B} &\in \tg_{l(\A)-3}(y, C_B'' C_T M_0, C_B'' C_T \tau_0)& \alpha\in\A'\\
			&|\partial^{\beta}(P' - P^y)(y)| &\leq C' C_T M_0(\epsilon^{-1}\delta_{Q})^{m-|\beta|} &\beta \in \M
			\end{array}
			\end{equation*}
			\uIf{Problem has solution}{
				$P^{\#} = P'$\;	
				$\A^{\#} = \A'$\;
				break For\;
			}\uElse{
			continue\;
			}
		}
	\KwRet{$y$, $\A^{\#}$, $P^{\#}$}\;
	}
	\caption{Algorithm to find the Nearest Neighbor}
\end{algorithm}
\clearpage
\section{Computing the interpolant (Main Algorithm)}
\label{sec:comp-main}

We describe the main algorithm that will return the jet of the required function at every point in $E$ for given positive real numbers $M_0$, $0 < \tau_0 \leq \tau_{\max}$, $\epsilon$, $A$ as well as a monotonic $\mathcal{A}_T \in \mathcal{M}$, $Q_T$ a dyadic cube, $x_T \in E \cap \frac{65}{64}Q_T$, $P_T$ following the conditions of Section \ref{statement-of-the-main-lemma}. 

As a reminder, constants written as $C, c, $etc. depend only on $m, n$ and may change from one occurrence to the next, while $C(\tau_0), c(\tau_0), $etc. depend only on $m, n, \tau_0$.

We will create a tree as explained in Section \ref{sec:tree} using the algorithms described so far. In this section we will show that the total one-time work to compute the jet of the interpolant at every point $x \in E$ is at most $C(\tau_0)N\log N$, and the space required is at most $C(\tau_0)N$. The algorithm will be run if the decision algorithm (Algorithm \ref{alg:dec}) returns $1$ and will produce always the jet at each $x\in E$ of a function $F$ satisfying the conclusions of the Main Lemma.

\begin{remarks}
	Note that we are guaranteed, when the function is called recursively, that $P_T$, $x_T$ will follow the assumptions of Section \ref{statement-of-the-main-lemma}. Furthermore, for the starting point of the induction ($\emptyset$), we just need to find if the set $\tg_{\emptyset}$ is empty and if it is not, select one polynomial in $\tg_{\emptyset}$.
\end{remarks}

For the data we assume $P_T \in \tg_{l(\A_T)}(x_T, C_T M_0, C_T \tau_0)$ (with $C_T$ belonging to our list of constants associated to $\A_T$ as explained in Section \ref{sec:const}), $x_T \in E_T$ and that the hypotheses of Section \ref{statement-of-the-main-lemma} hold. $[x]_{Q_T}$ is a list of all the points $x \in E_T$.

\begin{algorithm}
	\label{alg:onetime-fun}
	\SetKwFunction{NNeighbor}{FindNeighbor}
	\SetKwFunction{Transport}{TransportPoly}
	\SetKwFunction{onetime}{FindChildren}
	\SetKwProg{Fn}{Function}{:}{}
	\Fn{\onetime{($\A_T, x_T, P_T, Q_T, E_T, C_T$), $\epsilon$,  $(\tg_{l(\A)})_{\A,x}$, $M_0$, $\tau_0$}}{
		\tcc{These $C_T$ and $\tg$ correspond to the list of constants that we computed in Section \ref{sec:const}}
		\uIf{$\A_T== \M$ or $E_T == \emptyset$}{
			\KwRet{$\emptyset$}
		}
		$CZ$ = CZdec($E_T$, $x_T$, $P_T$, $Q_T$, $\A_T$, $C_T M_0$, $C_T \tau_0$)\;
		\uIf{$CZ == \{Q_T\}$}{
			\KwRet{$\emptyset$}
		}\uElse{
			$\{(Q_1, [x]_{Q_1}), \dots, (Q_{k_{\max}}, [x]_{Q_{k_{\max}}})\}$ list of all cubes $Q \in CZ$ such that $\frac{65}{64}Q_\nu \cap E_T \neq \emptyset$ and the points $x \in \frac{65}{64}Q_\nu\cap E_T$ corresponding.\;
			ret = $[]$ empty list\;
			\For{$k=1,\dots,k_{\max}$}{
				\Switch{$\# [x]_{Q_k}$}{
					\uCase{$\geq 2$}{
						$y_k$, $\A_{k}$, $P_{k}$, \_ := \NNeighbor(($\A_T, x_T, P_T, Q_T, E_T, C_T$), $Q_k$, $(\tg_{l(\A)})_{\A,x}$, $M_0$, $\tau_0$)\;
						$C_k = C_T \hat{C}(\A_T, \A_k)$\;
						ret = ret$\cup (\A_k, y_k P_k, Q_k, [x]_{Q_k}, C_k)$\;
					}\uCase{$==1$}{
						$y_k$ := only point in $E\cap \frac{65}{64}Q_k$\;
						$P_k$ := \Transport(($\A_T, x_T, P_T, Q_T, E_T, C_T$), $l(\A_T)$, $\tg_{l(\A_t)}$, $\epsilon^{-1}\delta_{Q_k}$, $M_0$, $\tau_0$, $y_k$)\;
						ret = ret$\cup (\A_T, y_k, P_k, Q_k, [y_k], C_T)$\;
					}
				}
			}
		}
		\KwRet{ret}	
	}
	\caption{Find Children of Node}
\end{algorithm}

Let $N_Q = \#(E\cap\frac{65}{64}Q)$ and $N_T = \#(E\cap\frac{65}{64}Q_T)$. Algorithm \ref{alg:onetime-fun} computes the children of a given node. It runs in at most $C(\tau_0)N_T\log N_T$ time and uses at most $C(\tau_0)N_T$ space. Indeed, computing the $CZ$ decomposition runs in at most $C(\tau_0)N_T\log N_T$ time and uses at most $C(\tau_0)N_T$ space. Finding all cubes $Q \in CZ$ such that $x \in \frac{65}{64}Q$ takes at most $ C(\tau_0)\log {N_{T}}$ time, and we call this query for each $x\in \frac{65}{64}Q_T$ to find the list of cubes $Q_1,\dots,Q_{k_{\max}}$ and the list of points $[x]_{Q_\nu}$ corresponding to each cube. Checking the length of a list is at most $C$ work. Finally, for some of the cubes we call \texttt{FindNeighbor} (at most $C(\tau_0)\log N_T$ work), and for each of them we return a single tuple. In the end there is a list of $k_{\max} \leq C(\tau_0)N_{T}$ tuples, each of them pointing to a cube with either $1$ element or $N_{Q_k}$. 
\begin{lemma}
	\label{lem:constwork}
	$\sum_{k=1}^{k_{\max}}N_{Q_k} \leq C(\tau_0)N_{T}$
\end{lemma}
\begin{proof}
	Each $x \in [x]_{Q_0}$ will appear in at most $C(\tau_0)$ of the new lists $[x]_{Q_k}$ (the reason is a Corollary of Lemma \ref{lemma-cz2} that can be seen in \cite{feffermanFittingSmoothFunction2009a}). 
\end{proof}

\begin{algorithm}
	\label{alg:onetime-alg}
	\SetKwFunction{onetime}{FindChildren}
	\KwData{$Q_0$ dyadic, $x_0 \in E \cap \frac{65}{64}Q_0$, $P_0 \in \P$, $M_0$, $\tau_0$, $\epsilon$}
	\KwResult{Tree of ($\A_T, x_T, P_T, Q_T, E_T, C_T$)}
	\tcc{Again, we use the list of constants and precomputed $\tg$ as explained in Section \ref{sec:const}}
	Tree[0] := $(\emptyset, x_0, P_0, Q_0, [x]_{Q_0}, 1)$\;
	\While{Tree[i] is not empty}{
		\For{($\A_T, x_T, P_T, Q_T, E_T, C_T$) in Tree[i]}{
			($\A_T, x_T, P_T, Q_T, E_T, C_T$).next = \onetime(($\A_T, x_T, P_T, Q_T, E_T, C_T$), $\epsilon$,  $(\tg_{l(\A)})_{\A,x}$, $M_0$, $\tau_0$)\;
			\tcc{Tree[i] refers to all nodes that are i levels deep in the tree.}
		}
	}
	\KwRet{Tree}
	\caption{Compute Tree}
\end{algorithm}

In Algorithm \ref{alg:onetime-alg} we call the function \texttt{FindChildren} one time for each node in the tree. Each node $T$ in the tree has at most $C(\tau_0)N_T$ children, but as seen in Lemma \ref{lem:constwork} the sum over all work of all children is still at most $C(\tau_0)N_{T}$. There are at most $C$ levels in the tree, because if a node is not a leaf, then the next node will have $\A' < \A$ and this can go on at most until $\M$. Therefore the total work of Algorithm \ref{alg:onetime-alg} is at most $C(\tau_0)N_{0}\log N_{0}$ and the total space used is at most $C(\tau_0)N_{0}$. 

\begin{algorithm}
	\SetKwFunction{query}{QueryFunction}
	\SetKwFunction{NNeighbor}{FindNeighbor}
	\SetKwFunction{Transport}{TransportPoly}
	\SetKwFunction{Repr}{FindRepresentative}
	\SetKwProg{Fn}{Function}{:}{}
	\Fn{\query{Tree, T=($\A_T, x_T, P_T, Q_T, E_T, C_T$), T'=($\A_{T'}, x_{T'}, P_{T'}, Q_{T'}, E_{T'}, C_{T'}$), $x \in \frac{65}{64}Q_{T'} \cap E$}}{
		\tcc{T' is a child of T, therefore $Q_T' \in CZ(Q_T)$.}
		\tcc{$M_0$ and $\tau_0$ are not needed because they were used to compute $P_T$ and the nodes of the tree}
		\uIf{$\A_T == \M$}{
			\KwRet{$P_T$}
		}
		\uIf{T'.next is empty}{
			\KwRet{$P_T'$}
		}\uElse{
			\For{T'' in T'.next, $x \in [y]_{T''}$}{
				$f_{T''} :=$ \query(Tree, T', T'', $x$)\;
			}
			\KwRet{$\sum_{\nu=1}^{\nu_{\max}}J_x((\theta_{Q_\nu}^{\A_T})^2)\odot_x f_{T''}$}\;			
		}
	}
	\label{alg:query}
	\caption{Main Algorithm: Finding the jet}
\end{algorithm}

Algorithm \ref{alg:query} returns the jet of a function $F$ at a point $x$. We only care about \texttt{QueryFunction} applied to the points $x \in E$. If we are in a leaf, we have finished. To find all nodes such that $x \in [y]_{T''}$ we query the CZ decomposition and use at most $C(\tau_0)\log \#(E \cap \frac{65}{64}Q_{T''})$ work. We make at most $C$ recursive calls. This will be true for all recursion levels and the number of levels is bounded by a constant depending only on $m$. Therefore the total work is at most $C(\tau_0)\log \#(E \cap \frac{65}{64}Q_{T'})$. When we call the query function on the root node of the tree, the total work is at most $C(\tau_0)\log N$. We call this \texttt{QueryFunction} once for each $x \in E$ to obtain the jet of $F$ at each $x$. Therefore the total work is at most $C(\tau_0)N\log N$. We compute the jet of $\theta_{Q_\nu}^{\A_N}$ as in Section 28 of \cite{feffermanFittingSmoothFunction2009a}. 

Once we have obtained the jet of $F$ at every $x$, the smooth selection problem (see Section \ref{sec:smoothsel}) becomes reduced to an interpolation problem that can be solved by the methods proposed in \cite{feffermanFittingSmoothFunction2009a}. That is, we can easily find the jet of a suitable function $F'$ for any $x \in \R^n$ such that $\|F'\|_{C^m(\R^n, \R^D)} \leq CM_0$ and $J_z(F') = J_z(F)$ for each $z \in E$. Furthermore we know that the problem will have a solution with norm bounded by $M_0$ times a constant $C$. This concludes our work in this paper.

\end{document}